\documentclass[10pt]{article}

\usepackage{amsmath}
\usepackage{amsfonts}
\usepackage{amssymb}
\usepackage{framed}
\usepackage{mathtools}
\usepackage{enumerate}
\usepackage{mathrsfs}
\usepackage[hidelinks]{hyperref}
\usepackage{enumitem}
\usepackage{mathrsfs}
\usepackage{scrextend}
\usepackage{amsthm}
\usepackage{bbold}
\usepackage{bbm}
\usepackage{comment}

\usepackage[T1]{fontenc}
\usepackage{lmodern}

\usepackage{thmtools}
\usepackage{thm-restate}
\usepackage{mathabx}
\usepackage[margin=1.5in]{geometry}
\usepackage{rotating}

\usepackage[nameinlink]{cleveref}

\usepackage[utf8]{inputenc}
\usepackage[T1]{fontenc}

\usepackage{xcolor}
\hypersetup{
    colorlinks,
    linkcolor={red!50!black},
    citecolor={blue!50!black},
    urlcolor={blue!80!black}
}

\usepackage{tikz-cd}

\newcommand{\res}{\textrm{res}}

\numberwithin{equation}{section}

\theoremstyle{plain}
\newtheorem{thm}{Theorem}[section]
\newtheorem{lem}[thm]{Lemma}
\newtheorem{prop}[thm]{Proposition}
\newtheorem{cor}[thm]{Corollary}
\newtheorem{rem}[thm]{Remark}

\theoremstyle{definition}
\newtheorem{defn}[thm]{Definition}
\newtheorem{conj}[thm]{Conjecture}
\newtheorem{exmp}[thm]{Example}

\newtheorem{qtn}[thm]{Question}

\theoremstyle{remark}

\usepackage[nottoc,numbib]{tocbibind}

\tikzset{
  symbol/.style={
    draw=none,
    every to/.append style={
      edge node={node [sloped, allow upside down, auto=false]{$#1$}}}
  },
    labl/.style={anchor=south, rotate=90, inner sep=.5mm}
}

\newcommand{\Spec}{\textrm{Spec} \hspace{0.15em} }

\newcommand\restr[2]{{
	\left.\kern-\nulldelimiterspace
	#1
	\vphantom{\big|}
	\right|_{#2}
	}}

\newcommand{\an}{\operatorname{an}}

\newcommand{\ep}{\varepsilon}
\newcommand{\Aut}{\textrm{Aut}}

\newcommand{\Hom}{\textrm{Hom}}

\newcommand{\ch}[1]{\widecheck{{#1}}}

\newcommand{\codim}{\operatorname{codim}}

\newcommand{\HL}{\operatorname{HL}}
\newcommand{\pos}{\operatorname{pos}}
\newcommand{\Gr}{\operatorname{Gr}}
\newcommand{\PU}{\operatorname{PU}}

\newcommand{\Ad}{\operatorname{Ad}}
\newcommand{\Lie}{\operatorname{Lie}}
\newcommand{\typ}{\operatorname{typ}}
\newcommand{\atyp}{\operatorname{atyp}}

\newcommand{\rank}{\textrm{rank}\,}

\newcommand{\Gal}{\operatorname{Gal}}

\newcommand{\MT}{\mathbf{MT}}

\usepackage{tikz-cd}
\usetikzlibrary{cd}
\usepackage{leftidx}
\usetikzlibrary{positioning}
\usepackage{dsfont}
\usepackage{tikz}
\usetikzlibrary{matrix,arrows,decorations.pathmorphing,positioning}

\newcommand{\GL}{\operatorname{GL}}
\newcommand{\SL}{\operatorname{SL}}
\newcommand{\Sp}{\operatorname{Sp}}

\usepackage{amsmath,calligra,mathrsfs}

\usepackage[cal=boondoxo]{mathalfa}

 \newcommand{\Addresses}{{
  \bigskip
  \footnotesize

\textsc{CNRS, IMJ-PRG, Sorbonne Universit\'{e}, 4 place Jussieu, 75005 Paris, France}\par\nopagebreak
  \textit{E-mail address}, G.~Baldi: \texttt{baldi@imj-prg.fr}  \texttt{baldi@ihes.fr}  

  \medskip

\textsc{Rice University. 6100 Main St MS 136, Houston, TX 77005 (United States)}\par\nopagebreak
  \textit{E-mail address}, D.~Urbanik: \texttt{du8@rice.edu}
}}

\DeclareMathOperator{\sheafhom}{\mathcal{H \kern -1pt o \kern -2pt m}}
\DeclareMathOperator{\sheafper}{\mathcal{P \kern -1pt e \kern -2pt r}}
\DeclareMathOperator{\sheafiso}{\mathcal{I \kern -1pt s \kern -2pt o}}
\DeclareMathOperator{\sheafend}{\mathcal{E \kern -1pt n \kern -2pt d}}
\DeclareMathOperator{\sheafaut}{\mathcal{A \kern -1pt u \kern -2pt t}}

\newcommand{\C}{\mathbb{C}}
\tikzset{
  trim node/.default=1cm,
  trim node/.style={
    overlay,
    append after command={
      ([xshift={+#1}]\tikzlastnode.north west)
      ([xshift={+-#1}]\tikzlastnode.south east)}},
  down and trim/.default=1cm,
  down and trim/.style={
    yshift=-(\pgfmatrixcurrentcolumn-1)*1.5\baselineskip,
    trim node={#1}},
  downup and trim/.default=1cm,
  downup and trim/.style={
    yshift=iseven(\pgfmatrixcurrentcolumn) ? -1.5\baselineskip : 0pt,
    trim node={#1}},
  -|/.style={to path={-|(\tikztotarget)\tikztonodes}},
  |-/.style={to path={|-(\tikztotarget)\tikztonodes}},
  -| sl/.style={-|, xslant=-1},
  |- sl/.style={|-, xslant= 1},
  center picture/.style={
    trim left=(current bounding box.center),
    trim right=(current bounding box.center)}}

\newcommand{\Qbar}{\overline{\mathbb{Q}}}

\newcommand{\Q}{\mathbb{Q}}
\newcommand{\N}{\mathbb{N}}
\newcommand{\Z}{\mathbb{Z}}
\newcommand{\R}{\mathbb{R}}
\newcommand{\V}{\mathbb{V}}

\newcommand{\DT}{\mathbb{S}}
\usepackage{microtype}

\title{Effective atypical intersections and applications to orbit closures}\date{\today}
\author{Gregorio Baldi and David Urbanik}

\begin{document}

\maketitle

\begin{abstract}
We propose a new unifying setting for dealing with monodromically atypical intersections that goes beyond the usual Zilber-Pink conjecture. In particular, we give a new proof of finiteness of the maximal atypical orbit closures in each stratum of translation surfaces $\Omega \mathcal{M}_g (\kappa)$, as given by Eskin, Filip, and Wright. We also describe a concrete algorithm, implementable in principle on a computer, which provably computes all maximal orbit closures which are ``atypical''  in a sense described by Filip. The same methods also give a general algorithm for computing atypical special loci associated to systems of differential equations, and in particular give an effective and o-minimal free proof of the geometric Zilber-Pink conjecture for variations of mixed Hodge structures.
\end{abstract}

\tableofcontents

\section{Introduction}
In the last 15 years, moduli spaces of translation surfaces $(X,\omega)$ have attracted a lot of attention, thanks to the work of Eskin, Filip, McMullen, Mirzakhani, Mohammadi, M\"{o}ller, Wright, and many others. A variety of tools have been used to try to grasp them, mostly from dynamics, analysis, and algebraic geometry. In particular, the work of the aforementioned authors gives striking results on so-called \emph{orbit closures} $\mathcal{N} = \overline{\GL_2(\R)^+ \cdot (X,\omega)}$. The aim of this paper is to push the Hodge-theoretic and algebro-geometric understanding of such loci $\mathcal{N}$ one step further. Using tools from differential geometry and differential algebra we will show that one can understand the distribution of orbit closures just by knowing whether they are \emph{typical or atypical intersections}. 

This new perspective will ultimately come from a new understanding of the (geometric) Zilber--Pink philosophy in the following setting. We start with $S$ a smooth complex algebraic variety, $H$ an algebraic group (which is a semidirect product of a unipotent and a semisimple group), and $\pi : P \to S$ an algebraic $H$-torsor. The torsor $P$ will carry a $H$-principal flat connection $\nabla$ with Galois group $\Gal (\nabla )=H$. We will then study the behaviour of \emph{atypical special subvarieties} $W$ of $S$ (e.g. in the sense of Hodge or Teichm\"{u}ller theory) by pulling back to $P$ certain algebraic data from an extended period domain. Our work will crucially use recent breakthrough work of Bl\'{a}zquez-Sanz, Casale, Freitag, and Nagloo \cite{2021arXiv210203384B}. Thanks to an observation of Filip \cite[\S5]{filipnotes} (partially motivated by recent progress in Hodge theory regarding the distribution of the Hodge locus \cite{2021arXiv210708838B}), such a setting recovers the atypical orbit closures. Indeed the main motivation behind our work is \Cref{question0}, recalled below. Our result on the orbit closures rests on the deep fact that orbit closures are algebraic and admit a purely algebro-geometric characterization (see e.g. \cite[Thm. 2]{filipnotes}).

\subsection{Special subvarieties in Teichm\"{u}ller and Hodge theory, after Eskin-Filip-Wright, and Baldi-Klingler-Ullmo}\label{intro:teich}
We start by recalling the main protagonists of the theory of moduli spaces of translation surfaces, mostly following the recent notes \cite{zbMATH06455787, filipnotes}. Let $X$ be an irreducible smooth projective algebraic curve over $\C$, or simply a compact \emph{Riemann surface}. A \emph{translation surface} is a pair $(X,\omega)$, where $X$ is a compact Riemann surface and $\omega$ is a non-zero global section of the cotangent bundle of $X$ (also known as an \emph{abelian differential}). If $g \geq 1$ is the genus of $X$, then $\omega$ has $ 2g-2$ zeros, counted with multiplicities. Let $\mathcal{M}_g$ denote the moduli space of genus $g$ compact Riemann surfaces and $\Omega \mathcal{M}_g \to \mathcal{M}_g$ the bundle whose fibre over $[X]\in \mathcal{M}_g$ is the space of all holomorphic 1-forms on $X$. It is a rank $g$ (algebraic) vector bundle on $\mathcal{M}_g$, naturally equipped with a stratification by smooth (not necessarily irreducible) algebraic subvarieties $\Omega \mathcal{M}_g (\kappa)$, where $\kappa=(\kappa_0,\dots, \kappa_n)$ satisfies $\sum_i \kappa_i = 2g-2$ and denotes the multiplicities of zeros of $\omega$. We will also write $n= \operatorname{length}(\kappa) -1 $. (In fact, to avoid orbifold issues, we will later work with finite covers of such moduli spaces).

A holomorphic 1-form $\omega$ on $X$ can be thought of as giving a collection of `charts' on $X$ to $\C$, such that:
\begin{itemize}
\item the transition maps are translations,
\item the charts are allowed to be ramified at finitely many points corresponding to the zeros of $\omega$,
\item they are locally given by $z \mapsto \int _{z_0}^z\omega$.
\end{itemize}
Such charts induce the so-called \emph{period coordinates} on $\Omega \mathcal{M}_g (\kappa)$. We therefore obtain an action of $\operatorname{GL}_2(\R)$ on $\Omega \mathcal{M}_g (\kappa)$, locally given by the diagonal action on a product of copies of $\C \cong \R^2$. Such an action is a priori outside the realm of algebraic geometry, but it turns out to be closely related to Hodge theory.

We will study the distribution of strict orbit closures $\mathcal{N}:=\overline{\operatorname{GL}_2(\R) \cdot (X, \omega)} \subset \Omega \mathcal{M}_g (\kappa)$, where $\overline{(\cdot )}$ denotes the topological closure, also known in the literature as \emph{affine invariant submanifolds}. (For terminology and notations we try to follow \cite{filipnotes}.) Associated to an orbit closure there are (at least) two natural numbers that measure its complexity, namely the degree and rank; see \S \ref{strsec} for precise definitions.

We start by recalling a theorem of Eskin, Filip, and Wright. Related methods were earlier used by Matheus-Wright \cite{zbMATH06443238} and Lanneau-Nguyen-Wright \cite{zbMATH06821226} to prove finiteness results in the special case of Teichm\"{u}ller curves.

\begin{thm}[{\cite[Thm. 1.5]{zbMATH06890813}}]\label{thmEFWoriginal}
In each (irreducible component of each) stratum $\Omega \mathcal{M}_g (\kappa)$, all but finitely many orbit closures have rank 1 and degree at most 2. In each genus there is a finite union of orbit closures $\mathcal{N}$ of rank 2 and degree 1 such that all but finitely many of the orbit closures of rank 1 and degree 2 are a codimension 2 subvariety of one of these $\mathcal{N}$.
\end{thm}
In fact, it will be better to work with an equivalent formulation of the above. We will see that in the above result, the orbit closures which arise in finite families are \emph{atypical intersections}, and the others are typical ones. Already in the introduction of \cite{zbMATH06890813} a link with the Andr\'{e}--Oort conjecture is mentioned. But it is only in the recent notes of Filip \cite[\S 5]{filipnotes} where this is made more precise. The following, again due to Eskin, Filip, and Wright \cite[Thm. 1.3, 1.5, and 1.7]{zbMATH06890813}, generalizes and gives a reformulation of \Cref{thmEFWoriginal} (see also \cite[Thm. 3]{filipnotes} and \S \ref{sectionrelatypefw} for the definition of relatively (a)typical).

\begin{restatable}{thm}{efwnew}\label{thmEFWnew}
Each orbit closure $\mathcal{M}$ contains only finitely many maximal atypical suborbit closures. If an orbit closure $\mathcal{M}$ admits relatively typical suborbit closures, then those are dense in $\mathcal{M}$. 
\end{restatable}

With the above formulation, the result looks similar to what happens in Hodge theory, as recently shown by Baldi, Klingler, and Ullmo \cite{2021arXiv210708838B} (see \S \ref{sec:zpapp} for all definitions, in a more general context). They consider a smooth quasi-projective irreducible complex variety $S$ supporting a pure polarized integral variation of Hodge structures ($\Z$VHS from now on, and usually denoted as a pair $(S,\V)$) and study the (tensorial) Hodge locus of $(S,\V)$, which can be written a countable union of algebraic subvarieties of $S$ (the so-called \emph{special subvarieties}). In \emph{op. cit.} they propose a decomposition of the Hodge locus into the \emph{typical} and \emph{atypical} part. Informally a special subvariety is atypical if it has less than the expected codimension (i.e. the one suggested by a na\"ive dimension count) and it is typical otherwise. They prove the following (see also the ICM report \cite{2021arXiv211213814K} for more details and related discussions):

\begin{thm}[{\cite[Thm. 3.1 and 3.9]{2021arXiv210708838B}}]\label{thmhodgelocus}
Let $(S, \mathbb{V})$ be a $\Z$VHS with $\mathbb{Q}$-simple algebraic monodromy. We have:
\begin{enumerate}
\item The positive period dimensional atypical Hodge locus is a strict algebraic subvariety of $S$.
\item The positive period dimensional typical Hodge locus is either empty or analytically (therefore Zariski) dense in $S$.
\end{enumerate}
\end{thm}

One expects the same results to be true without the \emph{positive-
period-dimensional} assumption (in fact this would follow from the full power of the Zilber-Pink conjecture). Nevertheless, the geometric part presents several interesting incarnations (for example \Cref{thmEFWnew} is about positive dimensional subvarieties). The statement 1. is a special case of the \emph{geometric part of Zilber-Pink} (or geometric ZP) in the case where $\mathbb{V}$ has $\mathbb{Q}$-simple monodromy.

The following appears in \cite[pg.5]{filipnotes}, after discussing the two theorems described above, as well as the main results of \cite{BFMS} (we will come back to this point in \S \ref{sec:hodgevsdyn}).
\begin{qtn}[S. Filip]\label{question0}
Can these methods and results be put into a common framework?
\end{qtn}
The goal of our paper is indeed to provide such a unifying framework, which we believe can be of independent interest. 

\subsection{The abstract/unified setting}\label{sec:unfi}

\paragraph{Special Manifolds and Intersections:} Certain manifolds $M$ come naturally with a distinguished class of submanifolds $\{ M_i \}_{i \in I}$ exhibiting a \emph{special} behaviour. Examples include abelian varieties with sub-abelian varieties and their torsion translates; Shimura varieties and more general period domains with sub-Shimura varieties and sub-Mumford-Tate domains; and locally homogeneous spaces with totally geodesic sub-spaces. Given a submanifold $S \subset M$, one then obtains from $\{ M_i \}_{i \in I}$ a distinguished class of \emph{special intersections} $S \cap M_{i}$. Here the examples of interest are (mixed) Hodge loci, and the case of particular interest to us: orbit closures within strata of abelian differentials (cf. (\ref{Filipdiagram}) and the corresponding discussion). Zilber-Pink-type results in such settings suggest distribution behaviour of the special intersections can often be understood by simple dimension criteria. In particular, one says that a component $Y$ of $S \cap M_i$ is \emph{atypical} if 
\begin{equation}\label{firstatyp}
\codim_{M}(Y) < \codim_{M}(S)+ \codim_{M}(M_i) ,
\end{equation}
and \emph{typical} otherwise. One then expects that:
\begin{itemize}
\item[(1)] $S$ contains only finitely many maximal atypical intersections.
\item[(2)] The following are equivalent: 
\begin{itemize}
\item[(i)] $S$ contains one typical intersection;
\item[(ii)] the collection of typical intersections is dense in $S$; and
\item[(iii)] there exists an $i \in I$ such that:
\begin{displaymath}
\dim S - \codim_{M}(M_i) \geq 0.
\end{displaymath}
\end{itemize}
\end{itemize}
Our goal is to discuss two concrete instances of such viewpoint. More precisely we will describe a general strategy to tackle (1) in both the setting of orbit closures and Hodge loci of variations of mixed Hodge structures. The setting we work with might look indirect at first, so we explain here the main protagonists, and how they relate to the above picture.

\paragraph{The Differential Viewpoint:} First of all, it will be crucial for us that $S$ is algebraic, since the main tools come from algebraic geometry --- in particular we will no longer speak of manifolds, but only of algebraic varieties $S$ and complex analytic period domains $D$, and their quotients by lattices $\Gamma$. (With reference to the above viewpoint, one can think of $M$ as some quotient $\Gamma \backslash D$ of $D$ into which $S$ embeds.) We consider normal (pointed) complex varieties $S$, supporting a local system $\mathbb{V}$ with algebraic monodromy group $\mathbf{H}_{S}$ --- for the applications, $\mathbb{V}$ will underline a $\Z$-variation of mixed Hodge structures. Inspired by an idea of Bakker and Tsimerman \cite{2022arXiv220805182B}, in \S \ref{sec:RH} we introduce $P= P(s_{0}) $, an algebraic principal $\mathbf{H}_{S}$-subbundle of the algebraic bundle $\sheafhom(\mathcal{V}, \mathcal{O}_{S} \otimes \mathcal{V}_{s_{0}})$ associated to $\sheafhom(\mathbb{V}, \mathbb{V}_{s_{0}}\times S)$ (via the Riemann-Hilbert correspondence), where $\mathcal{V}$ is the algebraic bundle associated to $\mathbb{V}$. The bundle $P$ has an equivariant flat algebraic connection induced from the connection on $\V$, which foliates $P$ by flat leaves $\mathcal{L}$. Through each point $x \in P$, there is a unique leaf $\mathcal{L}_{x}$. (Similar constructions appear elsewhere in the literature.)

The gist of our approach is the following. In mixed Hodge theory, the special intersections in $S$ are particular cases of \emph{weakly special subvarieties}, i.e. the maximal subvarieties of $S$ with a given monodromy group. A recent spectacular result of Bl\'{a}zquez-Sanz, Casale, Freitag, and Nagloo \cite{2021arXiv210203384B}, known as the \emph{Ax-Schanuel theorem} asserts that, if an algebraic subvariety $V \subset P$ has an \emph{atypical intersection} $U \subset V \cap \mathcal{L}$ with some leaf $\mathcal{L}$ then the projection of $U$ to the base lies in a strict weakly special subvariety. Our strategy now proceeds as follows:

\begin{enumerate}
\item[I.] Relate the atypicality inequality \eqref{firstatyp} to an analogous one appearing in the torsor $P \to S$ constructed in \S\ref{sec:RH}. 
\item[II.] Construct an algebraic family $f : \mathcal{Z} \to \mathcal{Y}$ of subvarieties of $P$ such that all subvarieties of $P$ which can give rise to atypical intersections when intersected with $\mathcal{L}$ appear as fibers of $f$.
\item[III.] Use a suitable version of the Ax-Schanuel theorem \emph{in families} \S\ref{axschanfamsec} to show that all atypical intersections with leaves parameterized by $\mathcal{Y}$ project into the fibres of finitely many families of weakly special subvarieties of $S$ (not necessarily atypical themselves); and
\item[IV.] Conclude by either using some rigidity or maximality property.
\end{enumerate}

\begin{rem}
\label{overparamrem}
We often refer to Step II as the ``over-parameterization'', since there are usually extra varieties parameterized by $f$ that need not give rise to atypical intersections inside $P$ when intersected with $\mathcal{L}$. For more details regarding what we mean by a \emph{suitable version} of the Ax-Schanuel Theorem in families in Step III, we refer to the end of \Cref{axschanfamsec}.
\end{rem}

To implement the above strategy, Hodge theory (essentially through the \emph{fixed part theorem}) is needed at two steps. It is used to guarantee, a priori, that the atypical intersections we are interested in can be over-parametrized by algebraic families (of subvarieties of $P$), as well as to say that the weakly special subvarieties one obtains after applying the Ax-Schanuel theorem fibre-by-fibre can be assembled in countably many families (of subvarieties of $S$).

Arguments of this kind are not new, see for example \cite[\S 3.6]{2021arXiv210708838B} for a history of successful applications of Ax-Schanuel-type theorems to Zilber-Pink. Also, in the case of subvarieties of Shimura varieties, a ``differential-geometric'' approach (rather than the o-minimal one of Baldi-Klingler-Ullmo, cf. also \cite{2025arXiv250203071B}) to the geometric Zilber-Pink theorem can already be found in \cite{binyamini2021effective}, and the generalization to arbitrary VMHS does not pose particular difficulties. There is however a key difference between earlier approaches and ours: namely, we use (an opportune) Ax-Schanuel result in $P$ rather than the Hodge-theoretic one associated to some period domain $D$. At first sight, this might look like a small and technical difference but, in reality, it makes the scope of the Zilber-Pink theory much broader. The reason is that the Hodge-theoretic Ax theorem considers only atypical intersections inside $S \times D$, which typically has much smaller dimension than $P$. This means that there can be more atypical intersections in $P$, since there is more freedom to impose algebraic conditions on leaves $\mathcal{L} \subset P$ than on the graph of some period map. This difference is crucial for treating the case of orbit closures: the space $\Omega \mathcal{M}_g(\kappa)$ supports a $\Z$VMHS, but the atypicality of the orbit closure $\mathcal{N}$ can be detected only in an opportune automorphic bundle above the mixed Shimura variety associated to the $\Z$VMHS. Said in more plain terms: the conditions defining orbit closures include algebraic conditions on periods of the one-form $\omega$ which do not come from conditions on Hodge structures, but can be imposed inside $P$. In particular our main results are evidence in favour of a deeper Zilber-Pink philosophy that enriches already the case of varieties mapping to mixed Shimura varieties, but also includes mixed Hodge theory.

\subsection{Applications of the main result to Teichm\"{u}ller and Hodge theory}\label{sectionspec}
\subsubsection{What do we mean by effective?}\label{sec:eff}

When we use the term ``effective'' to describe a proof or a result, we will mean in the sense that the proof describes an explicit algorithm, implementable in principle on a computer, to compute the object whose existence is described by the proof; in particular, if we say that the proof of the finiteness of a set of orbit closures is effective, we mean that the proof gives an algorithm to compute this set. Note that our use of the term ``algorithm'' is always understood in the sense of ``algorithm that is provably guaranteed to terminate with the correct output'' --- we generally do not speak of ``non-terminating algorithms,'' nor does the word ``algorithm'' refer to a kind of heuristic procedure or method.

To make such statements truly precise in an algebro-geometric context, it is important to specify computational models. More precisely, one should specify precisely in what terms the data involved in the proof --- whether it be an algebraic variety, a connection, or an algebraic map --- is being presented, both when viewed as ``input data'' and when viewed as ``output data''. To do this carefully, we will largely follow the conventions of \cite[\S2]{urbanik2021sets}. We give additional details in \S\ref{finalsection:eff}.

\subsubsection{Teichm\"{u}ller--new proofs and new results}\label{sectionteich}
The first result is a new and effective proof of the first part of \Cref{thmEFWoriginal} (which is a special case of \Cref{thmEFWnew}), which appears in \S \ref{new1}. 
\begin{thm}[Effective finiteness for atypical intersections]\label{efffinteich}
Each stratum $\Omega \mathcal{M}_g (\kappa)$ contains only finitely many maximal atypical suborbit closures. There is moreover an algorithm that computes them.
\end{thm}
In particular the above gives a positive answer to the first part of the following, see also related discussions from \cite{zbMATH07666127}.

\begin{qtn}[{\cite[Qtn. 5.1.5. (see also \S 6.2 and 6.3)]{filipnotes}}]\label{qtn:filip}
Can one make the bound on the number of [maximal] atypical orbit closures effective. More generally, can one effectively bound their numerical invariants (rank, torsion corank) and arithmetic invariants (discriminant of order of number field, index of lattice)?
\end{qtn}
More precisely, we will actually obtain a new approach to \Cref{thmEFWnew} in its entirety. Namely we will give, a priori different, definitions of \emph{(a)typical suborbit closures} $\mathcal{N}$ (relatively to some $\mathcal{M}\subset \Omega \mathcal{M}_g(\kappa)$) and prove both the finiteness and density results. At least in the case where $\mathcal{M}$ agrees with some stratum our definitions will agree with the ones given by Eskin, Filip, and Wright. See indeed \Cref{propcomparison} for all details.

Another related finiteness is the one appearing in the work of Wright, recalled below.
\begin{thm}[{\cite[Thm. 1.2]{zbMATH07225031}}]
There are only finitely many totally geodesic submanifolds (with respect to the Kobayashi metric) of $\mathcal{M}_{g,n}$ of dimension greater than $1$.
\end{thm}
We recall how the above theorem is proven, and the relationship between totally geodesic submanifolds and orbit closures, following the end of \cite[\S 2]{zbMATH07225031}. By Theorem 1.3 in \emph{op. cit.}, if $N$ is a totally geodesic submanifold of dimension at least 2, then $\Omega N$ has rank at least 2. Here we denote by $\Omega N$ the locus of square roots of quadratic differentials in the largest dimensional stratum of the cotangent bundle to $N$.
Since $\Omega N$ determines $N$, and there is a finite list of strata that may contain $\Omega N$ for $N$ a totally geodesic submanifold of $\mathcal{M}_{g,n}$, the result in turn follows from \Cref{thmEFWoriginal}. In particular \Cref{efffinteich} gives also an effective version of the above finiteness (since there is an algorithm to compute the finitely many orbit closures of rank at least 2).

\subsubsection{Hodge--new proofs and new results}
Let $S$ be a smooth quasi-projective variety, and $\V\to S$ an admissible graded-polarizable integral VMHS. 

\begin{defn}\label{deffam}
By a \emph{family} $f : Z \to Y$ of subvarieties of a variety $V$ we mean a pair $(f, \pi)$, where $f$ is a surjective map of varieties\footnote{Here by \emph{variety} we mean a reduced separated scheme of finite type over the field of complex numbers, not necessarily irreducible.}, and $\pi : Z \to V$ is an algebraic map restricting to a closed embedding on each fibre of $f$. 
\end{defn}
Our next main theorem is the following; previously it was known only for pure VHS \cite{2021arXiv210708838B} and in the case of mixed Shimura varieties of ``Kuga type'' by the work of Gao \cite{zbMATH06801925} and \cite[Thm. 8.2]{zbMATH07305885}. (In \Cref{lemmapre} we will show that this simpler reformulation of ``geometric Zilber-Pink'' implies the usual one.)
\begin{thm}[Geometric mixed ZP]\label{thm:mixedZP}
 There are only finitely many families of maximal monodromically atypical subvarieties of $(S,\V)$.
\end{thm}

\begin{rem}
One may interpret the statement of \autoref{thm:mixedZP} as saying that all such weakly special subvarieties fit into a common family over a (possibly disconnected) base.
\end{rem}

\noindent Our proof of \autoref{thm:mixedZP} is effective in the sense that it gives an algorithm to compute the monodromically atypical families given the data of the algebraic vector bundle, algebraic connection, and algebraic Hodge filtration underlying $(S, \mathbb{V})$ as input. One can also obtain from the algorithm a priori bounds on the degree of the polynomials defining the fibres of the finitely many monodromically atypical families. (In fact, this is also true in many cases in the orbit closure setting as well, though we will not dwell on it; the key point is explained in \S\ref{aprioridegbounds}.)

 We remark here that the approach of \cite{2021arXiv210708838B} to \Cref{thmhodgelocus} is not, a priori, effective: in general the use of o-minimality is highly non-constructive and makes describing an effective procedure to determine the finitely many families described in \Cref{thm:mixedZP} difficult. Instead our approach is more similar to previous work on effective atypical intersections \cite{Daw19, urbanik2021sets}, and in particular \cite{binyamini2021effective}. Our approach can be used to give an effective version of the following corollary of \Cref{thm:mixedZP} (which already follows from \cite{2021arXiv210708838B}, only the effective part is new).

\begin{cor}\label{corpure}
The images of the finitely many positive period dimensional families of maximal atypical subvarieties in $S$ are not Zariski dense when $\V$ has $\mathbb{Q}$-simple algebraic monodromy group.
\end{cor}

Finally we hope that an effective geometric Zilber-Pink could have interesting number theoretic applications, once plugged in the main result of \cite{BKU2}, as well as some applications to the geometry of the \emph{Noether-Lefschetz locus} (cf. \cite{2023arXiv231211246B} and for example Qtn. 8 in \emph{op. cit.}).

\subsubsection{Effective characterization of complex totally geodesic subvarieties in non-arithmetic ball quotients}\label{secballquotientsintro}
Finally, we record another interesting application of our effective proof of the geometric Zilber-Pink conjecture (this time for pure $\Z$VHS). Recall indeed the following question of Bader, Fisher, Miller and Stover (that could be compared with \Cref{qtn:filip}, see indeed \S \ref{sec:hodgevsdyn} for a deeper comparison between the two worlds). See \S \ref{sec:balleffective} for more details.

\begin{qtn}[{\cite[Qtn. 9.5]{BFMS2}}]\label{qtn:BFMS}
Can one classify the totally geodesic curves on nonarithmetic Deligne--Mostow \cite{MR586876, zbMATH03996010} orbifolds?
\end{qtn}
Once our new approach to the pure geometric Zilber-Pink is plugged in the main result of \cite{BU} we at least get an algorithm that solves their question. The following is an effective version of the finiteness theorem from \cite{BU, BFMS2}. It is an application of the construction of a certain $\Z$VHS $\widehat{\mathbb{V}}$ on the nonarithmetic Deligne--Mostow orbifolds $S$ from \cite{BU} and \Cref{corpure} (which ultimately follows from \Cref{thm:mixedZP}) applied to $(S,\widehat{\mathbb{V}})$.

\begin{cor}\label{cordm}
Let $\Gamma \subset \operatorname{PU}(1,n)$ be a Deligne-Mostow non-arithmetic lattice. Then there exists an algorithm which computes the finitely many maximal complex totally geodesic subvarieties of $S_\Gamma := \Gamma \backslash \mathbb{B}_{\mathbb{C}}^n$.
\end{cor}

\noindent We will explain in \S\ref{sec:balleffective} how we model the varieties $S_\Gamma$; in particular, one can always find a representative of the commensurability class of $\Gamma$ such that $S_\Gamma$ and the natural $\mathbb{Z}$VHS over it comes from geometry.

It seems very hard to make the strategy from \cite{BFMS, BFMS2} effective, since it relies on dynamical methods (as it was for \Cref{thmEFWoriginal}). More generally, the proof of the above corollary works for any non-arithmetic lattice given as the monodromy group of an ``explicit'' family of varieties.

\subsection{Final comments: dynamics vs functional transcendence}\label{sec:hodgevsdyn}
To put \Cref{question0} into context, we quickly survey some other results partially motivated by \Cref{thmEFWoriginal}. On the dynamical side, this result inspired \cite{BFMS} (and the later \cite{BFMS2}), where it was shown that the arithmeticity of a finite-volume hyperbolic manifold is implied by the abundance of totally geodesic submanifolds. Both results inspired the main result of \cite{BU}, which already obtained finiteness results thanks to some functional transcendence. From here, the functional transcendence/Zilber-Pink viewpoint was repeatedly advanced further. The last mentioned result, and the insightful notes \cite{klin}, were the basic foundation for \cite{2021arXiv210708838B} (later generalised even further, in a different direction, in \cite{2022arXiv220611389U}), which finally hinted to the right level of generality of the theory that we develop here.

By drawing a diagram where the arrow ``$\rightsquigarrow$'' is understood as ``inspired''
\begin{center}
\cite[Thm. 1.5]{zbMATH06890813} $\rightsquigarrow$ \cite[Thm. 1.1]{BFMS} and \cite[Thm. 1.1]{BFMS2} $\rightsquigarrow$ \cite[Thm.  1.2.1]{BU} $\rightsquigarrow$ \cite[Thm. 3.1]{2021arXiv210708838B}
\end{center}
we see that the approach of our paper \emph{closes the circle}, by offering a common setting where one can deal with all of the above. The only notable exception is \cite[Thm. 1.1]{BFMS}, since it is outside complex algebraic geometry. See also the ICM report \cite{2021arXiv211114922F} for an overview of such results as well as \cite{baldi2024rich} for more about the second arrow.

We remark here that in dynamics one often studies the frame flow and the frame bundle (e.g. \cite[page 1053]{zbMATH07513355}, to name one). Likewise for our approach it is essential to work with a frame bundle, as discussed in \S \ref{sec:RH}. It would be exciting to explore this similarity, and in particular, to find, loosely speaking, atypical intersection/functional transcendence phenomena in the theory of finite-volume hyperbolic manifolds. To conclude, we also mention that the very first results towards the geometric Andr\'{e}-Oort conjecture (a very special case of \Cref{thmhodgelocus}) were built on dynamical tools (Ratner theory to be precise), see indeed \cite{MR2180407}.

\subsection{Related works}\label{relatedwork}
Our paper is not the first one to use ideas from functional transcendence to study orbit closures, although the previous ones had different motivations. To conclude the introduction, we mention here the history, to the best of our knowledge, of results on translation surfaces/abelian differentials linked to the Zilber-Pink viewpoint. A first link between Zilber-Pink in $\mathbb{G}_m^n \times \mathbb{G}_a^n$ and orbit closures was observed by Bainbridge, Habegger, and M\"{o}ller \cite{zbMATH06681141}. Klingler and Lerer \cite{2022arXiv220206031K} started analysing strata of abelian differentials, through the lens of \emph{bi-algebraic geometry} (however the $\mathrm{GL}_2(\R)$-action does not enter in the picture). In \cite{2022arXiv220805182B}, Bakker and Tsimerman proved \cite[Conj. 2.15]{2022arXiv220206031K}, giving an \emph{Ax–Lindemann} type of statement for abelian differentials. See also the related \cite{2023arXiv230300642D}, where the authors construct a non-linear bi-algebraic curve, resp. surface, of translation surfaces of genus 7, resp. 10. Even more recently, Benirschke \cite{2023arXiv231007523B} investigated some special cases of Ax-Schanuel results for translation surfaces. In a different but related direction, we also mention the work of Pila and Tsimerman, regarding an Ax-Schanuel theorem for complex curves and differentials \cite[Thm. 3.2]{2022arXiv220204023P}. We hope the framework we consider will have more applications to this emerging viewpoint.

On the technical side, as we discussed above, the novelty of our approach is to work in the period torsor $P$. Its construction and the link between monodromy and differential Galois group is certainly not new and was noticed by several people in the nineties (but certainly known even before, since it is a generalization of Picard-Vessiot theory), and inspired by the work of Deligne \cite{zbMATH03385791}. However we weren't able to find a reference that accommodates our construction. For example O. Gabber pointed out to us \cite[\S 5]{zbMATH05718025} (in particular point b) of \S 5.4) and various papers and preprints of B. Malgrange (like \cite{zbMATH01787094}, where a more complicated setting is studied). One of the oldest references we are aware of is due to Schlesinger \cite{zbMATH02625646}. The role of the Gauss-Manin connection (and related foliated bundles) for the study of some Hodge loci is also highlighted in the work of Movasati, see for example \cite{zbMATH07423960}. With the aim of obtaining effective versions of the geometric Zilber-Pink for Shimura varieties (first obtained in \cite{MR3867286}), Binyamini and Daw \cite{binyamini2021effective} also noticed the advantage of the differential algebraic geometry viewpoint for atypical intersection problems (regarding the arithmetic side of the story, we refer also to Binyamini's work \cite{zbMATH07497568}). The framework considered by Binyamini and Binyamini-Daw, together with the very general Ax-Schanuel theorem of Bl\'{a}zquez-Sanz, Casale, Freitag, and Nagloo discussed above, and its Hodge theoretic incarnation first explored by Bakker and Tsimerman \cite{2022arXiv220805182B} were the main source of inspiration for our work. We remark here that, very recently, also the approach using o-minimality turned out to be effective \cite{2024arXiv240516963B}. Finally we mention also the work of Pila and Scanlon \cite{2021arXiv210505845P}, who used some interesting constructibility and model theoretic arguments related to algebraic differential equations as well as \cite{KO} (especially Thm. 6.1 in \emph{op. cit.}).

Since the first draft of this paper appeared online, it has found intersecting applications \cite{2024arXiv240619366K, 2024arXiv240701304G} to the study of Ceresa cycles in families. Moreover, we recently proved, following the strategy developed in this article, a \emph{uniform version} of the general mixed geometric Zilber-Pink theorem; see indeed \cite[Thm. 4.4]{2026arXiv260402186B}.

\subsection{Outline of paper}
In \S \ref{sectionGbund}, we recall the general formalism of $H$-torsors and principal connections (beyond the Hodge theoretic setting) and recall some basic facts on the Riemann-Hilbert correspondence and weakly special subvarieties. The main new result of \S \ref{atypicalproofofB} is the \emph{Ax-Schanuel theorem in families}, that will be behind the next applications. From \S \ref{sec:zpapp} on we crucially use Hodge theory. We first recall some of the key facts of the theory of $\Z$VMHS, and then, in \S \ref{new1}, apply such ideas to explore atypical intersections in an enhanced Hodge-theoretic setting --- incorporating both the data of a VMHS and the periods of the natural holomorphic one form $\omega$ --- to establish Eskin-Filip-Wright finiteness (\Cref{thmEFWnew}). In \S \ref{new2} we present a similar argument to obtain the geometric mixed Zilber-Pink, confirming that our (foliated) approach answers \Cref{question0}. In the rest of the paper, we explain how our results are effective and discuss algorithmic aspects. In particular, in \S \ref{secappliactionseff}, we provide the effective refinements of the finiteness statements obtained in previous sections.

\subsection{Acknowledgements}
The first author thanks the organizers of the workshop `Finiteness results for special subvarieties: Hodge theory, o-minimality, dynamics' at the Lake Como School of Advanced Studies (partially supported by the Grant ERC-2020-ADG n.101020009 - TameHodge) where various discussions on Zilber-Pink and orbit closures happened (in particular with S. Filip, D. Fisher, Z. Gao, B. Klingler, L. Lerer, N. Miller, M. Stover, and E. Ullmo). In particular we thank S. Filip, whose notes \cite{filipnotes} helped us clarify the theory of orbit closures and the link with atypical intersections. We thank also O. Gabber for discussions on foliated $H$-bundles. We are also grateful to the anonymous referees for their precious comments which improved the exposition.

The authors were partially supported by the NSF grant DMS-1928930, while they were in residence at the MSRI in Berkeley, during the Spring 2023 semester. Key parts of the paper were done also while the authors were hosted by the I.H.E.S. (France) which we thank for the excellent working conditions.

Finally we hope it will be clear to the reader that our applications to Teichm\"{u}ller theory owe a great debt to the breakthrough work of Eskin, Filip, and Wright.

\section{$\nabla$-weakly special subvarieties}\label{sectionGbund}
Before starting our discussion on weakly special subvarieties and the Riemann-Hilbert correspondence, we recall a well-known theorem of Chevalley (see e.g. Prop. 3.1 in Deligne's article in the collection \cite{zbMATH03728195}).

Let $G$ be an affine algebraic group over a field $K$, acting linearly faithfully on a $K$-vector space $V$. This gives an inclusion $G \subset \operatorname{GL}(V)$.
\begin{thm}[Chevalley’s Theorem]\label{Chevalley} 
 For any algebraic subgroup $H\subset G$, there exists a representation $S(V)$ constructed from $V$ by standard tensor operations (i.e. taking direct sums, duals, or tensor products, in any order and any finite number of times), and a line $L\subset S(V)$, such that $H$ is the stabilizer of $L$ in $G$.
 
 Moreover if $H$ is reductive, then there exist a representation $S(V)'$ and a vector $v \in S(V)'$ such that $H$ is the stabilizer of $v$ in $G$.
\end{thm}

\subsection{Preliminary definitions}
We start with a pointed complex irreducible smooth variety $(S,s_{0})$ and a complex local system $\mathbb{V}$ on $S$, corresponding to a representation $\rho : \pi_1(S,s_{0})\to \GL(\mathbb{V}_{s_{0}})$, where we identify $S$ with its analytification.

\begin{defn}\label{monodromydef}
Given an irreducible (not necessarily smooth) subvariety $(Y,s_{0}) \subset  (S,s_{0})$, we define its \emph{algebraic monodromy group} $\mathbf{H}_{Y}=\mathbf{H}_{Y, \mathbb{V}}=\mathbf{H}_{Y, \mathbb{V},s_{0}}$ as the identity component of the Zariski closure of the induced monodromy representation $\pi_1(Y^{\textrm{nor}},s_{0}) \to \GL(\mathbb{V}_{s_{0}})$, where $Y^{\textrm{nor}} \to Y$ denotes the normalization. When $\mathbb{V}$ and the point $s_{0}$ are understood, we will just write $\mathbf{H}_{Y}$.
\end{defn}

Note that algebraic monodromy groups are unchanged by passing to an open subvariety of $Y$, or replacing $Y$ with a smooth resolution --- this is because if $X$ is an irreducible complex normal variety and $U \subset X$ is an open subvariety, $\pi_{1}(U) \to \pi_{1}(X)$ is surjective, cf. \cite[(0.7) (B)]{zbMATH03760312}.

\begin{defn}\label{def:weaklysp}
A geometrically irreducible variety $Y \subset S$ is said to be \emph{weakly special (for $\V$)} if it is maximal under inclusion among geometrically irreducible subvarieties of $S$ for the dimension of $\mathbf{H}_{Y}$. In other words, any $Y'$ strictly containing $Y$ satisfies $\mathbf{H}_{Y} \varsubsetneq \mathbf{H}_{Y'}$.
\end{defn}
Directly from the definition $S$ is a weakly special subvariety, and in cases of interest --- for instance when $\mathbb{V}$ comes from a family of cohomology groups of algebraic varieties inducing a quasi-finite period map --- every closed point $s \in S$ will also be weakly special. We will primarily interest ourselves in other weakly specials, not equal to either $S$ or a closed point. The former we call \emph{strict} weakly special subvarieties.

\subsection{A recap of Riemann-Hilbert}\label{sec:RH}

Given a smooth irreducible algebraic variety $S$ and complex local system $\V\to S$, there is the \emph{Riemann-Hilbert correspondence} functor
\begin{align*} 
\tau : \textrm{LocSys}(S^{\textrm{an}}) &\to \textrm{MIC}_{\textrm{reg}}(S) \\
\mathbb{V} &\mapsto \tau(\mathbb{V}) = (\mathcal{V}, \nabla) .
\end{align*}
Here $\textrm{MIC}_{\textrm{reg}}$ denotes the category of locally free $\mathcal{O}_S$-modules $\mathcal{V}$ on $S$ with regular-singular integrable connection $\nabla$.\footnote{We note that it is part of the theory of Deligne's Riemann-Hilbert correspondence \cite{zbMATH03385791} that the connection $\nabla$ is \emph{algebraic}, i.e., after fixing local algebraic sections for $\mathcal{V}$ and a basis of algebraic differentials on $S$, the connection matrices have as entries algebraic functions on $S$. We will explain how to compute this connection explicitly in cases of interest in \S\ref{computedRdatasec}. } The functor has the property that the space of $\nabla^{\textrm{an}}$-flat sections of the locally free sheaf $\mathcal{V}^{\textrm{an}}$ may be identified with $\mathbb{V}$ up to natural isomorphism. Moreover, the analytic vector bundle $\mathbb{V} \otimes_{\mathbb{C}} \mathcal{O}^{\textrm{an}}$ with its natural flat connection has a unique algebraic structure $(\mathcal{V}, \nabla)$ with regular singularities. The functor $\tau$ is an equivalence of (Tannakian) categories \cite{zbMATH03385791}.

Given a fixed such $(\mathcal{V}, \nabla)$, the tensor powers $\mathcal{V}^{a,b} := \mathcal{V}^{\otimes a} \otimes (\mathcal{V}^{*})^{\otimes b}$ also carry natural algebraic connections, where we use $(-)^{*}$ to denote the dual vector bundle. Correspondingly, write $\mathbb{V}^{a,b} = \mathbb{V}^{\otimes a} \otimes (\mathbb{V}^{*})^{\otimes b}$. Now let $\mathbf{H}_{S}$ be the algebraic monodromy group of $\mathbb{V}$. Using Chevalley's \Cref{Chevalley} on invariants, we may understand the group $\mathbf{H}_{S}$ as the identity component of the stabilizer of a one-dimensional subsystem $\mathbb{L} \subset \bigoplus_{a,b} \mathbb{V}^{a,b}$, and hence it admits an algebraic realization $\mathbf{H}_{S,s}$ in each fibre $\GL(\mathcal{V}_{s})$ as the identity component of the stabilizer of $\tau(\mathbb{L})_{s}$.

Associated to this data we may construct a certain algebraic subbundle $P(s_{0})$ of the principal left $\GL(\mathcal{V}_{s_{0}})$-bundle $\sheafper(\mathcal{V}) := \sheafiso(\mathcal{V}, \mathcal{O}_{S} \otimes_{\mathbb{C}} \mathcal{V}_{s_{0}})$; here the fibres $\sheafper(\mathcal{V})_{s}$ above points $s \in S$ are identified with the set of isomorphisms $\textrm{Iso}(\mathcal{V}_{s}, \mathcal{V}_{s_{0}})$. We write points of $\sheafper(\mathcal{V})$ as pairs $(\beta, s)$, where $\beta : \mathcal{V}_{s} \to \mathcal{V}_{s_{0}}$. To construct $P(s_{0})$ we define
\begin{equation}\label{eqPP}
 P(s_{0}) := \{ (\beta, s) \in \sheafper(\mathcal{V}) : \beta(\tau(\mathbb{L})_{s}) = \tau(\mathbb{L})_{s_{0}} \} 
\end{equation}
where $\mathbb{L} \subset \bigoplus_{a,b} \mathbb{V}^{a,b}$ is as above. This defines an algebraic subbundle of $\sheafper(\mathcal{V})$, and using flat frames we may observe that each $P(s_{0})$ is naturally a left torsor for the Zariski closure $\mathbf{H}^{\textrm{full}}_{S}$ of the monodromy representation at $s_{0}$.

The torsors $P(s_{0})$ contain a canonical analytic subvariety $\mathcal{L} \subset P(s_{0})$, which we also call the \emph{canonical leaf}, constructed as follows. The identity transformation $\textrm{id}_{s_{0}} \in \Hom(\mathbb{V}_{s_{0}}, \mathbb{V}_{s_{0}})$ is a point of the fibre of $P(s_{0})$ over $s_{0}$. If we analytically continue $\textrm{id}_{s_{0}}$ using the flat structure coming from $\mathbb{V}$, we obtain a closed analytic subvariety $\mathcal{L} \subset \sheafper(\mathcal{V})$. It suffices to observe that $\mathcal{L}$ actually lies in $P(s_{0})$. This amounts to saying that if $\beta \in \Hom(\mathbb{V}_{s}, \mathbb{V}_{s_{0}})$ is a flat translate of $\textrm{id}_{s_{0}}$, then $\beta(\tau(\mathbb{L}))_{s} = \tau(\mathbb{L})_{s_{0}}$. But $\mathbb{L}$ is a flat subbundle of $\bigoplus_{a,b} \mathbb{V}^{a,b}$, and so the value of $\beta(\tau(\mathbb{L})_{s})$ is independent of the flat translate $(\beta, s)$ of $(\textrm{id}_{s_{0}}, s_{0})$ chosen.

We will usually work instead with the component $P$ of $P(s_{0})$ which contains the point $(\textrm{id}_{s_{0}}, s_{0})$, where $\textrm{id}_{s_{0}} : \mathcal{V}_{s_{0}} \to \mathcal{V}_{s_{0}}$ is the identity map. Then $P$ is a torsor for $\mathbf{H}_{S}$.

\subsection{Torsor preliminaries}
\label{princprelimsec}

In this section we first introduce the formalism of $H$-torsors (also known as \emph{principal bundles}) and principal connections. The reader may refer for example to \cite[\S 2]{2021arXiv210203384B} for more details, although we will try to be as self contained as possible. Throughout we work over the complex numbers.

Let $S$ be an (irreducible) smooth variety, and $H$ a linear algebraic group. Let $\pi : P \to S$ be an $H$-\emph{torsor}; that is $H$ acts on $P$ (on the right, with the action denoted by $\cdot$) and such that the action induces an isomorphism
\begin{displaymath}
P \times H \simeq P \times_S P, \  (p,g )\mapsto (p,p\cdot g).
\end{displaymath}
In particular the fibres of $\pi$ are principal homogeneous spaces, and the choice of a point $p \in P_s:=\pi^{-1}(s)$ induces an isomorphism $H \cong P_s $, given by $g \mapsto p \cdot g$. We have the following exact sequence of $S$-bundles
\begin{equation}\label{firstdefver}
0 \to \textrm{ker} \, d \pi \to TP \to TS \times_{S} P \to 0 ,
\end{equation}
where $\textrm{ker} \, d \pi$ is the so called \emph{vertical bundle}. A \emph{connection} $\nabla$ is a section of \eqref{firstdefver}. We say that $\nabla$ is $H$-\emph{principal} if it is $H$-equivariant. We additionally say that $\nabla$ is \emph{flat} if it lifts to a morphism of Lie algebras: the equation $\nabla_{[v,w]} = [\nabla_{v}, \nabla_{w}]$ holds on the level of maps of $\mathbb{C}$-schemes. From now on all connections are $H$-principal and flat.

Since $\nabla $ is flat, the space of rational $\nabla$-horizontal vector field is a singular foliation on $P$. Moreover, if $\nabla$ is an algebraic flat connection we obtain a regular foliation $\mathcal{F}$, cf. \cite[\S 2.1 and 2.2]{2021arXiv210203384B}. Let $m$ be the rank of the foliation $\mathcal{F}$.
\begin{defn}
An \emph{integral submanifold} of $\mathcal{F}$ is an $m$-dimensional immersed analytic submanifold $S\subset P$
(not necessarily embedded in $P$) whose tangent space at each point is generated by the values
of vector fields in $\mathcal{F}$ (i.e. the image of $\nabla$). Maximal connected integral submanifolds are called \emph{leaves}. 
\end{defn}
Frobenius theorem ensures that through any regular point passes a unique leaf. Finally, we notice here that $\mathcal{F}$ has also \emph{vertical} leaves included
in the fibers of $P$ at non regular points of $\nabla$, but we will never consider such case. (In particular when speaking of leaves, we always mean \emph{horizontal} leaves, cf. \cite[end of \S 2.2]{2021arXiv210203384B}.)

\begin{exmp}\label{examplebundle}
The component $P$ containing $(\textrm{id}_{s_{0}}, s_{0})$ of the bundle $P(s_{0})$ introduced in \S\ref{sec:RH} is a principal $\mathbf{H}_{S}$-bundle that naturally comes with an algebraic flat connection. We start by constructing an analytic map
\[ \nabla : TS \times_S \sheafper(\mathcal{V}) \to T\sheafper(\mathcal{V}) \]
of bundles over $S$ which sends a pair $(v, (\beta, s))$ to the tangent vector at $(\beta, s)$ which corresponds to the unique lift of $v$ to the leaf through $(\beta, s)$. Using the algebraicity of the connection on $\mathcal{V}$, one may show that this map is algebraic by working in explicit coordinates. For simplicity we may identify the fixed fibre $\mathcal{V}_{s_{0}}$ with $\mathbb{C}^{m}$, with $m$ the rank of $\mathcal{V}$. Working on some Zariski open $U \subset S$, we may choose a frame $\omega_{1}, \hdots, \omega_{m}$ which trivializes $\mathcal{V}$ over $U$, and a basis of differentials $ds_{1}, \hdots, ds_{n}$ for sheaf $\Omega^{1}_{U}$. We obtain a corresponding trivialization $\sheafper(\mathcal{V}) \simeq U \times \GL_{m}$. Letting $[f_{tu}]_{1 \leq t, u \leq m}$ denote the coordinates on $\GL_{m}$, we wish to show that the resulting map
\[ TU \times \GL_{m} \to T U \times T \GL_{m} \]
is an algebraic map. Note that the component of this map mapping to $TU$ is just $(v, (\beta, s)) \mapsto v$, so it suffices to determine the second component. This means expressing the differentials $df_{tu}$ as algebraic functions of $ds_{i}$, the $s_{j}$, and the $f_{tu}$. How to do this is explained in \cite[Claim 3.3]{zbMATH07871772}.

Note that, by construction, the map $\nabla$ restricts to give a map $TS \times_{S} P(s_{0}) \to T P(s_{0})$, which gives the desired algebraic flat connection on the principal $\mathbf{H}_{S}$-bundle $P(s_{0})$. Note that by construction this connection is both flat and $\mathbf{H}_{S}$-equivariant.
\end{exmp}



\begin{rem}\label{remkomin}
We remark here that similar considerations appeared (at least in the VHS setting) in \cite[Thm. 2.5]{2021arXiv211003489C} and \cite[Lem. 2.8]{2022arXiv220805182B}. In these references, compared to our account, the authors use \emph{o-minimality} to produce an algebraic bundle $P$.
\end{rem}

\section{Atypical weakly special subvarieties}\label{atypicalproofofB}

\subsection{Ax-Schanuel after Bl\'{a}zquez-Sanz, Casale, Freitag, and Nagloo}\label{sectionAS}

\begin{defn}[{\cite[Def. 3.4]{2021arXiv210203384B}}]\label{sparsegroup}
An algebraic group $H$ is \emph{sparse} if for any strict Lie subalgebra $\mathfrak{l} \subsetneq \mathfrak{h}= \operatorname{Lie}(H)$, the algebraic envelope $\mathfrak{l}^a$ of $\mathfrak{l}$ (i.e. the smallest algebraic Lie subalgebra containing $\mathfrak{l}$) is again a strict algebraic Lie subalgebra of $\mathfrak{h}$ (i.e. there exists a strict algebraic subgroup $L= L_\mathfrak{l^{a}} \subset H$ such that $\mathfrak{l} \subset \operatorname{Lie}(L)$).
\end{defn}
\noindent We remark that a subgroup of a sparse group does not have to be sparse. The main applications of the paper will concern the case where $H$ is a semidirect product of a unipotent group and a semisimple one; such groups are sparse. For example, for the main results on the orbit closures in $\Omega \mathcal{M}_g(\kappa)$, $H$ will actually be the complex group $\Sp_{2g}\ltimes \mathbb{G}_{\operatorname{a}}^{2g\cdot n}$ (where $n+1$ is the length of $\kappa$).

Fix a principal $H$-torsor $\pi : P \to S$ with flat $H$-equivariant connection $\nabla$, as in \S\ref{princprelimsec}. Assume that $H$ is sparse. See \cite[\S 2.2 and 2.3]{2021arXiv210203384B} for more details on the next definition:

\begin{defn}
\label{mininvdef}
An irreducible subvariety $V \subset P$ is a \emph{minimal $\nabla$-invariant subvariety} if it is the Zariski closure of a leaf. 
\end{defn}

\begin{defn}[{\cite[Lem. 2.3]{2021arXiv210203384B}}]
\label{Galdef}
Let $V$ be a minimal $\nabla$-invariant subvariety. Then the \emph{Galois group of $S$}
\begin{displaymath}
\Gal(\nabla)=\Gal (S)=\{g\in H : V \cdot g =V\},
\end{displaymath}
is an algebraic subgroup of $H$ and $\pi_{|V}$ is a $\Gal(S)$-torsor.
\end{defn}

In our setting, without loss of generality, from now on we assume that $H=\Gal(\nabla)=\Gal(S)$. This will always be true for us as a consequence of \Cref{galequalsmono} below.

\begin{defn}\label{def:nablasp}
A Zariski closed irreducible subvariety $Y \subset S$ is \emph{$\nabla$-special} if its Galois group $\Gal(Y) = \Gal (\nabla_{| Y})$ is strictly contained in $\Gal(S)$ and it is maximal among the subvarieties with fixed Galois group $\Gal(Y)$.
\end{defn}

The following is the main theorem of Bl\'{a}zquez-Sanz, Casale, Freitag, and Nagloo. (For an effective version of the below see also \cite[Thm. 4]{2022arXiv220604304D}.)

\begin{thm}[{\cite[Thm. 3.6]{2021arXiv210203384B}}]\label{thm:newAS}
Let $H$ be a sparse group, and let $\nabla$ be an $H$-principal connection on $\pi : P \to S$ with Galois group $H$. Let $V$ be a subvariety of $P$, $x\in V$, and let $\mathcal{L} \subset P$ be a leaf through $x$. Let $U$ be an analytic irreducible component of $V \cap \mathcal{L}$.

 If \begin{displaymath}
\operatorname{codim}_{P} U < \operatorname{codim}_P V + \operatorname{codim}_P\mathcal{L} ,
\end{displaymath}
(or equivalently $\dim V < \dim U + \dim H$),
then the projection of $U$ in $S$ is contained in a $\nabla$-special subvariety (in the sense of \Cref{def:nablasp}). 
\end{thm}

\begin{rem}
We refer to \cite[Ex. 3.5]{2022arXiv220805182B} for a discussion about how the Ax-Schanuel theorem fails when $H$ does not have sparse monodromy/differential Galois group.
\end{rem}

We will apply this theorem through the following corollary:

\begin{cor}\label{cor:asloc}
Let $S$ be a smooth algebraic variety and let $\V\to S$ be a complex local system. Let $(P, \nabla)$ be the associated torsor of \autoref{examplebundle}. Assume that the monodromy group $H=\mathbf{H}_S$ of $\V$ is sparse. Let $V$ be a subvariety of $P$, $x\in V$, and let $\mathcal{L} \subset P$ be a leaf through $x$. Let $U$ be an analytic irreducible component of $V \cap \mathcal{L}$

 If \begin{displaymath}
\operatorname{codim}_{P} U < \operatorname{codim}_P V + \operatorname{codim}_P\mathcal{L} ,
\end{displaymath}
then the projection of $U$ in $S$ is contained in a weakly special subvariety (in the sense of \Cref{def:weaklysp}).
\end{cor}

To prove \autoref{cor:asloc} we check that $\nabla$-special subvarieties (in the sense of \Cref{def:nablasp}) are indeed weakly special (as in \Cref{def:weaklysp}). See also \cite[Sec. 3]{2021arXiv211003489C}. We note that both the definition of algebraic monodromy group and differential Galois group of a variety $Y \subset S$ are stable under replacing $Y$ with either a smooth resolution or normalization. Moreover the construction of the torsor in \autoref{examplebundle} is compatible with restriction and pullback. It is therefore enough to prove the following.

\begin{lem}
\label{galequalsmono}
Let $(\mathcal{V}, \nabla)$ be a regular singular connection on $S$, and let $P$ be the principal $\mathbf{H}_{S}$-bundle constructed from $(\mathcal{V}, \nabla)$ as in \Cref{examplebundle}. Then $\mathbf{H}_{S} = \operatorname{Gal}(\nabla)$.
\end{lem}

\begin{proof}
Note that if $P$ is an $H$-torsor with $H$-principal flat connection $\nabla$, one always has $\operatorname{Gal}(\nabla) \subset H$, by \autoref{Galdef}. In our situation, we have $H = \mathbf{H}_{S}$ (which crucially uses that $P$ came from a regular singular algebraic connection), hence $\operatorname{Gal}(\nabla) \subset \mathbf{H}_{S}$. It therefore suffices to observe that any $V$ as in \Cref{mininvdef} necessarily contains a monodromy orbit in each of its fibres above points $s \in S(\mathbb{C})$. 

Observe that for each point $p \in V(\mathbb{C})$ mapping to $s \in S(\mathbb{C})$ there is a unique leaf $\mathcal{L}$ passing through $p$, and hence also contains any analytic continuation of $\mathcal{L}$ which passes through the fibre of $P \to S$ above $s$. The points of all such analytic continuations above $s$ are in the image of a monodromy orbit at $s$ coming from the monodromy representation associated to $\mathbb{V}$. Allowing $s$ to vary one sees that $V = P$. 
\end{proof}

\begin{exmp}\label{exsub}
We stress that \Cref{galequalsmono} uses crucially that the algebraic structure on $P$ arises from a connection with regular singularities. If one were to consider a different algebraic structure on $P(\C)$, one only knows that $\mathbf{H}_{S} \subsetneq \Gal(\nabla)$. For instance, the connections given by $\nabla = d - df$ on $\mathbb{A}^{1} \setminus \{ 0 \}$ and with $f$ a function on $\mathbb{A}^{1} \setminus \{ 0 \}$ all have an underlying constant local system, and hence $\mathbf{H}_{S} = \{ \textrm{id} \}$, but can have non-trivial Galois groups. 
\end{exmp}

\subsection{A constructibility fact}\label{sec:constr}

A useful fact which we will use repeatedly is the following:

\begin{prop}
\label{nondenseA}
Let $(P, \nabla)$ be an algebraic $H$-torsor with flat connection on an algebraic variety $S$ and let $f : \mathcal{Z} \to \mathcal{Y}$ be a family of subvarieties of $P$ (in the sense of \Cref{deffam}). Then the locus
\begin{equation}\label{yfe}
\mathcal{Y}(f,e) := \{ (x, y) \in P \times \mathcal{Y} : \dim_{x} (\mathcal{L}_{x} \cap \mathcal{Z}_{y}) \geq e \} 
\end{equation}
is algebraically constructible.
\end{prop}
Above and in the rest of the text, by $\dim_{x}$ we denote the local analytic dimension at the point $x$.

\begin{proof}
We begin with some preliminary reductions.
\begin{itemize}
\item[-] By stratifying the base $\mathcal{Y}$, we may assume that $\mathcal{Y}$ is smooth, irreducible, and use \cite[\href{https://stacks.math.columbia.edu/tag/05F9}{Lemma 05F9}]{stacks-project} to ensure all fibres of $f$ have the same dimension. 
\item[-] Let $\mathcal{Z}_{1}, \hdots, \mathcal{Z}_{\ell}$ be the irreducible components of $\mathcal{Z}$, and write $f_{i}$ for the associated families. Then
\[ \mathcal{Y}(f,e) = \bigcup_{i} \mathcal{Y}(f_{i}, e) \]
so it suffices to assume $\mathcal{Z}$ is irreducible.
\item[-] Consider the torsor $(Q, \delta)$ obtained from $(P,\nabla)$ by pulling back along the natural projection $\pi : S \times \mathcal{Y} \to S$. From the definition of the fibre product one obtains an embedding $\mathcal{Z} \subset Q$. Then $\mathcal{Z} \subset Q$ is foliated by the intersections $\mathcal{L}_{z} \cap f^{-1}(y)$ where $\mathcal{L}_{z} \subset Q$ is an analytic leaf passing through $z \in \mathcal{Z}$. The locus $\mathcal{Y}(f,e)$ is just the image of 
\[ \mathcal{Z}(f,e) := \{ z \in \mathcal{Z} : \dim_{z} \mathcal{L}_{z} \cap f^{-1}(f(z)) \geq e \} , \]
so it suffices to show the constructibility of $\mathcal{Z}(f,e)$. 
\end{itemize}
Now we argue by induction on $\dim \mathcal{Z}$. Note that we have
\[ \mathcal{Z} = \mathcal{Z}(f,0) \supset \mathcal{Z}(f,1) \supset \cdots \supset \mathcal{Z}(f, e) \supset \mathcal{Z}(f,e+1) \supset \cdots \]
so there is some first integer $e$, call it $e_{f}$, such that $\mathcal{Z}(f, e_{f}+1)$ is strictly contained in $\mathcal{Z}$.  Suppose we can produce some closed algebraic locus $\mathcal{Z}'$ containing $\mathcal{Z}(f, e_{f} + 1)$ and strictly contained in $\mathcal{Z}(f, e_{f})$. Then because $\mathcal{Z}$ is irreducible, this implies $\dim \mathcal{Z}' < \dim \mathcal{Z}$. If we then consider the family $f' : \mathcal{Z}' \to \mathcal{Y}$ obtained by restriction, we will have as a consequence of $\mathcal{Z} \supset \mathcal{Z}' \supset \mathcal{Z}(f, e_{f} + 1)$ that
\[ \mathcal{Z}'(f', e) := \{ z \in \mathcal{Z}' : \dim_{z} \mathcal{L}_{z} \cap f'^{-1}(f'(z)) \geq e \} = \mathcal{Z}(f, e) \]
for all $e \geq e_{f} + 1$, so replacing $\mathcal{Z}$ with $\mathcal{Z}'$ and $f$ with $f'$ we will be done by the induction hypothesis on $\dim \mathcal{Z}$.

Now consider the locus 
\[ \mathcal{Z}(f,e_{f} + 1)^{1} := \{ z \in \mathcal{Z} : \dim T_{z} \mathcal{L}_{z} \cap T_{z} f^{-1}(f(z)) \geq e_{f} + 1 \} . \]
Because $\nabla$ is algebraic, so is the bundle $\bigcup_{\mathcal{L} \subset P} T \mathcal{L}$, hence the locus $\mathcal{Z}(f,e_{f} + 1)^{1}$ is defined by algebraically constructible conditions. Let us take $\mathcal{Z}'$ to be the Zariski closure of $\mathcal{Z}(f, e_{f}+1)^{1}$; since $\mathcal{Z}(f, e_{f}+1)^{1}$ is algebraically constructible, this is the same as the topological closure. We observe that $\mathcal{Z}'$ is strictly contained in $\mathcal{Z} = \mathcal{Z}(f, e_{f})$: it suffices to produce an open analytic neighbourhood of $\mathcal{Z}(f, e_{f})$ which does not intersect $\mathcal{Z}(f, e_{f}+1)^{1}$, which we can do since the locus where the intersections $\mathcal{L}_{z} \cap f(f^{-1}(z))$ are smooth of dimension $e_{f}$ is open in $z$. But $\mathcal{Z}' \supset \mathcal{Z}(f, e_{f}+1)$ by construction, so this completes the proof.
\end{proof}


Once we have an algorithmic way of working with torsors and their connections, the above proof gives an algorithm for computing the loci $\mathcal{Z}(f,e)$ and $\mathcal{Y}(f,e)$. This will be the essential input in making our results effective in \S\ref{finalsection:eff}.

\subsubsection{A priori degree bounds} 
\label{aprioridegbounds}

Although it is not our main focus, let us explain how one can use the above proof to obtain an a priori bound on the degree of the closure of $\mathcal{Y}(f,e)$, and therefore also the closure of its image in $S$. Since in applications one will use such $\mathcal{Y}(f,e)$ to construct atypical (weakly) special loci, including atypical orbit closures, we believe this can be of interest in situations where computing the orbit closure itself is computationally difficult but computing the degree estimates is more tractable. 

The basic input is the following elementary fact (see for e.g. \cite{63463}): 

\begin{lem}
Let $f : \mathbb{A}^{m} \to \mathbb{A}^{m}$ be a polynomial map. Suppose $f = (f_{1}, \hdots, f_{m})$, and let $\kappa = \operatorname{max}_{i} \{ \deg f_{i} \}$ (where $\deg$ denotes the total degree). Then for an algebraic subvariety $W \subset \mathbb{A}^{m}$ of dimension $r$, one has $\deg \overline{f(W)}^{\textrm{Zar}} \leq [\deg W] \kappa^{r}$. \qed
\end{lem}

If one then runs the argument of \autoref{nondenseA}, then this fact can be used to bound the degree of $\overline{\mathcal{Y}(f,e)}^{\textrm{Zar}}$ and its image in $S$ after choosing explicit coordinates for $S, P,\mathcal{Z}$, $\mathcal{Y}$, etc, and that we have an explicit presentation for $\nabla$ as a section of (\ref{firstdefver}). Indeed, since $\mathcal{Y}(f,e)$ is the image of $\mathcal{Z}(f,e)$, this reduces to bounding the degree of $\mathcal{Z}(f,e)$. Then $\mathcal{Z}(f,e)$ is constructed as the closure of loci of the form $\mathcal{Z}(f',e')^{1}$ for appropriate $(f',e')$ determined by $(f,e)$. These loci, in turn, are projections of loci in the tangent bundle $TP$ defined by explicit equations, so we may bound their degrees as well. The number of such loci needed in the proof can be bounded by bounding the number of steps needed in the Noetherian induction, which can be done by using the degree bound on each $\mathcal{Z}'$ appearing in the proof to bound the number of its components as well. 

\subsection{Ax-Schanuel theorems for families}
\label{axschanfamsec}

We keep the notation and setting of \Cref{nondenseA} and assume that the algebraic group $H$ is sparse. Given a family $h$ of subvarieties of $S$, we refine the set $\mathcal{Y}(f,e)$ introduced in the previous section \eqref{yfe}, as follows:
\begin{align*}
\mathcal{Y}(f,e,h) := \left\{ (x, y) \in P \times \mathcal{Y} : \begin{array}{c} \textrm{ there exists an embedded analytic germ } \\ (x, \mathcal{D}) \subset (x, \mathcal{Z}_{y} \cap \mathcal{L}_{x}) \textrm{ of dimension } \\ \textrm{ at least } e \textrm{ mapping into a fibre of }h\end{array} \right\}
\end{align*}
The following results complement \autoref{cor:asloc}. They say that if one applies the Ax-Schanuel theorem to a family of subvarieties then the resulting intersections belong to a family of weakly special subvarieties. Many variants of such results are possible; see the discussion at the end of this section. (For a family version of the Ax-Schanuel theorem in a special case see \cite[Prop. 2.1]{2019arXiv190504364P}.)

For the next two results, we fix some family $f : \mathcal{Z} \to \mathcal{Y}$ of subvarieties of $P$ and an integer $e \geq 1$. We suppose that there exists a countable collection of families $\{ h_{i} : C_{i} \to B_{i} \}_{i=1}^{\infty}$ of subvarieties of $S$ in the sense of \autoref{deffam} such that all weakly special subvarieties of $S$ are among the fibres of these families.

\begin{thm}
\label{protogeoZP2}
Consider a constructible subset $\mathcal{K} \subset P \times \mathcal{Y}$. Suppose that each germ $(x, \mathcal{D})$ of dimension $\geq e$ and corresponding to a point $\mathcal{Y}(f,e) \setminus \mathcal{K}$ maps into a fibre of some family in the collection $\{ h_{i} : C_{i} \to B_{i} \}_{i \in I}$, where $I \subset \mathbb{N}$ is some subset. Then in fact this is true for some finite subset of $I$.
\end{thm}


\begin{proof}
We define $f \cap h_{i}$ to be the family over $\mathcal{Y}_{i} := \mathcal{Y} \times B_{i}$ of subvarieties of $P$ whose fibres are the intersections of fibres of $f$ with the subbundles $\restr{P}{h^{-1}_{i}(b_{i})} \subset P$. We have a natural map $\mathcal{Y}_{i} \to \mathcal{Y}$. Then we have that $\textrm{im}[ \mathcal{Y}(f \cap h_{i}, e) \to P \times \mathcal{Y} ] = \mathcal{Y}(f,e,h_{i})$, so $\mathcal{Y}(f,e,h_{i})$ is constructible by \autoref{nondenseA}. By assumption we have that $\mathcal{Y}(f,e) \setminus \mathcal{K}$ is contained in the countable union $\bigcup_{i=1}^{\infty} \mathcal{Y}(f,e,h_{i})$, but since everything is constructible, it is in fact contained in finitely many of these sets by \autoref{lemma432} below.
\begin{lem}
\label{lemma432}
Let $Y $ be a complex variety and $\mathcal{I}\subset Y(\C)$ be a constructible subset, and $\{ U_r \}_{r = 1}^{\infty}$ be constructible subsets of $Y(\mathbb{C})$. If $\mathcal{I} \subset \bigcup_{r = 1}^{\infty} U_r$, then in fact it lies in the union of finitely many of the $U_{r}$.
\end{lem}

\begin{proof}
This follows from \cite[Lem. 4.32]{urbanik2021sets}, where one takes the sequence of intersections in \cite[Lem. 4.32]{urbanik2021sets} to be trivial (all equal to $\mathcal{I}$) and the countable union to be $\mathcal{I} = \bigcup_{r=1}^{\infty} (\mathcal{I} \cap U_{r})$. 
\end{proof}
\end{proof}

\begin{cor}
\label{protogeoZP}
Suppose that each fibre of $f$ has dimension $< \dim H + e$. Then the set $\mathcal{Y}(f,e)$ introduced in \autoref{nondenseA} is contained in $\bigcup_{i=1}^{m} \mathcal{Y}(f,e,h_{i})$ for some index $m$ after reordering the $i$'s.
\end{cor}

\begin{proof}
Recall that $(P, \nabla)$ is an algebraic $H$-torsor with flat connection on $S$ and that $f : \mathcal{Z} \to \mathcal{Y}$ is a family of subvarieties of $P$ such that each fibre of $f$ has dimension $< \dim H + e$. Thanks to the Ax-Schanuel theorem \autoref{cor:asloc}, a point in the set $\mathcal{Y}(f,e)$ has to map into a fibre of some family in the collection $\{ h_{i} : C_{i} \to B_{i} \}_{i \in I}$. (From the beginning of the section, we assumed that $H$ was sparse as needed to apply the aforementioned corollary, as well as the existence of a countable collection of families $h_i$ over-parametrizing the weakly special subvarieties of $S$). We are therefore in the setting of \Cref{protogeoZP2} (with $\mathcal{K}=\emptyset$) and so we can conclude that only finitely many $ \mathcal{Y}(f,e,h_{i})$ are needed.
\end{proof}

We conclude this section explaining some subtleties of the general strategy outlined in \Cref{sec:unfi} to control atypical intersections, and how it relates to the above statements. When applying the above results, one faces (at least) two challenges: 
\begin{enumerate}
\item Construct the ``output'' families $h_i$.
\item Construct the ``input'' algebraic family $f : \mathcal{Z} \to \mathcal{Y}$ of subvarieties of $P$;
\end{enumerate}
\Cref{sec:familieswp} will be devoted to 1. and will use various powerful inputs from Hodge theory. The same families will appear when dealing with orbit closures as well as with the mixed geometric Zilber-Pink conjecture. 

In the orbit closure (resp. the VHS) setting, step 2. will be achieved in \Cref{orbitoverparamlem} (resp. \Cref{prop2par}). Nevertheless, in practice, a new problem appears and the above statements are not ready-to-use even after the two steps are achieved; or at least they require some extra care to be applied. Namely different algebraic subvarieties of $P$ can give rise to the same intersection with a leaf $\mathcal{L}_x$ and the same issue can happen when working with families of intersections. Morally speaking, especially when dealing with the geometric Zilber-Pink conjecture, we would like to consider a ``normalized/optimal'' family $f$ where, whenever an atypical intersection with a leaf appears, such an intersection is Zariski dense in the fiber. Albeit this can be achieved by an induction argument and refining the input family at each step, we will take a more direct approach by allowing ourselves to refine the set $\mathcal{Y}(f,e)$ in various steps in order to take care of such ``un-normalised intersections''. This is why we gave a statement with $\mathcal{Y}(f,e) \setminus \mathcal{K}$ which serves as a model for the more complicated argument around \Cref{L3}. We will give more details in \S\ref{geozpproofsec}.

\section{Introduction to variations of mixed Hodge structures: period maps and Zilber-Pink}\label{sec:zpapp}

In this section we briefly recall the Zilber--Pink conjecture for admissible variations of polarizable integral mixed Hodge structures and explain more carefully \Cref{thmhodgelocus} (and its variants in the mixed setting). There is a first version of the Zilber--Pink conjecture appearing in \cite{klin}, we explain here the \emph{strong version}, following \cite{2021arXiv210708838B} (in \emph{op. cit.} the authors state only the pure version, but, as they notice, it is easy to obtain the general mixed formulation).

\subsection{Notations}
First of all some general notations/conventions. In this section an algebraic variety $S$ is a reduced scheme of finite type over the field of complex numbers, not
    necessarily irreducible. If $S$ is an algebraic (resp. analytic) variety, by a subvariety
  $Y \subset S$ we always mean a \emph{closed} algebraic
  (resp. analytic) subvariety. (Usually we assume $S$ to be smooth and consider subvarieties that are not necessarily smooth themselves). A pure $\Q$-Hodge structure of weight $n$ on a finite dimensional $\Q$-vector space $V_\Q$ is a decreasing filtration $F^\bullet$ on the complexification $V_\C$ such that $V_\C= \oplus_{p\in \Z} F^{p} \cap \overline{F^{n-p}}$. The category of pure $\mathbb{Q}$-Hodge-structures is Tannakian, and the category of polarizable pure $\mathbb{Q}$-Hodge-structures is semisimple (all pure Hodge structures appearing in this paper will be polarizable). Recall that the Mumford--Tate group
$\MT(V) \subset \GL(V)$ of a $\mathbb{Q}$-Hodge structure  
$V$ is the Tannakian group of the Tannakian subcategory $\langle
V\rangle ^\otimes$ of $\mathbb{Q}$-Hodge structures
generated by $V$. Equivalently, it is the smallest $\mathbb{Q}$-algebraic subgroup of
$\GL(V)$ whose base-change to $\mathbb{R}$ contains the image of $h: \mathbb{S} \to
\GL(V_{\mathbb{R}})$ (here $\mathbb{S}$ denotes the \emph{Deligne torus}, i.e. the Weil restriction to $\R$ of the complex multiplicative group). It is also the fixator in $\GL(V)$
of the Hodge tensors for $V$. As $V$ is polarised, this is a reductive group. 
See for example \cite{2021arXiv210708838B} for more Hodge-theoretic details. We will recall below in more detail the case of VMHS, but assuming some familiarity with the pure setting. As a reference for Tannakian categories we refer to the article of Deligne and Milne from \cite{zbMATH03728195} and for general facts in Hodge theory to \cite{MR0498551}. For basic knowledge on variations of \emph{mixed} Hodge structures and \emph{mixed} Mumford–Tate domains, we refer for example to \cite[\S 2-5]{2021arXiv210110938G} and also to \cite{zbMATH00047436}.

\subsection{The Tannakian category of $\Q$VMHS, first definitions}\label{mixedvhs}
A mixed $\Z$-Hodge structure ($\Z$MHS) is a triple $(V, W_\bullet, F^\bullet) $ consisting of a (torsion-free) finitely generated $\Z$-module $V$, a finite ascending filtration $W_\bullet$ of $V_\Q:= V \otimes _{\Z}\Q$ (called the \emph{weight filtration}) and a finite decreasing filtration $F^\bullet$ of $V_\C$ (called the \emph{Hodge filtration}) such that for each $n \in \Z$, 
\begin{displaymath}
(\Gr_n^W V_\Q, \Gr_n^W F^\bullet)
\end{displaymath}
is a pure $\Q$-Hodge structure of weight $n$. A \emph{graded polarization} of a mixed $\Z$-Hodge structure is the datum of a polarization on the pure $\Q$-Hodge structure $\Gr^W V = \bigoplus_{n} \Gr^{W}_{n} V$.

The following is \cite[1.4, 1.5]{zbMATH00047436} and explains under which conditions a homomorphism $h: \DT_\C\to \GL(V_\C)$ gives a MHS:
\begin{lem}\label{lemmamixed}
Let $ V$ be a finite dimensional $\Q$-vector space. A morphism $h: \DT _\C\to \GL(V_\C)$ defines a $\Q$MHS on $V$ if and only if there exists a connected $\Q$-algebraic subgroup $\mathbf{G} \subset\GL(V) $ such that $h$ factors through $\mathbf{G} _\C$ and satisfies the following conditions (where we write $W_{-1}$ for the unipotent radical of $\mathbf{G}$):
\begin{itemize}
\item The composition morphism $\DT_\C \to \mathbf{G}_\C \to (\mathbf{G}/W_{-1})_\C$ is defined over the reals;
\item Precomposing $\DT \to (\mathbf{G}/W_{-1})_\R$ with the weight homomorphism $\mathbb{G}_{m,\R} \to \DT $ we obtain a cocharacter of the center of $(\mathbf{G}/W_{-1})_\R$ defined over $\Q$;
\item The weight filtration on $\Lie \mathbf{G}$ defined by $\Ad \circ h$ satisfies $W_0 (\Lie \mathbf{G}) = \Lie \mathbf{G}$ and $W_{-1}(\Lie \mathbf{G} ) = \Lie W_{-1}$.
\end{itemize}

\end{lem}

\begin{defn}\label{mixedhodgedatumdef}
A \emph{mixed Hodge datum} is a pair $(\mathbf{G},D)$ where $\mathbf{G}$ is a connected linear algebraic $\Q$-group (with unipotent radical $\mathbf{W}_{-1}$) and $D$ is a $\mathbf{G}(\R)\mathbf{W}_{-1}(\C)$-conjugacy class of homomorphisms $\mathbb{S}\to \mathbf{G}(\C)$, satisfying the conditions of \Cref{lemmamixed}.
\end{defn}
For the next definition, see for example \cite[Sec. 2.5]{klin}.
\begin{defn}
Let $S$ be a smooth quasi-projective complex variety and $\mathcal{O}_S$ its sheaf of holomorphic functions. A \emph{variation of mixed $\Z$-Hodge structures} ($\Z$\textit{VMHS}) over $S$ is a triple $(\mathbb{V}, W_\bullet, F^\bullet)$, where:
\begin{itemize}
  \item[(1)] $\mathbb{V}$ is a locally constant $\Z_S$-module on $S$,
  \item[(2)] $W_\bullet$ is a finite increasing filtration (called the weight filtration) of the $\Q$-local system $\mathbb{V}_\Q$ by $\Q$-local sub-systems,
  \item[(3)] $F^\bullet$ is a finite descending filtration (called the Hodge filtration) of the holomorphic vector bundle $\mathcal{V} := \mathbb{V} \otimes_{\Z_S} \mathcal{O}_S$ by holomorphic subbundles,
\end{itemize}
such that:
\begin{itemize}
  \item[(a)] for each $s \in S$, the triple $(\mathbb{V}_s, (W_\bullet)_s, F^\bullet_s)$ is a mixed $\Z$-Hodge structure,
  \item[(b)] the flat connection $\nabla : \mathcal{V} \longrightarrow \mathcal{V} \otimes \Omega^1_S$ whose sheaf of horizontal sections is $\mathbb{V}_\mathbb{C}$ satisfies the Griffiths’ transversality condition
  \begin{equation*}
    \nabla F^\bullet \subset \Omega^1_S \otimes F^{\bullet - 1}.
  \end{equation*}
\end{itemize}

A graded polarization $q$ for $(\mathbb{V}, W_\bullet, F^\bullet)$ is a sequence
\[
q_k : \mathrm{Gr}^W_k(\mathbb{V}_\Q) \times \mathrm{Gr}^W_k(\mathbb{V}_\Q) \longrightarrow \Q(-k)
\]
of $\nabla$-flat bilinear forms inducing graded polarisations $q_{k,s}$ on the mixed $\Z$-Hodge structure $(\mathbb{V}_s, (W_\bullet)_s, F^\bullet_s)$ for all $s \in S$. Here $\Q(-k)$ is the usual constant variation of $\mathbb{Q}$-Hodge structure of weight $2k$.
\end{defn}

An important \emph{admissibility} condition for variations of mixed Hodge structure was introduced in \cite{zbMATH04006406,zbMATH04017061} (see also, for example, \cite[Def. 3.3]{2021arXiv210110938G}). From now on we denote by $\Z$VMHS the category of graded-polarizable and admissible VMHS over $S$. When $\Z$ is replaced by $\Q$, this is a Tannakian category, as discussed for example in \cite[\S 2.4]{klin}, and references therein. (The fact that the category of admissible $\Q$VMHS is abelian is due to Kashiwara, see \cite[\S 4.5 and Prop. 5.2.6]{zbMATH04006406} for the relevant statements.)

To ease the notation, we will say: let $\mathbb{V}=(\mathbb{V}, W_\bullet, F^\bullet)$ be a $\Z$VMHS and denote, for example, by $\mathbb{V}_s$ the $\Z$MHS above the point $s\in S$.

\begin{defn}\label{mtdef}
For every $s\in S$, the \emph{Mumford-Tate} group $\mathbf{G}_s(\mathbb{V})$ of the Hodge structure $\V_s$ is the Tannakian group of $\langle \V_s\rangle^\otimes$ of $\Q$MHS (here the exact faithful $\Q$-linear tensor functor $\operatorname{For}: \Q MHS \to \operatorname{Vec}_\Q$ is the forgetful functor).
\end{defn}
 It is a connected $\Q$-algebraic group, and 
\begin{itemize}
\item reductive if $\V_s$ (or $\V$) is pure; 
\item in general it is an extension of $\mathbf{G}_s (\textrm{Gr}^W \V)$ by a unipotent group (here $W$ denotes the weight filtration on $\V$).
\end{itemize}

We note that $\textrm{Gr}^W$ can be viewed as an exact functor \cite[Thm. 2.3.5(iv)]{MR0498551} from $\mathbb{Q}$VMHS to $\mathbb{Q}$VHS. We have also a notion of \emph{Mumford--Tate group} of the $\Q$VMHS $\V$, defined as the connected component of the identity of the Tannakian group of $\V$. Here we pick the fibre functor associated to some base point $s\in S$, that is the functor
\begin{displaymath}
\omega_s: \Q \operatorname{VMHS} \to \Q \operatorname{MHS} \to \operatorname{Vec}_\Q, \hspace{2em}  \V \mapsto \V_s \mapsto \operatorname{For}(\V_s) .
\end{displaymath}
We denote the aforementioned Mumford--Tate group by $\MT(\V, \omega_s)$. Thanks to \cite[II, Thm. 3.2]{zbMATH03728195}, the resulting group, by the general theory, will not depend on such choice of a base point $s$ (that is there is a \emph{unique natural} isomorphism $\MT(\V, \omega_s)\cong \MT(\V, \omega_{s'})$). See also \cite{MR1154159} for an equivalent point of view. By the general Tannakian formalism we have a canonical inclusion $\mathbf{MT}(\V_s) \subset \mathbf{MT}(\V, \omega_s)$, at every $s\in S$.

We define\footnote{It is easy to see that our definition agrees with the others in the literature, even if it does not often appear in this shape. Indeed usually the fibre functor $\omega_s$ is implicit and one just writes $\mathbf{MT}(\V_s) \varsubsetneq \mathbf{MT}(\V)$, or often the inclusion is understood up to \emph{parallel transport}.} the following subset of $S^{\an}=S(\C)$.
\begin{equation}
\HL(S, \V^\otimes):=\{s\in S^{\an} : \mathbf{MT}(\V_s) \varsubsetneq \mathbf{MT}(\V, \omega_s)  \}.
\end{equation}
It is naturally a countable union of complex analytic subvarieties, and in fact Cattani, Deligne, and Kaplan \cite[Thm. 1.1]{CDK} proved the following
in the pure case, and Brosnan, Pearlstein, and Schnell in the mixed case \cite{zbMATH05757878}. We also refer to \cite[Thm. 1.6]{BKT} (see also \cite{zbMATH07720398}) for an alternative
proof, using o-minimality, and also to \cite{zbMATH07855543}, for the mixed case.

\begin{thm} \label{CDK}
Let $S$ be a smooth connected complex quasi-projective algebraic variety and
$\V$ be a $\Z$VMHS over $S$. 
Then $\HL(S, \V^\otimes)$ is a countable union of
closed irreducible algebraic subvarieties of $S$: the strict maximal special
subvarieties of $S$ for $\V$.
\end{thm}

\noindent In general, the notion of \emph{special subvariety} is defined as follows:

\begin{defn}
Let $(S, \V)$ be a $\Z$VMHS. A \emph{special} subvariety $Y \subset S$ is a closed irreducible complex algebraic subvariety which is maximal for the Mumford-Tate group $\mathbf{G}_{Y}$ associated to $(Y, \restr{\mathbb{V}}{Y})$. 
\end{defn}

For us, from now on, all special subvarieties, so in particular all components of the Hodge locus, will be equipped with their \emph{reduced} algebraic structure.

\subsection{Period maps and atypical intersections}\label{sectonperiodmaps}

From now on we look only at admissible $\Z$VMHS $\V$ on a smooth quasi-projective variety $S$. Then the data of such a pair $(S,\V)$ is the same datum as a \emph{(mixed) period map}. 
 Let $(\mathbf{G},D)$ be the (connected, mixed) Hodge datum associated to $(S, \V)$ as in \Cref{mixedhodgedatumdef}, and $\Gamma \backslash D$ the associated mixed Mumford--Tate quotient (called also Hodge variety by some authors). (Here and in the sequel, we always assume that $\Gamma$ is a torsion free finite index subgroup of $\mathbf{G} (\Z)$). We will be mainly concerned with period maps:
\begin{displaymath}
\Psi: S^{\operatorname{an}}\to \Gamma \backslash D,
\end{displaymath}
that is locally liftable, holomorphic and horizontal maps (see \cite[\S 3.5]{klin}).

\begin{defn}\label{perioddim} \hfill
  \begin{enumerate}
    \item
  A subvariety $Z$ of $S$ is said of {\em positive period dimension for $\V$} if $\Psi(Z^{\an})$ has
  positive dimension. 
 \item
  The \emph{Hodge locus of positive period dimension} 
  $\HL(S, \V^\otimes)_{\pos}$ is the
  union of the special subvarieties of $S$
  for $\V$ of positive period dimension. 
\end{enumerate}
\end{defn}

Here we follow the atypicality notion first introduced in \cite{2021arXiv210708838B}, rather than the one originally proposed by Klingler in \cite{klin}.

\begin{defn}\label{defatyphodge}
A special subvariety $Y\subset S$, with Hodge datum $(\mathbf{G}_Y,D_{Y})$ (associated to $\V_{| Y^{\operatorname{sm}}}$) is \emph{atypical} if
\begin{equation}\label{atypfirst}
\codim_{\Gamma \backslash D} \Psi (Y^{\operatorname{an}}) < \codim_{\Gamma \backslash D} \Psi(S^{\operatorname{an}}) + \codim_{\Gamma \backslash D} \Gamma_Y \backslash D_Y,
\end{equation}
or equivalently,
 \begin{displaymath}
 \dim \Psi (S^{\an})-\dim \Psi(Y^{\an})< \dim D- \dim D_Y .
 \end{displaymath}
Otherwise, it is said to be \emph{typical}. Here $\Gamma_Y $ denotes $\Gamma \cap \mathbf{G}_Y(\Q)$. 
\end{defn}

Note that we only have a natural map $\Gamma_Y \backslash D_Y \to \Gamma \backslash D$, so the codimension $\codim_{\Gamma \backslash D} \Gamma_Y \backslash D_Y$ is interpreted formally as $\dim \Gamma \backslash D - \dim \Gamma_{Y} \backslash D_{Y}$.

\begin{defn}
The \emph{atypical Hodge locus} $\HL(S,\V^\otimes)_{\atyp} \subset \HL(S, \V^\otimes)$
  (resp. the \emph{typical Hodge locus} $\HL(S,\V^\otimes)_{\typ} \subset \HL(S, \V^\otimes)$) is
  the union of the atypical (resp. strict typical) special subvarieties of $S$ for $\V$.
\end{defn}

In this paper we will be mainly concerned with the case of atypical special subvarieties. In particular we want to investigate the following:
\begin{conj}[Zilber--Pink conjecture for $\Z$VMHS]
$\HL(S,\V^\otimes)_{\operatorname{atyp}}$ is algebraic, i.e. there are only finitely many maximal (reduced) atypical special subvarieties.
\end{conj}

\subsection{Weakly special subvarieties in Hodge theory}
\label{wspsubsinHdgthy}

Essentially from \Cref{CDK}, we observe that there are countably many special subvarieties for $(S,\V)$. It is also easy to observe that an intersection of special subvarieties is a union of special subvarieties. In this section we compare such notions and properties with the more general concept of \emph{weakly special subvarieties} introduced in \Cref{def:weaklysp}. These arguments are essentially easy generalizations of the ones appearing in \cite{2021arXiv210708838B, KO} to the mixed case and most of them appear in \cite{2021arXiv210110938G}.

Much of what we will need is a consequence of the following \cite[Thm. 1]{MR1154159} (see also \cite{MR0498551}). Let $\V$ be a $\Z$VMHS on $S$. For any closed point $s\in S$, let $\mathbf{G}_{S,s}$ be the Mumford-Tate group of $(S, \V)$ at $\omega_{s}$, and $\mathbf{H}_{S,s}$ the algebraic monodromy at $s$ in the sense of \Cref{monodromydef} applied to the $\mathbb{Q}$-local system $\V_{\mathbb{Q}}$.

\begin{thm}[André-Deligne Monodromy Theorem]\label{monodromytheorem}
The monodromy group $\mathbf{H}_{S,s}$ is a normal subgroup of the derived subgroup of $\mathbf{G}_{S,s}$.
\end{thm}
\noindent We notice here that, in the general mixed case, the derived subgroup of $\mathbf{G}_{S,s}$ is not a semisimple group.

\begin{defn}
\label{weakmtdef}
 A subset $D'$ of the classifying space $D$ is called \emph{a weak Mumford-Tate subdomain} if there exists an element $t\in D'$ and a normal subgroup $\mathbf{N}$ of the derived group of $\MT(t)$ such that $D'= \mathbf{N}(\R)^+ \mathbf{N}^u(\C)\cdot t$, where $\mathbf{N}^u$ denotes the unipotent radical of $\mathbf{N}$.
\end{defn}

A \emph{weakly
special subvariety} of the Hodge variety $\Gamma
\backslash D$ is an irreducible component of the image in $\Gamma \backslash D$ of a weak Mumford-Tate subdomain. The datum $((\mathbf{H}, D_H), t)$ is called a \emph{weak Hodge subdatum} of
$(\mathbf{G}, D)$. Given $(S,\V)$ with associated period map $\Psi: S \to \Gamma \backslash D$, every weakly special subvariety (in the sense of \Cref{def:weaklysp}) of $(S,\V)$ can be described as an irreducible component of the preimage along $\Psi$ of a weakly special subvariety of the Hodge variety $\Gamma
\backslash D$; see \cite[Cor. 2.9]{2022arXiv220805182B}. 



To conclude we record the following fact about the monodromy of $\Z$VMHS. It is essentially a corollary of \Cref{monodromytheorem} and \cite[Lem. 3.3]{2021arXiv211003489C}, which asserts that if $G$ is a linear algebraic group whose quotient by its unipotent radical $G^u$ is semisimple, then $G$ is sparse (recall \autoref{sparsegroup}). 
\begin{lem}\label{lemmasparse}
The algebraic monodromy of a $\Z$VMHS is sparse.
\end{lem}

\subsection{Parameter spaces of mixed Hodge structures and weak atypicality}\label{paraspace}

For our exposition we essentially follow \cite[\S 3.5, and \S 6]{bakker2020quasiprojectivity}.

Let $V$ be a finite dimensional real vector space equipped with an increasing filtration $W_\bullet$ and a collection of non-degenerate bilinear forms $q_k : \Gr^W_k \otimes \Gr^W_k \to \R$ that are $(-1)^k$-symmetric. Fix a partition of the dimension of $V$ into non-negative integers $h^{p,q}$ satisfying the $h^{p,q}=h^{q,p}$ symmetry relation. For any integer $k$ we denote by $\ch{\Omega}_k$ the complex projective algebraic variety parametrizing the $(q_k)_{\C}$-isotropic filtrations $\{F'^p_k\}$ on $\Gr^W_k V_{\C}$ with $\dim F'^p_k = \sum _{r\geq p} h^{r,k-r}$. The complex group $\Aut (q_k)$ acts transitively by algebraic automorphisms on $\ch{\Omega}_k$.

Let $(V, W_\bullet, F^\bullet, q) $ be a mixed $\Z$-Hodge structure as above, polarized by $q$. We denote by $D$ the corresponding \emph{mixed period domain}: the set of decreasing filtrations $F'^\bullet$ of $V_\C$ such that $(V_{\R}, W_\bullet, F'^\bullet, q)$ is a real mixed Hodge structure graded-polarized by $q$ and such that
\begin{displaymath}
\dim_\C (F'^p \Gr_{p+q}^W\cap \overline{F'}^q \Gr_{p+q}^W )=\dim F^p \Gr_{p+q}^W\cap \overline{F}^q \Gr_{p+q}^W =h^{p,q}.
\end{displaymath}
By definition, $D$ is a semi-algebraic open subset of the smooth projective complex variety $\ch{D}$ that parametrizes the decreasing filtrations of $V_\C$ by complex vector subspaces such that the filtration induced on the graded pieces $\Gr^W_k V_\C$ is inside $\ch{\Omega}_k$ for each $k$. In particular $\ch{D}$ maps to the product of $\ch{\Omega}_k$. 

Given a $\Z$VMHS $\V$ on $S$, we have associated a mixed Hodge datum $(\mathbf{G}_S,D_S)$ as in \Cref{mixedhodgedatumdef} (we always assume that $\mathbf{G}_{S}$ is the generic Mumford-Tate group of $\V$). In this case, $\mathbf{G}_{S}$ comes with a faithful linear representation $\rho: \mathbf{G}_{S} \to GL(V)$, and thanks to this representation, we obtain a morphism from $D_S$ to the appropriate classifying space $\ch{D}$ introduced above. 

\paragraph{Notation} We denote by $\ch{D}_S$ the Zariski closure of $D_S$ in $\ch{D}$ (it comes with a transitive action of $\mathbf{G}_S(\C)$) and by $\ch{D}^0_S\subset \ch{D}_S$ the monodromy orbit of some base point $h_0 \in D_S$, and write $D^{0}_{S} = \ch{D}^{0}_{S} \cap D_S$. In particular to a $\Z$VMHS $\V$ on $S$ we obtain the complete \emph{Hodge-monodromy datum}:
\begin{displaymath}
(\mathbf{H}_S \subset \mathbf{G}_S, D^0_S \subset D_S)
\end{displaymath}

\begin{rem}
\label{opennessremark}
We note that $D_{S}$ (resp. $D_{S}^{\circ}$) is open in $\ch{D}_{S}$ (resp. $\ch{D}^{0}_{S}$), see \cite[Thm. A.1]{2021arXiv211003489C}.
\end{rem}

\noindent In \S\ref{new2} we prove the so called \emph{geometric part of Zilber-Pink}. For geometric Zilber-Pink we will work with a finer notion of atypicality, that we now recall. This corresponds to what, in \cite[Def. 5.3]{2021arXiv210708838B} is called \emph{monodromically atypical}. See also Lem. 5.6, Ex 5.7, and Rmk. 5.7 in \emph{op. cit.} for a more detailed discussion, in the pure case, of such a definition and comparisons with other variants. (The link between the two notions ultimately comes from \Cref{monodromytheorem}.) 
\begin{defn}\label{moatyp}
A weakly special subvariety $Y$ for $(S,\V)$ is called \emph{monodromically atypical} if 
 \begin{equation}\label{eqatypi}
  \dim \Psi (S^{\an})-\dim \Psi(Y^{\an})< \dim  \ch{D}^{0}_{S} - \dim  \ch{D}^{0}_{Y}.
 \end{equation}
\end{defn}

\begin{rem}
Suppose $\Psi$ is quasi-finite, so every point is a weakly special subvariety. Then unless $\Psi$ is dominant, each point is in fact a monodromically atypical weakly special subvariety.
\end{rem}

Geometric Zilber-Pink says, roughly, that there are finitely many \emph{families} of maximal monodromically atypical subvarieties. (The more precise statement gives additional information about the Hodge-theoretic data defining those families.) In the case where the ambient monodromy group of $S$ is $\mathbb{Q}$-simple, this can be used to recover special cases of the true Zilber-Pink conjecture for positive-dimensional atypical intersections, which is why one often prefers the statement appearing in \autoref{thmhodgelocus}. 

One subtlety about the generalization of the argument from \cite{2021arXiv210708838B} and other prior work in the pure setting is that semisimplicity of monodromy is often used (usually to say that a group, up to conjugation, has only finitely many subgroups). This is not true in the general mixed case, so our extension will require more care to ensure that atypical mixed intersections can still be appropriately ``over-parameterized'' (recall \autoref{overparamrem}) in line with the strategy discussed in \S\ref{sec:unfi}.

We also recall the Ax-Schanuel theorem for period maps. It is implied by the more general theorem \autoref{thm:newAS}. The original proof of the theorem was established through a series of papers, see \cite{MR3958791, 2021arXiv210110968C, 2021arXiv210110938G, 2022arXiv220805182B}. 

\begin{thm}[Ax-Schanuel theorem in the period domain]\label{astheoremperioddomain}  
Set $\ch{D}^{0} = \ch{D}^{0}_{S}$. Let $W \subset S \times \ch{D}^0$ be an algebraic subvariety. Let $U$ be an
irreducible complex analytic component of $W \cap S \times_{\Gamma \backslash D^{0}} D^{0}$
such that 
\begin{displaymath}
\codim_{S \times \ch{D}^0} U< \codim_{S \times \ch{D}^0} W+ \codim_{S
  \times \ch{D}^0} (S \times_{ \Gamma \backslash D^0}
D^{0})\;\;.  
\end{displaymath}
Then the projection of $U$ to $S$ is contained in a strict weakly
special subvariety of $(S,\V)$.
\end{thm}
Notice that $S \times_{\Gamma \backslash D^{0}} D^{0}$ is the image of the graph of $\tilde{\Phi}: \widetilde{S} \to D^{0}$ under $
\widetilde{S} \times D^{0} \to S \times D^{0}$. Then the intersection $W \cap S \times_{\Gamma \backslash D^{0}} D^{0}$ can be identified with the intersection in $S \times D^{0}$ between $W$ and the image of $ \widetilde{S}$ in $S \times D^{0}$ along the map $(\pi, \tilde{\Phi})$. 

\subsection{Relation with period torsors associated to $\mathbb{V}$}\label{constructionBT}

The relationship between \autoref{astheoremperioddomain} and \autoref{cor:asloc} is ultimately explained by the torsor $P$ in \autoref{examplebundle} associated to the underlying local system of $\mathbb{V}$ and appears also in the work of Bakker-Tsimerman \cite{2022arXiv220805182B}. We stress here that $P$ is a torsor under the complex points of the algebraic monodromy of the VMHS, rather than its Mumford-Tate group.

Let $(\mathbf{H}_{S}, D^{0}_{S})$ be the weak Hodge datum associated to $\mathbb{V}$. We may fix a basepoint $s_{0} \in S$, and regard this data as being associated to the fibre $\mathbb{V}_{s_{0}}$. Then in particular, $\ch{D}^{0} = \ch{D}^{0}_{S}$ is a variety parameterizing mixed Hodge data at $s_{0}$. Recall that $P$ was constructed as subbundle of the total space of the vector bundle $\sheafhom(\mathcal{V}, \mathcal{O}_{S} \otimes_{\mathbb{C}} \mathcal{V}_{s_{0}})$, where $\mathcal{V}$ is the algebraic bundle underlying $\mathbb{V} \otimes_{\mathbb{Z}} \mathcal{O}_{S^{\textrm{an}}}$. Then we obtain a natural algebraic map $\nu : P \to \ch{D}^{0}$. Concretely, this map acts as 
\[ [\eta \in \Hom(\mathcal{V}_{s}, \mathcal{V}_{s_{0}})] \mapsto \eta(W_{\bullet,s}, F^\bullet_s) , \]
where we use that both the Hodge flag and the filtration by weight are given by algebraic subbundles of $\mathcal{V}$. This is a combination of the fact that the weight filtration on $\mathbb{V}$ is given by local subsystems to which we can apply the Riemann-Hilbert functor, together with the corresponding result for the Hodge filtration on the pure graded quotients, which is a consequence of work of Schmid \cite{sch73}. To check that the map is well-defined (the image lands inside $\ch{D}^{0}$ rather than some larger flag variety), it suffices by $\mathbf{H}_{S}(\mathbb{C})$-invariance to check this for a single leaf, where it follows from the fact that the associated period map lands inside $\ch{D}^{0}$. 

To recover \autoref{astheoremperioddomain} from \autoref{cor:asloc} one pulls back the intersections appearing in the statement of \autoref{astheoremperioddomain} along the map $\pi \times \nu : P \to S \times \ch{D}^{0}$. We note that the image of a leaf $\mathcal{L} \subset P$ which arises by parallel translating the identity $\textrm{id} \in \Hom(\mathbb{V}_{s_{0}}, \mathbb{V}_{s_{0}})$ using the integral structure of $\mathbb{V}$ has as image in $S \times \ch{D}^{0}$ exactly the fibre product $S \times_{\Gamma \backslash D^{0}} D^{0}$ appearing in \autoref{astheoremperioddomain}. Since the map $\pi \times \nu$ is surjective, the atypicality condition is preserved. See indeed the discussion in \cite{2022arXiv220805182B}.

\subsection{Families of weakly specials in the mixed setting}\label{sec:familieswp}
In our applications of the results of \S\ref{axschanfamsec}, we will need the following fact.

\begin{prop}
\label{rigidprop}
Let $(S,\mathbb{V})$ be a $\Z$VMHS. Then there exists a countable collection $\{ h_{i} : C_{i} \to B_{i} \}_{i=1}^{\infty}$ of families of irreducible algebraic varieties of $S$, all of whose fibres are weakly special, and such that every weakly special subvariety arises as a fibre in some such family. Moreover, the maps $C_{i} \to S$ may be assumed quasi-finite. 
\end{prop}

\begin{proof}
Let $Y \subset S$ be a weakly special subvariety with associated monodromy/Hodge datum $(\mathbf{H}_{Y}\subset \mathbf{G}_{Y}, D_{Y})$. Let $Y^{\textrm{sp}}$ be a special subvariety containing $Y$ defined by $D_{Y}$. Then we can construct a finite cover $e : \widetilde{Y}^{\textrm{sp}} \to Y^{\textrm{sp}}$ and a period map $\Xi : \widetilde{Y}^{\textrm{sp}} \to \Gamma_{\widetilde{Y}^{\textrm{sp}}/Y} \backslash D_{\widetilde{Y}^{\textrm{sp}}/Y}$ associated to the quotient datum $(\mathbf{G}_{\widetilde{Y}^{\textrm{sp}}/Y}, D_{\widetilde{Y}^{\textrm{sp}}/Y}) = (\mathbf{G}_{Y}, D_{Y})/\mathbf{H}_{Y}$ (in the sense of \cite[\S 5]{2021arXiv210110938G}). From \cite[Prop. 5.1 (ii)]{2021arXiv210110938G}, the irreducible components $\widetilde{Z}$ of the fibres of $\Xi$ are weakly special subvarieties of $\widetilde{Y}^{\textrm{sp}}$ for the period map $\Xi$, and their images $Z$ in $Y^{\textrm{sp}}$ are weakly special subvarieties of $Y^{\textrm{sp}}$ for $\restr{\mathbb{V}}{Y^{\textrm{sp}}}$ (use that the algebraic monodromy group is unchanged by passing to a finite covering), and in fact weakly special subvarieties of $S$ for $\mathbb{V}$. 

Although a priori depending on $Y$, the construction above in fact depends only on the choice of $Y^{\textrm{sp}}$ and a $\mathbb{Q}$-subgroup of $\mathbf{G}_{Y}$; there are countably many such choices. From either the main result of \cite{bakker2020quasiprojectivity}, or by first taking a proper compactification of $\Xi$ as in \cite[Prop. 2.4]{bakker2020quasiprojectivity} and applying the Stein factorization, it follows that the fibres of $\Xi$ belong to a common algebraic family $g : \widetilde{Y}^{\textrm{sp}} \to T$. We conclude thanks to \autoref{imghasfamlem} below, applied to $(g,e)$. 
\end{proof}

\begin{lem}
\label{imghasfamlem}
Let $e : N \to M$ be a quasi-finite map, and $g : Q \to T$ a family of subvarieties of $N$. Then there exists a family $h : C \to B$, possibly over a disconnected base, of subvarieties of $M$ whose fibres are exactly the closures of images under $e$ of components of fibres of $g$.
\end{lem}

\begin{proof}
We may assume that $e$ is dominant, and that $N$, $M$ and $T$ are irreducible. By stratifying $T$ with respect to the dimension of its fibres, we may assume that all fibres have the same dimension. By removing a closed locus from $Q$, one can further assume that the fibres of $g$ are equidimensional.

 We observe that there exists some finite-type family $h : C \to B$ of subvarieties of $M$, all of whose fibres are irreducible, and such that all components of varieties of the form $\overline{e(g^{-1}(t))}^{\textrm{Zar}}$ occur as fibres. This can be deduced by observing that, with respect to a fixed line bundle $\mathcal{L}$ on some projective compactification $M \subset \overline{M}$, the varieties $\overline{e(g^{-1}(t))}^{\textrm{Zar}}$ all have uniformly bounded degree, so the locus of interest can be cut out of a Hilbert scheme associated to $\overline{M}$. We then argue that the locus $L$ of $b \in B$ such that $h^{-1}(b)$ agrees with a component of $\overline{e(g^{-1}(t))}^{\textrm{Zar}}$ for some $t$ is constructible. It suffices to consider the induced family $h' : C' \to B$ whose fibre above $b$ is $e^{-1}(h^{-1}(b))$, and construct the locus in $T \times B$ where some component of $g^{-1}(t)$ contains some component of $h'^{-1}(b)$. Since both families have equidimensional fibres, this can be achieved by imposing a dimension condition on the intersection and so constructibility of this locus follows from \cite[\href{https://stacks.math.columbia.edu/tag/05F9}{Lemma 05F9}]{stacks-project}. Having done this, we can partition $L$ by locally closed subvarieties of $B$, and replace $B$ with the disjoint union of the strata in this partition. Afterwards, the fibres of $h$ are exactly all those varieties of the form $\overline{e(g^{-1}(t))}^{\textrm{Zar}}$ for some $t \in T$.
\end{proof}

\section{Recap on orbit closures}\label{sec:orbitapp}

In this section we survey some known results on the orbit closures, expanding the discussion from the beginning of \Cref{intro:teich}, and recall how they relate to the map
\begin{displaymath}
\Omega \mathcal{M}_{g}(\kappa)\to \Omega \mathcal{A}_{g,n}.
\end{displaymath}
We mainly follow \cite{zbMATH06455787, filipnotes}, but see also \cite{zbMATH05183579} for more background about the ergodic viewpoint.

In the sequel, $\mathcal{M}_g$ denotes the moduli space of genus $g$ compact Riemann surfaces and $\Omega \mathcal{M}_g \to \mathcal{M}_g$ the bundle whose fibre over $[X]\in \mathcal{M}_g$ is the space of all (non-zero) holomorphic 1-forms on $X$ (in particular it has dimension $4g-3$). It is a rank $g$ (algebraic) vector bundle on $\mathcal{M}_g$ (with the zero section removed), naturally equipped with a stratification by smooth (not necessarily irreducible) algebraic subvarieties $\Omega \mathcal{M}_g (\kappa)$, where $\kappa=(\kappa_0,\dots, \kappa_n)$ satisfies $\sum_i \kappa_i = 2g-2$ and denotes the multiplicities of zeros of $\omega$. Notice that $\Omega \mathcal{M}_g(\kappa)$ has dimension $2g+n$, where $n=\textrm{length}(\kappa)-1$. For example the only open stratum is $\Omega  \mathcal{M}_g (1,...,1)$ (where the 1 is repeated $2g-2$ times).

Below we use some rudiments of variational (mixed) Hodge theory. We have already discussed such theory in greater generality in \S\ref{sec:zpapp}, but here we only deal with variation of Hodge structures (VMHS) whose weight filtration has two steps. In fact in the orbit closure setting all of our period spaces will belong to the theory of mixed Shimura varieties of \emph{Kuga type}; see \cite{zbMATH07305885} for an introduction to this notion. 

\begin{rem}
In what follows, we often secretly view $\Omega \mathcal{M}_{g}$ and its strata as orbifolds or Deligne-Mumford stacks. To avoid this, one can replace $\Omega \mathcal{M}_g$ by a fixed covering over which the pairs $(X, \omega)$ glue together into a universal algebraic family and always work just with algebraic varieties and variations of Hodge structure on this covering. This will be harmless for our considerations (the finiteness/abundance of orbit closures in a stratum can equivalently be detected on a finite covering).
\end{rem}

\subsection{Recollection and first properties}
 We first recall some basic facts about \emph{relative cohomology}. Let $(X,Z)$ be a pair of a Riemann surface (i.e. an irreducible smooth projective algebraic curve over $\C$) and $n+1$ points $Z=\{z_0, z_1, \dots, z_n\} \subset X$ (that will usually be given by the zeros of some (non-zero) $1$-form $\omega$ on $X$). We have a short exact sequence:

\begin{equation}\label{eqrelcoho}
0 \to \tilde{H}^0(Z; \Z) \to H^1(X,Z; \Z)\to H^1(X;\Z)\to 0.
\end{equation}
Above, $\tilde{H}^0(Z; \Z)$ is the \emph{reduced cohomology}: it denotes the group of all formal linear combinations of points in $Z$, modulo the element which takes each point in $Z$ with coefficient $1$ (it is the dual of the reduced homology group).
We will always refer to the middle term as \emph{relative cohomology} and the last term as \emph{absolute cohomology}. Note that the form $\omega$ naturally induces a class inside both $H^1(X,Z(\omega); \Z)$ and $H^1(X;\Z)$ by integrating over homological cycles and using homology-cohomology duality.

We have a version of \eqref{eqrelcoho} for flat bundles above $\Omega \mathcal{M}_g (\kappa)$ (see also \cite[3.1.13]{filipnotes}):
\begin{equation}
\label{bundextseq}
0 \to W_0 \to H^1_{rel}\to H^1 \to 0.
\end{equation}
The notation is justified by the fact that $W_0$ denotes the weight-zero local subsystem for the mixed VHS structure on $H^{1}_{rel}$. 

The $\GL_2(\R)$-action on $H^1(X,Z(\omega);\C)$ is given by identifying $\C$ with $\R^2$ and considering
\begin{displaymath}
H^1(X,Z(\omega);\C)\cong H^1(X,Z(\omega);\R)\otimes_{\R} \R^2.
\end{displaymath}
The action is now the natural $\GL_2(\R)$ on the $\R^2$ at the very end of the equation (it is less straightforward to describe in a more global way the $\GL_2(\R)$ action on the moduli space of translation surfaces).

We now recall the concept of developing map and monodromy. To do so, let $\operatorname{Mod}(X,Z)$ be the mapping class group of diffeomorphisms of $X$ preserving the subset $Z$. We fix a base point $s:= (X, \omega )\in \Omega \mathcal{M}_g (\kappa)$, and look at the developing map/period coordinates on its universal covering
\begin{equation}\label{def:dev}
\operatorname{Dev}: \widetilde{\Omega \mathcal{M}_g (\kappa)}\to H^1(X,Z(\omega); \C).
\end{equation}
Concretely, the map sends $(X',\omega')$ to $\omega' \in H^1(X',Z(\omega'); \C)$ and then applies the flat transport to obtain an element of $H^1(X,Z(\omega); \C)$.
By a theorem of Veech \cite[Thm. 7.5]{zbMATH04190113}, locally $\operatorname{Dev}$ is biholomorphic. Moreover $\operatorname{Dev}$ is equivariant for a representation of
\begin{displaymath}
\pi_1(\Omega \mathcal{M}_g (\kappa))\to  \operatorname{Mod}(X,Z),
\end{displaymath}
where the mapping class group acts on cohomology by
\begin{equation}\label{eqmonodromy}
L : \operatorname{Mod}(X,Z)\to \Sp (H^1(X,Z))\cong \Sp (H^1(X))\ltimes \operatorname{Hom} (H^1(X),\tilde{H}^0(Z; \C) ).
\end{equation}

All monodromy groups presented so far are naturally subgroups of $\operatorname{Aut}(H^1_{rel,s})$: this is the group of automorphisms of $H^1_{rel,s}$ that act as the identity on $W_{0,s}$ and act symplectically on $H^1_s$. If we pick a splitting $H^1_s \to H^1_{rel,s}$, we obtain a semidirect product decomposition $\Sp(H^1_s)\ltimes\Hom (H^1_s, W_{0,s})$. Indeed we have
\begin{displaymath}
0 \to \Hom (H^1_s, W_{0,s})\to \operatorname{Aut}(H^1_{rel,s})\to \Sp (H^1_s)\to 0.
\end{displaymath}
 Recall that an element $\xi \in \Hom (H^1_s, W_{0,s})$ acts as a unipotent automorphism of $H^1_{rel,s}$ by $v \mapsto v + \xi (p(v))$ (where $p: H^1_{rel}\to H^1$ is the projection to absolute cohomology). All automorphisms of $H^1_{rel,s}$ that act as the identity on $W_{0,s}$ and symplectically on $H^1_s$ are in fact of this form.

The bundle $H^1_{rel}$ carries naturally the structure of an admissible, graded-polarized, variation of $\Z$-mixed Hodge structures on (each irreducible component of) $\Omega \mathcal{M}_g (\kappa)$. The weight filtration on $\V$ has two steps (for facts and definitions about mixed VHS, we refer to \Cref{sec:zpapp}):
\begin{equation}\label{mainvhsfororbits}
  0 \to W_0 (\V)\to \V\to \operatorname{Gr}_1^W(\V) \cong p^* R^1f_*\Z\to 0.
\end{equation}
Here $p:  \Omega \mathcal{M}_g (\kappa) \to \mathcal{M}_g $ is just the projection, and $f: \mathcal{C}_g\to \mathcal{M}_g$ is the universal family of genus $g$ curves. The local system $W_0 (\V)$ becomes isomorphic to $ \Z(0)^{n}$ after passing to a finite covering. See also \cite[\S 3.3 and \S 3.4]{2022arXiv220206031K} for a related discussion (in \emph{op. cit.} the authors find $n-1$ copies of $ \Z(0)$ since they started with $n$ points, whereas we had $n+1$-points; this is coherent with our choice of denoting the multiplicities of the zeros of $\omega$ by $(\kappa_0, \dots, \kappa_n)$). 

Finally, we refer to \cite{zbMATH02001031} (see also \cite[3.1.16]{filipnotes}) for a description of the connected components of $\Omega \mathcal{M}_g (\kappa) $, that we will implicitly use when speaking of atypical orbit closures.

\subsection{Structure of orbit closures}
\label{strsec}

We recall a special case of a theorem of Eskin, Mirzakhani, and Mohammadi \cite[Thm. 2.1]{zbMATH06487150} see also \cite[Thm. 1]{filipnotes}:
\begin{thm}[Linearity/topological rigidity]\label{linearity}
Each orbit closure $\mathcal{N} \subset \Omega \mathcal{M}_g(\kappa)$ is locally in period coordinates a \emph{linear manifold}, i.e. any sufficiently small open $U\subset \mathcal{N}$ is such that $\operatorname{Dev}(U)$ is an open set inside a linear subspace (see \eqref{def:dev} for the definition of the developing map). 
\end{thm}
Let $\mathcal{N} \subset \Omega \mathcal{M}_g(\kappa)$ be an orbit closure, and denote by $T \mathcal{N}$ its tangent bundle. The linearity of orbit closures, as recalled in \Cref{linearity}, allows one to realize $T \mathcal{N}$ as a local subsystem of $H^1_{rel}$. As a consequence we have another short exact sequence of real bundles above $\mathcal{N}$:
\begin{equation}
0 \to W_0 (T \mathcal{N})\to T \mathcal{N} \xrightarrow{p} H^1 (T \mathcal{N})\to 0 .
\end{equation}
To explain the notation: $H^1 (T \mathcal{N})$ is just the image of $T \mathcal{N} \subset H^1_{rel}$ in $H^1$  (the absolute cohomology), and $W_0(T \mathcal{N}) = T \mathcal{N} \cap W_0$.

Let $T \subset \SL_2(\mathbb{R})$ denote the upper triangular group.

\begin{thm}[Isolation property, {\cite[Thm. 2.3]{zbMATH06487150}, see also \cite[Thm. 4.1.8.]{filipnotes}}]\label{isolation}
For every sequence of orbit closures $(\mathcal{N}_i)$ there are $T$-invariant measures $\mu_i$, and after passing to a subsequence $(\mathcal{N}_i,\mu_i)$ there is another linear immersed submanifold $\mathcal{M}$, with a finite $T$-invariant measure $\mu$ such that $\mathcal{N}_i\subset \mathcal{M}$ and $\mu_i \to \mu$.
\end{thm}

\noindent Note that the above result implies the same statement with $T$ replaced by $\SL_2(\mathbb{R})$, which is how we will use it; see \cite[Rem. 4.1.9 (ii)]{filipnotes}.

We can now recall the following theorem of Filip \cite{zbMATH06575346} (which builds on the earlier work \cite{zbMATH06641855}, as well as earlier work of Wright \cite{zbMATH06323273} and M\"{o}ller \cite{zbMATH05015436}), see also \cite[Thm. 2 and Thm. 4.4.2]{filipnotes}. We denote by $K=K_{\mathcal{N}}\subset \R$ the smallest field over which $\mathcal{N}$ is $K$-linear in the sense of \Cref{linearity}. Let $\mathcal{J}$ be the abelian scheme over $\Omega \mathcal{M}_g(\kappa)$ whose fiber at $(X,\omega)$ is the Jacobian of $X$. We have:

\begin{thm}[S. Filip]\label{mainthmalgfilip}
Let $\mathcal{N}\subset \Omega \mathcal{M}_g(\kappa)$ be an orbit closure. There exists, up to isogeny, a non-trivial factor $\mathcal{F}\subset \mathcal{J}_{| \mathcal{N}}$ and a subgroup $\mathcal{S}$ of the free abelian group on the zeros of $1$-forms, such that:
\begin{enumerate}
\item $\mathcal{F}$ admits real multiplication by an order of the ring of integers of $K$, which is, in particular, a totally real number field\footnote{i.e. $K$ is equal to the rational endomorphism ring of $\mathcal{F}$ over the generic point of $\mathcal{N}$, and see \cite[4.3.18]{filipnotes} for more about \emph{orders}.};
\item The Abel-Jacobi map, possibly twisted by real multiplication, $AJ: \mathcal{S}\to \mathcal{F}$ assumes torsion values;
\item For each $(X,\omega)\in \mathcal{N}$, $\omega$ is an eigenform for the real multiplication on the fibre of $\mathcal{F}$ above $X$ (see also \cite[\S 4.2.17]{filipnotes}).
\end{enumerate}
Furthermore, these conditions, together with a dimension bound (cf. \Cref{explaindimbound}), characterize $\mathcal{N}$.
\end{thm}
(See \cite[\S 4.5]{filipnotes} for more details on the \emph{twisted torsion} condition.) 

\begin{rem}
Notice also that in \cite[Thm. 2 and Thm. 4.4.2]{filipnotes}, the factor $\mathcal{F}$ is implicitly understood as a $\Q$-\emph{factor} of the associated pure $\mathbb{Z}$VHS. See also \cite[Thm. 2.7]{zbMATH05015436}, for a related statement highlighting the possibility of the splitting up to isogeny.
\end{rem}

\begin{cor}\label{thmstrorbit}
Each orbit closure is a closed algebraic subvariety of $\Omega\mathcal{M}_g$. It is defined over $\Qbar$ and its Galois conjugates are also orbit closures.
\end{cor}
From \Cref{mainthmalgfilip}, one may introduce \emph{real multiplication, torsion and eigenform} invariants (see \Cref{diagramatyp}, for a geometric description). In particular, we write
\begin{itemize}
\item  $r=r _\mathcal{N}$, the \emph{(cylinder) rank} $\frac{1}{2}\dim H^1(T\mathcal{N})$ (which is in fact an integer, since $T\mathcal{N}$ is known to be symplectic thanks to \cite{zbMATH06801922});
\item $d=d_\mathcal{N}$, to denote the degree of the number field defining the linear equations of $\mathcal{N}$ (that is $[K:\Q]$);
\item Finally $t=t_\mathcal{N}$ the \emph{torsion corank} is $\operatorname{rank} W_0(T \mathcal{N})$.
\end{itemize}
With these notations, the dimension bound appearing at the end of \Cref{mainthmalgfilip} is just
\begin{displaymath}
\dim \mathcal{N}=2r_{\mathcal{N}}+t_{\mathcal{N}}.
\end{displaymath}

\begin{rem}\label{remkzarcl}
Combining \Cref{thmstrorbit} and \Cref{isolation}, we observe that the Zariski closure of a collection of orbit closures is a finite union of orbit closures. Cf. the beginning of the proof of \cite[Thm. 5.4.5.]{filipnotes}.
\end{rem}

\begin{rem}
\label{explaindimbound}
Let us comment on the ``dimension bound'' appearing in \autoref{mainthmalgfilip} (cf. \cite[Rem. 5.5]{zbMATH06575346} and the discussion in the proof of \cite[Cor. 5.7]{zbMATH06575346}). Any locus $L$ satisfying the first three conditions in \autoref{mainthmalgfilip} will always have dimension \emph{at most} $2r+t$, where $r$ and $t$ are integers naturally associated to the data. This occurs for the following reasons.
\begin{itemize}
\item[(1)] The third condition in \autoref{mainthmalgfilip} implies that, over $L$, the form $\omega$ lies in an eigensubspace of $H^{1}_{\textrm{abs}}$ which has dimension $2r$, and hence imposes $2g - 2r$ linear conditions.
\item[(2)] For holomorphic forms over $L$ lying in such an eigensubspace, the second condition in \autoref{mainthmalgfilip} implies linear conditions on the relative periods (cf. \cite[Rem. 1.5]{zbMATH06575346}). As explained in \cite[\S4.5.4]{filipnotes}, the number of such conditions imposed on periods of $\omega$ is $n-t$, where $n$ is the number of marked points associated to the ambient stratum $\Omega \mathcal{M}_{g}(\kappa)$. 
\item[(3)] The conditions are independent, so since $\rank H^{1}_{\textrm{rel}} = 2g + n$ one finds in flat coordinates over $L$ that the form $\omega$ lies in a subspace of dimension $2r + t$.
\item[(4)] The restriction of the developing map (\ref{def:dev}) to $L$ is locally injective, so this implies $\dim L \leq 2r + t$.
\end{itemize}
\end{rem}

\subsection{Orbit closures and relative dynamical (a)typicality}\label{sectionrelatypefw}

In this subsection we first recall the definition of (a)typicality appearing in \Cref{thmEFWnew}. Actually, rather than giving the intrinsic definition as in \cite{zbMATH06890813, filipnotes} in terms of \emph{algebraic hulls}, we give one that captures the essential property of algebraic hull computations which will be relevant to us.

\begin{defn}\label{tautplane}
 At each $p =(X,\omega)\in \mathcal{N}$, we have the \emph{tautological 2-dimensional subspace} $T(p) := \operatorname{Span} (Re(\omega), Im(\omega)) \subset (H^1_{rel})_p$ (and analogously for $H^1_{abs}$). 
\end{defn}
We remark here that the above construction gives a $\mathbf{SL}_2(\R)$-invariant subbundle, which is not flat. See also \cite[\S 5.3.7.]{filipnotes}.

\begin{defn}
Given an orbit closure $\mathcal{N}$, the \emph{algebraic hull} for the $\mathbf{GL}_2(\R)$-action on $H^1_{rel,\mathcal{N}}$ is the subgroup of the monodromy group (at a fixed point $p\in \mathcal{M}$) $\mathbf{H}_{\mathcal{N}, \mathbb{R}}$ stabilizing the tautological plane $T =T(p)$, introduced in \Cref{tautplane}.
\end{defn}

The following is implicit in \cite[Thm. 1.3]{zbMATH06890813}, and the (a)typical terminology can be found, explicitly, in \cite[\S 5.1.13]{filipnotes} (see also \S 5.1.16 and Thm. 5.4.5 in \emph{op. cit.}):

\begin{defn}[Eskin-Filip-Wright]\label{dyanatyp}
Let $\mathcal{N}\subset \mathcal{M}$ be an orbit closure. We say that $\mathcal{N}$ is \emph{(dynamically) typical relatively to $\mathcal{M}$} if the algebraic hull of $\mathcal{N}$ is equal, up to compact factors, to the algebraic hull of $\mathcal{M}$. Otherwise, $\mathcal{N}$ is said \emph{(dynamically) atypical relatively to $\mathcal{M}$}.
\end{defn}

We are finally ready to recall the theorem of Eskin, Filip, and Wright stated in the introduction (to ease the notation, we omit the parenthetical ``dynamically''): 
\efwnew*
Our goal is to introduce an a priori different notion of (a)typicality defined directly from their Hodge-theoretic description, rather than a dynamical description. Even this Hodge-theoretic viewpoint is strongly influenced by Filip's results described above, and it is essentially taken from \cite{filipnotes}. As we will see below, some extra care is needed when $\mathcal{M}$ is different from the ambient stratum of $\Omega \mathcal{M}_g$. (In particular our proof will use the results from \Cref{strsec} but not the ones on equidistribution of algebraic hulls from \cite{zbMATH06890813}.)

\subsection{Local Systems and Monodromy}
\label{locsysandmonosec}

We recall several facts about the natural local systems appearing over an orbit closure $\mathcal{M}$. We let $(\mathcal{V}, \nabla)$ be the algebraic vector bundle with integral connection associated to $\mathbb{V} = H^{1}_{\operatorname{rel}}$. We denote by $\mathcal{V}_{\operatorname{abs}}$ the graded quotient of $\mathcal{V}$ which corresponds to $H^1_{\operatorname{abs}}$. Applying \cite[Thm. 1.5]{zbMATH06323273} we have a direct sum decomposition 
\begin{equation}
\label{locsystemdecomp}
\mathbb{V}_{\textrm{abs}} = \underbrace{\left(\bigoplus_{\iota : K_{\mathcal{M}} \hookrightarrow \mathbb{C}} \mathbb{U}_{\iota}\right)}_{\mathbb{U}} \oplus \mathbb{W}
\end{equation}
of variations of pure Hodge structures \cite[Thm. 8.2]{zbMATH06641855}, where $\mathbb{U}$ and $\mathbb{W}$ are $\mathbb{Q}$-summands and the $\mathbb{U}_{i}$ are $\mathbb{R}$-summands. Moreover the summands $\mathbb{U}_{i}$ and $\mathbb{U}$ are isotypic, again as a consequence of \cite[Thm. 1.5]{zbMATH06323273}. Applying the Riemann-Hilbert functor we obtain a corresponding decomposition 
\[ \mathcal{V}_{\textrm{abs}} = \underbrace{\left (\bigoplus_{\iota : K_{\mathcal{M}} \hookrightarrow \mathbb{C}} \mathcal{U}_{\iota} \right)}_{\mathcal{U}} \oplus \mathcal{W} \]
of filtered vector bundles with regular singular connections. 

As explained in \S\ref{strsec}, the tangent bundle $T\mathcal{M}$ is naturally a flat subbundle of $\mathbb{V}_{\mathbb{C}}$, and by \cite[Thm. 1.5]{zbMATH06323273} its projection $p(T\mathcal{M})$ to $\mathbb{V}_{\operatorname{abs}, \mathbb{C}}$ can be identified with one of the $\mathbb{U}_{\operatorname{id},\mathbb{C}}$. The algebraic monodromy group of $T \mathcal{M}$, viewed as a local system, is known to be the full semi-direct product $U \ltimes \Sp(T \mathcal{M}))$, where $U$ is a unipotent group isomorphic to $\Hom(p(T \mathcal{M}), W_{0}(T\mathcal{M}))$. This is explained in \cite[Prop. 4.7]{zbMATH06890813}, and uses Wright's cylinder deformations \cite{zbMATH06323273}. The same results are also true for the Galois conjugates of $T \mathcal{M}$ with respect to the rational structure of $\mathbb{V}$, one learns that the algebraic monodromy groups of each $\mathbb{U}_{\iota}$ are the full symplectic groups $\Sp(\mathbb{U}_{\iota, s_{0}})$, where $s_{0} \in \mathcal{M}$ is any fixed point. For the case where $\mathcal{M}$ is a component of $\Omega\mathcal{M}_{g}(\kappa)$, these monodromy results can also be deduced from the work of Calderon-Salter \cite{zbMATH07367029}.

\section{Finiteness of atypical orbit closures and density of the typical ones}\label{new1}

In this section we study orbit closures via various \emph{period maps}. Denote by $\mathcal{A}_g$ the moduli space of principally polarized abelian varieties of dimension $g$, and by $j$ the Torelli morphism $j: \mathcal{M}_g \to \mathcal{A}_g, \ X\mapsto (J(X), \theta_X)$ which associates to a genus $g$ curve $X$ its Jacobian $J(X)$ (naturally equipped with its principal polarization $\theta_X$). For every $n$, let $\mathcal{A}_{g,n} \to \mathcal{A}_g$ denote the fibration whose fibre over a point $[A]\in \mathcal{A}_g$ is $\operatorname{Sym}^{[n]}A$, the unordered $n$-tuples of points of $A$.

The varieties $\mathcal{A}_{g}$ and $\mathcal{A}_{g,n}$ can be interpreted as mixed period spaces (or mixed Shimura varieties). For $\mathcal{A}_{g}$ this is the standard Siegel space description, and a detailed construction of $\mathcal{A}_{g,1}$ as a mixed Shimura variety appears in \cite[\S2]{zbMATH07305885}. In general one has a natural bijection 
\[ \mathcal{A}_{g,n} \simeq \{ (H, E) : H \in \Sp_{2g}(\mathbb{Z}) \backslash \mathbb{H}_{g}, \hspace{0.5em} E \in \textrm{Ext}^{1}_{\textrm{MHS}}(H, \mathbb{Z}^{n})) \} , \]
where we view $\Sp_{2g}(\mathbb{Z}) \backslash \mathbb{H}_{g} \simeq \mathcal{A}_{g}$ as the moduli space for principally polarized pure weight one $\mathbb{Z}$-Hodge structures, and we view $\mathbb{Z}^{n}$ as a pure weight zero Hodge structure; see \cite[4.3.14]{filipnotes}. Using (\ref{bundextseq}) and applying this bijection, one obtains an algebraic mixed period map $\varphi : \Omega \mathcal{M}_{g}(\kappa) \to \mathcal{A}_{g,n}$ which sends $(X, \omega)$ to the isomorphism class of the extension described by (\ref{eqrelcoho}). The data of $\varphi$ is equivalent to the data of the mixed $\mathbb{Z}$VHS $H^{1}_{\textrm{rel}}$ (at least when viewed as a map of orbifolds). Note that $\varphi$ is quasi-finite, cf. the discussion at the end of \cite[\S 3.4]{2022arXiv220206031K}. We will typically work with the factorization
\begin{displaymath}
\Omega \mathcal{M}_g(\kappa) \xrightarrow{\Omega \varphi} \Omega \mathcal{A}_{g, n} \to  \mathcal{A}_{g, n},
\end{displaymath}
of $\varphi$, where the first arrow is obtained by seeing the differential $\omega$ on the Jacobian of $X$.

\begin{rem}\label{mixedrem} 
With the language of (mixed) Shimura varieties one can equivalently describe $\mathcal{A}_{g, n}$ as follows. The classifying space for the extensions of a weight one $\Z$VHS of dimension 2g by $\Z(0)$ is the mixed Shimura variety
$\mathbb{A}_g$. It is the universal principally polarized abelian variety of dimension $g$ over $\mathcal{A}_g$. Finally
\begin{displaymath}
\mathcal{A}_{g, n}\cong  \mathbb{A}_g\times_{\mathcal{A}_g} \mathbb{A}_g \times_{\mathcal{A}_g} \dots \times_{\mathcal{A}_g}\mathbb{A}_g/S_n,
\end{displaymath}
product of $n$ factors, modulo the action of $S_n$ (the symmetric group on $n$ letters). In fact, $\mathcal{A}_{g,n}$ is a \emph{mixed Shimura variety of Kuga type}, see \cite[\S 2]{zbMATH07305885} for more details on this. This means that, with the notation of \S\ref{sec:zpapp}, $W_{-2}=0$ and $W_{-1}$ is the unipotent radical of its generic Mumford--Tate group, which is given by $\mathbb{G}_a^{2g \cdot n}$.
\end{rem}

\subsection{Orbit closures as (a)typical intersections, inspired by Filip}\label{diagramatyp}
In this section we give a geometric reinterpretation of the conditions appearing in \Cref{mainthmalgfilip}. Here we slightly depart from the exposition of \cite{filipnotes} (which was already implicit in \cite{zbMATH06575346}), and explain the difference in \Cref{filipfollowing}.

We will think of an orbit closure $\mathcal{N}$ (with associated invariants $r,d,t$ as in \Cref{strsec}) as the reduced subvariety underlying an irreducible component of $(\Omega \varphi)^{-1}(E[\mathcal{N}])$, where $E[\mathcal{N}] \subset \Omega \mathcal{A}_{g,n}$ is an algebraic subvariety constructed using the conditions appearing in \autoref{mainthmalgfilip}. In accordance with the structure of \autoref{mainthmalgfilip}, the variety $E[\mathcal{N}]$ is constructed in three stages, depicted in the following diagram:

\begin{equation}
\label{Filipdiagram}
\begin{tikzcd}
  \mathcal{N} \arrow[r] \arrow[d, hookrightarrow]
  &   E[\mathcal{N}] \arrow[r,  two heads] \arrow[d, hookrightarrow]
   &    M[\mathcal{N}]  \arrow[r,  two heads] \arrow[d, hookrightarrow]
    & S[\mathcal{N}] \arrow[d, hookrightarrow] \\    
  \Omega \mathcal{M}_g(\kappa ) \arrow[r]
&\Omega \mathcal{A_{g,n}} \arrow[r,  two heads]
& \mathcal{A}_{g,n} \arrow[r,  two heads]
&\mathcal{A}_{g}
 \end{tikzcd}.
 \end{equation}

We now give details and definitions regarding the objects appearing in the above diagram. 
\begin{enumerate}
\item Let $S[\mathcal{N} ] \subset \mathcal{A}_g$ be the smallest weakly special subvariety of $\mathcal{A}_g$ containing the image of $\mathcal{N}$. It lies in the locus $R[\mathcal{N}]$ of $g$-dimensional principally polarized abelian varieties that have real multiplication \emph{of the same type as of $\mathcal{N}$} (in the sense of \Cref{mainthmalgfilip}). The latter is a special subvariety of $\mathcal{A}_{g}$ whose generic Mumford-Tate group is isomorphic to $\left(\operatorname{Res}_{K/\Q} {\operatorname{GSp}_{2r,K}} \right)\times \operatorname{GSp}_{2 (g-dr)}$. 
\item Let $M [\mathcal{N}] \subset \mathcal{A}_{g,n}$ be the smallest weakly special subvariety of $\mathcal{A}_{g,n}$ containing $\mathcal{N}$. It lies in $T[\mathcal{N}]$, the bundle over $R [\mathcal{N}]$ consisting of $n$-tuples of points satisfying the same torsion conditions as $\mathcal{N}$.
\item Let $E[\mathcal{N}]$ be the bundle of $1$-forms over $M[\mathcal{N}]$ inside the $K$-eigenspace containing the restriction of the section $\omega$ to $\mathcal{N}$. 
\end{enumerate}
The notation $S[\mathcal{N}]$ is chosen as we think of the variety $S[\mathcal{N}]$ as giving the \emph{Shimura condition}. Similarly, $M[\mathcal{N}]$ gives the \emph{mixed-Shimura condition}, and $E[\mathcal{N}]$ includes also the eigenform condition in addition to the previous two. 

The conditions $R[\mathcal{N}]$ and $T[\mathcal{N}]$ are geometric reinterpretations\footnote{We stress here the fact that the first two conditions of \Cref{mainthmalgfilip} are written in terms of rational Hodge structures (as already noticed and used by Filip). At first sight, the \emph{twist by real multiplication} appearing in the condition on the Abel-Jacobi variety could look like a statement about $K$-Hodge structures, but this becomes a $\Q$-condition after including Galois conjugate conditions.} of the first two items in \autoref{mainthmalgfilip}, whereas the conditions $S[\mathcal{N}]$ and $M[\mathcal{N}]$ may in principle be stronger. The condition $E[\mathcal{N}]$ corresponds to the third condition in \autoref{mainthmalgfilip}. One can think of $M[\mathcal{N}]$ as a variety that records all mixed Hodge-theoretic conditions present on $\mathcal{N}$, and $E[\mathcal{N}]$ as a variety that records, in addition to $M[\mathcal{N}]$, the third condition in \autoref{mainthmalgfilip}.

We now give the \emph{intersection theoretic} (a)typicality condition that will be suitable for our method. The reader interested only in the $\mathcal{M}$ a component of $\Omega \mathcal{M}_g (\kappa)$ case can skip the rest of this section, and just consult equation \eqref{eq123} below.

\begin{defn}[Relative (a)typicality --- Intersection theoretic version]\label{defatyfili}
Let $\mathcal{M}$ be an orbit closure in some stratum $\Omega \mathcal{M}_g (\kappa)$ (possibly equal to component of the latter). Then an orbit closure $\mathcal{N} \subset \mathcal{M}$ is called \emph{(intersection theoretically) atypical} (relative to $\mathcal{M}$) if
\begin{displaymath}
\codim_{E[\mathcal{M}]} \mathcal{N} < \codim_{E[\mathcal{M}]} (E[\mathcal{N}]) +  \codim_{E[\mathcal{M}]} (\mathcal{M}),
\end{displaymath}
and \emph{(intersection theoretically) typical} (relative to $\mathcal{M}$) otherwise.
\end{defn}
The codimensions appearing in \autoref{defatyfili} are formal, but one can make sense of them by taking the image of $\mathcal{N}$ in $\Omega \mathcal{A}_{g,n}$. For most of this section, we will just work with the intersection-theoretic notion of (a)typical.

\begin{rem}
\label{filipfollowing}
In Filip's notation, the following is the key diagram to see orbit closures as intersections:
\begin{center}
\begin{tikzcd}
  \mathcal{N} \arrow[r] \arrow[d, hookrightarrow]
  &    \mathcal{E}\Omega \mathcal{A}_{r,d,t}\arrow[r,  two heads] \arrow[d, hookrightarrow]
   &      \mathcal{E} \mathcal{A}_{r,d,t} \arrow[r,  two heads] \arrow[d, hookrightarrow]
    & \mathcal{E} \mathcal{A}_{r,d} \arrow[d, hookrightarrow] \\
    
  \Omega \mathcal{M}_g(\kappa ) \arrow[r]
&\Omega \mathcal{A_{g,n}} \arrow[r,  two heads]
& \mathcal{A}_{g,n} \arrow[r,  two heads]
&\mathcal{A}_{g}
 \end{tikzcd}.
\end{center}
In the notation above, $\mathcal{E} \mathcal{A}_{r,d} = R[\mathcal{N}]$, $\mathcal{E} \mathcal{A}_{r,d,t} = M[\mathcal{N}]$, and $\mathcal{E} \Omega \mathcal{A}_{r,d,t}$ is nothing more than the eigenform bundle above $\mathcal{E} \mathcal{A}_{r,d,t}$ corresponding to the third condition in \autoref{mainthmalgfilip}; this follows from the monodromy results at the very end of \S\ref{locsysandmonosec}. In the case where $\mathcal{M} = \Omega \mathcal{M}_g (\kappa)$, an orbit closure $\mathcal{N} \subset \mathcal{M}$ is (intersection theoretically) atypical if and only if
\begin{equation}\label{eq123}
\codim_{\Omega \mathcal{A}_{g,n}} \mathcal{N} < \codim_{\Omega \mathcal{A}_{g,n}} (\mathcal{E} \Omega \mathcal{A}_{r,d,t}= E[\mathcal{N}]) +  \codim_{\Omega \mathcal{A}_{g,n}} \Omega \mathcal{M}_g (\kappa) .
\end{equation}
To get back the dimensions appearing in \cite[5.1.3]{filipnotes} one just needs to add to the right-hand side $0=\dim \Omega \mathcal{M}_g(\kappa)-\dim \Omega \mathcal{M}_g(\kappa)$ (in this way $\codim_{\Omega \mathcal{M}_g(\kappa)}\mathcal{N}$ will appear). 
\end{rem}

\begin{prop}[Filip \cite{filipnotes} + $\epsilon$]\label{propcomparison}
Let $\mathcal{N} \subset \mathcal{M}$ be a suborbit closure of $ \mathcal{M}$, with $\mathcal{M}$ a component of $\Omega \mathcal{M}_g(\kappa)$. Then for all but finitely many exceptional orbit closures, the following are equivalent:
\begin{enumerate}
\item $\mathcal{N}$ is dynamically typical relatively to $\mathcal{M}$ (see \Cref{dyanatyp});
\item $\mathcal{N}$ is intersection theoretically typical relatively to $\mathcal{M}$ (see \Cref{defatyfili});
\item $\mathcal{N}$ belongs to the three cases of \cite[Prop. 5.1.6]{filipnotes}: torus covering, Hilbert modular surface, Weierstrass curve.
\end{enumerate}
\end{prop}
\begin{proof}
Filip \cite[\S 5.1]{filipnotes}, by a direct computation, proved that 1 is equivalent to 3. He also proves toward the end of the proof of \cite[Thm. 5.4.5]{filipnotes} that 1. is equivalent to $\delta(\mathcal{N}) = 0$, where $\delta(\mathcal{N})$ is the so-called \emph{degree of atypicality} defined in \cite[\S 5.1.13 (5.1.14)]{filipnotes} and \cite[Def. 5.1.3]{filipnotes}. By explicit computation $\delta(\mathcal{N}) = 0$ if and only if
\[ \codim_{E[\mathcal{M}]} \mathcal{N} = \codim_{E[\mathcal{M}]} (E[\mathcal{N}]) + \codim_{E[\mathcal{M}]} \mathcal{M} \]
where $E[\mathcal{M}] = \Omega \mathcal{M}_g(\kappa)$ and $E[\mathcal{N}]$ is equal to the variety $\mathcal{E} \Omega \mathcal{A}_{r,d,t}$ defined in \cite[\S5.1.10]{filipnotes}. This agrees with our definition of $E[\mathcal{N}]$ for the three cases listed precisely when the monodromy of the torus covering, Hilbert modular surface, or Weierstrass curve is as large as possible. 

That there are at most finitely many exceptions with smaller monodromy follows from the arguments in \cite[\S 5.4.6]{filipnotes} (it is stated only for absolute cohomology, but same argument works for relative cohomology as well). 
\end{proof}
\begin{rem}
The ``Eierlegende Wollmilchsau'' \cite{zbMATH05245768} is a well known example of torus covering $\mathcal{N}$ with \emph{smaller monodromy group than expected}, i.e. an orbit closure such that $M[\mathcal{N}]$ is strictly contained in $T[\mathcal{N}]$.
\end{rem}

\subsection{The Eskin-Filip-Wright finiteness from the intersection-theoretic viewpoint}
We are now ready to prove our main result. That is, to give a new proof, in our unifying setting, of the finiteness parts of \Cref{thmEFWoriginal} (and equivalently \Cref{thmEFWnew}), as described in \Cref{intro:teich}. More precisely, with the vocabulary of \Cref{defatyfili}:

\begin{thm}
\label{orbitclosurefinthm}
Let $\mathcal{M} \subset \Omega \mathcal{M}_g$ be an orbit closure (possibly equal to an irreducible component of some stratum $\Omega \mathcal{M}_g (\kappa)$). Then $\mathcal{M}$ contains at most finitely many (strict) orbit closures $\mathcal{N}$ that are:
\begin{enumerate}
\item Maximal, i.e. the only suborbit closures of $\mathcal{M}$ containing $\mathcal{N}$ are $\mathcal{N}$ and  $\mathcal{M}$;
\item Atypical (relative to $\mathcal{M}$) in the sense of \Cref{defatyfili}.
\end{enumerate}
In particular in each (connected component of each) stratum $\Omega \mathcal{M}_g (\kappa)$, all but finitely many orbit closures have rank 1 and degree at most 2.
\end{thm}

The above gives the first part of \Cref{efffinteich} and the proof provided below is the starting point for the \emph{effective finiteness} we will prove in the final sections. After discussing some preliminaries, in \S\ref{secitonproofefw} we give our proof.

\vspace{0.5em}

For the rest of the section, let $\mathcal{M} \subset \Omega \mathcal{M}_g(\kappa)$ be some fixed orbit closure. We fix a base point $s_{0} := (X, \omega )\in\mathcal{M}$. We write $(\mathbf{H},D^{0} \subset \ch{D}^{0})=(\mathbf{H}_{\mathcal{M}},D^{0}_{\mathcal{M}}\subset \ch{D}^{0}_{\mathcal{M}})$ for the associated weak Shimura (or Hodge) datum associated to $\mathcal{M}$. For example, as recalled above, $\mathbf{H} = \operatorname{Sp}_{2g} \ltimes \mathbb{G}_{\operatorname{a}}^{2g\cdot n}$ if $\mathcal{M}=\Omega \mathcal{M}_g(\kappa)$, where $n = |\kappa| - 1$. 

\subsubsection{Rephrasing the atypical condition in the period torsor}
\label{rephsec}

We now consider $\mathbb{V}, \mathbb{V}_{\operatorname{abs}}, \mathcal{V}, \mathcal{V}_{\operatorname{abs}}$, etc. as in \S\ref{locsysandmonosec}, and consider the bundle $P$ and map $\nu$ constructed in \S\ref{constructionBT}. Then $\omega$ is an algebraic section of $\mathcal{V}_{\textrm{abs}}$, and then we obtain an algebraic map $r : P \to \mathcal{V}_{abs, s_{0}}$ given by 
\[ [\eta \in \Hom(\mathcal{V}_{s}, \mathcal{V}_{s_{0}})] \mapsto \eta(\omega_{s}) . \]

\begin{lem}
\label{mapfromPsurj}
The map $r$ lands inside a unique eigenspace $E_{\mathcal{M}} = E_{\mathcal{M},s_{0}} \subset \mathcal{V}_{\textrm{abs},s_{0}}$ for the field of real multiplication $K_{\mathcal{M}}$ associated to $\mathcal{M}$. Consider the subvariety
\[ Z_{\mathcal{M}} := \{ (v, F^{\bullet}) \in (E_{\mathcal{M}} \setminus \{ 0 \}) \times \ch{D}^{0} : v \in F^{1}_{\textrm{abs}} \} \subset (E_{\mathcal{M}} \setminus \{ 0 \}) \times \ch{D}^{0} . \]
Then the complex algebraic variety $Z_{\mathcal{M}}$ consists of a single $\mathbf{H}(\mathbb{C})$-orbit which factors $r \times \nu$, and the map $r \times \nu : P \to Z_{\mathcal{M}}$ is $\mathbf{H}(\mathbb{C})$-equivariant and surjective.
\end{lem}

\begin{proof}
We recall the construction of the canonical leaf $\mathcal{L} \subset P$ which comes from the underlying integral local system $\mathbb{V}$: we consider the identity map $\textrm{id}_{s_{0}} \in \Hom(\mathbb{V}_{s_{0}}, \mathbb{V}_{s_{0}})$, which lies inside the fibre $P_{s_{0}}$, and then let $\mathcal{L}$ be the unique leaf passing through $\textrm{id}_{s_{0}}$. As explained in \S\ref{sec:RH}, this leaf lies in $P$. Then from the third part of \autoref{mainthmalgfilip}, the image $r(\mathcal{L})$ lies in a $K_{\mathcal{M}}$-eigenspace $E_{\mathcal{M}} = E_{\mathcal{M},s_{0}}$. Because $\mathcal{L}$ is Zariski dense in $P$ as a consequence of \autoref{galequalsmono}, we have $r^{-1}(\overline{r(\mathcal{L})}^{\textrm{Zar}}) = P$ and hence $r(P) \subset E_{\mathcal{M}}$. 

It is clear by construction that the map $r \times \nu$ is $\mathbf{H}(\mathbb{C})$-invariant, so surjectivity will follow if we can show that in fact $Z_{\mathcal{M}}$ is an orbit of $\mathbf{H}(\mathbb{C})$. That $\ch{D}^{0}$ is an orbit of $\mathbf{H}(\mathbb{C})$ is by construction, and the fibre of $Z_{\mathcal{M}} \to \ch{D}^{0}$ over a flag $F^{\bullet}$ is just $(F^{1}_{\textrm{abs}} \cap E_{\mathcal{M}}) \setminus \{ 0 \}$. It therefore suffices to fix a flag $F^{\bullet} \in \ch{D}^{0}$ and show that the stabilizer of $F^{\bullet}$ in $\mathbf{H}(\mathbb{C})$ acts transitively on $(F^{1}_{\textrm{abs}} \cap E_{\mathcal{M}}) \setminus \{ 0 \}$. Without loss of generality, we take $F^{\bullet} = F^{\bullet}_{s_{0}}$ to be the Hodge flag at $s_{0}$. 

Now $\mathbf{H}(\mathbb{C})$ acts on the absolute cohomology $\mathcal{V}^{1}_{\textrm{abs}, s_{0}}$ through its reductive quotient $\mathbf{H}_{\textrm{abs}}$. Note that the stabilizer of $F^{\bullet}_{s_{0}}$ in $\mathbf{H}(\mathbb{C})$ surjects onto the stabilizer of $F^{\bullet}_{\textrm{abs},s_{0}}$ in $\mathbf{H}_{\textrm{abs}}(\mathbb{C})$: it suffices to check the corresponding statement on the level of Lie algebras, where the corresponding map is identified with $F^{0} \mathfrak{h} \to F^{0} (\mathfrak{h}/W_{-1})$ (cf. the proof following \cite[Lem. A.1]{2021arXiv210110938G}), with $\mathfrak{h}$ the Lie algebra of $\mathbf{H}$. But that this map is surjective is immediate from the definition of the mixed Hodge structure on $\mathfrak{h}$, since the Hodge filtration on $(\mathfrak{h}/W_{-1})$ is induced by the one on $\mathfrak{h}$. To determine this latter stabilizer subgroup of $\mathbf{H}_{\textrm{abs}}$, we consider the structure of the local system $\mathbb{V}_{\textrm{abs}}$ obtained as the pure quotient of $\mathbb{V}$. From our discussion in \S\ref{locsysandmonosec}, we have subspaces $\mathcal{U}_{\iota,s_{0}} \subset \mathcal{V}_{\textrm{abs},s_{0}}$ such that $\mathcal{U}_{s_{0}} = \bigoplus_{\iota : K_{\mathcal{M}} \hookrightarrow \mathbb{C}} \mathcal{U}_{\iota,s_{0}}$. The eigenspace $E_{\mathcal{M}}$ is equal to one of the fibres $\mathcal{U}_{\iota,s_{0}}$ for some choice of $\iota$, so it suffices to consider the action of the semisimple factor $\mathbf{H}_{\iota} \subset \mathbf{H}_{\textrm{abs}}$ that acts non-trivially on $\mathcal{U}_{\iota,s_{0}}$. Write $F^{1}_{\iota} := F^{1}_{\textrm{abs}} \cap \mathcal{U}_{\iota,s_{0}}$. 

As explained in \S\ref{locsysandmonosec}, one has $\mathbf{H}_{\iota,\mathbb{C}} = \textbf{Sp}(\mathbb{U}_{\iota,\mathbb{C}})$, where the symplectic form comes from cup product. The stabilizer of $F^{1}_{\iota}$ in $\textbf{Sp}(\mathbb{U}_{\iota,\mathbb{C}})$ acts on $F^{1}_{\iota}$ through the full general linear group $\GL(F^{1}_{\iota})$, hence acts transitively on $F^{1}_{\iota} \setminus \{ 0 \}$. 
\end{proof}

Given a suborbit closure $\mathcal{N} \subset \mathcal{M}$, we will write $Z_{\mathcal{N}} \subset E_{\mathcal{N}} \times \ch{D}^{0}_{\mathcal{N}}$ for the analogous subvariety of $E_{\mathcal{M}} \times \ch{D}^{0}$ associated to $\mathcal{N}$. It is well-defined up to the action of the monodromy group $\Gamma_{\mathcal{M}}$.

\begin{cor}
\label{ineqinsidebundle}
We have $\dim Z_{\mathcal{N}} = \dim E[\mathcal{N}]$. If $\mathcal{N}$ is an intersection-theoretically atypical suborbit closure of $\mathcal{M}$, then we have the following inequality of formal codimensions:
\[ \codim_{P} \mathcal{N} < \codim_{P} (r \times \nu)^{-1}(Z_{\mathcal{N}}) + \codim_{P} \mathcal{M} . \]
\end{cor}

\begin{proof}
The equality $\dim Z_{\mathcal{N}} = \dim E[\mathcal{N}]$ of dimensions is formal. More precisely, the variety $M[\mathcal{N}]$ is by definition the image in $M[\mathcal{M}]$ of $D^{0}_{\mathcal{N}}$ under the uniformization $D^{0}_{\mathcal{M}} \to M[\mathcal{M}]$, and $D^{0}_{\mathcal{N}}$ is open in (and hence has the same dimension as) $\ch{D}^{0}_{\mathcal{M}}$. Then $Z_{\mathcal{N}}$ (resp. $E[\mathcal{N}]$) is a fibre bundle above $\ch{D}^{0}_{\mathcal{N}}$ (resp. $M[\mathcal{N}]$) and both these bundles have the same rank.

For the second inequality we use the fact that $P \to Z_{\mathcal{M}}$ has constant fibre dimension (recall from \autoref{mapfromPsurj} that $P \to Z_{\mathcal{M}}$ is equivariant for the $\mathbf{H}(\mathbb{C})$ action and surjective) to rewrite the inequality in \autoref{defatyfili}. The key fact is that
\begin{align*}
\codim_{P} (r \times \nu)^{-1}(Z_{\mathcal{N}}) &= \dim Z_{\mathcal{M}} -  \dim Z_{\mathcal{N}} \\
&= \dim E[\mathcal{M}] - \dim E[\mathcal{N}] .
\end{align*}
\end{proof}

\subsubsection{Over-parametrization}\label{sectionoverpar}

In this section, we describe a family of subvarieties of $P$ that (over)-parametrizes (recall \autoref{overparamrem}) all the data that can give rise to suborbit closures of $\mathcal{M}$, regardless whether such an orbit closure is typical or atypical. 

\begin{lem}
\label{orbitoverparamlem}
There exists an algebraic family $f : \mathcal{Z} \to \mathcal{Y}$ of subvarieties of $Z_{\mathcal{M}}$ such that all subvarieties $Z_{\mathcal{N}} \subset Z_{\mathcal{M}}$ associated to suborbit closures $\mathcal{N} \subset \mathcal{M}$ arise as a fibre $f^{-1}(y)$ for some $y \in \mathcal{Y}$.
\end{lem}

\begin{proof}
It suffices to handle the two ``pieces'' of the varieties $Z_{\mathcal{N}}$ separately. In particular, since $E_{\mathcal{N}} \subset E_{\mathcal{M}}$ is an inclusion of linear subspaces, it is clear all possible choices for such a subspace are parameterized by a union of Grassmannian varieties. It thus suffices to show that all weakly special subdomains of $\ch{D}^{0}$ belong to a common algebraic family, which is a consequence of \autoref{gaolemma} below.
\end{proof}

Let $(\mathbf{H}_{\mathcal{M}},\ch{D}^0_{\mathcal{M}})$ be the monodromy datum associated to $(\mathcal{M},\V_{| \mathcal{M}})$ (where $\V $ is the standard $\Z$VMHS on some stratum containing $\mathcal{M}$). 

\begin{lem}\label{gaolemma}
There is an algebraic family (over a disconnected base) $f: \mathcal{Z}\to \mathcal{Y}$ of subvarieties of $\ch{D}^0_{\mathcal{M}}$ such that, for every weakly special sub-datum $(\mathbf{H}_i,\ch{D}_i)$, $\ch{D}_i$ appears as a fibre of $f$.
\end{lem}

\begin{proof}
This is a consequence of \cite[Lem. 12.3]{zbMATH06801925} and \cite[Sec. 8.2]{zbMATH07305885}. In particular, it is enough to take as base $\mathcal{Y}$ the disjoint union of finitely many copies of $\mathcal{G} \times \ch{D}^0_{\mathcal{M}}$, where the group $\mathcal{G}$ is defined right above \cite[Lem. 12.3]{zbMATH06801925}. The algebraic family obtained parametrizes orbits at some point $x\in \ch{D}^0_{\mathcal{M}}$ of some $\mathcal{G}(\C)$-translate of a finite set of representatives of weakly special subdomains of $\ch{D}^0_{\mathcal{M}}$.  (In the pure case a stronger result can be found for example in \cite[\S 4.2]{zbMATH06492665}.)
\end{proof} 

\noindent We note that we will prove the same statement in a more general setting in \Cref{prelhodgedataconj} (see especially \Cref{prop2par})

By pulling back the family of subvarieties of $\mathcal{Z}_{\mathcal{M}}$ from \autoref{orbitoverparamlem} along the map $r \times \nu : P \to \mathcal{Z}_{\mathcal{M}}$ from \autoref{mapfromPsurj} we obtain a new family of subvarieties of $P$. By abuse of notation in what follows we denote this family by $f : \mathcal{Y} \to \mathcal{Z}$. We write $f^{(j)} : \mathcal{Y}^{(j)} \to \mathcal{Z}^{(j)}$ for the subfamily where the fibres have dimension $j$.

\subsubsection{Proof of \Cref{orbitclosurefinthm}}\label{secitonproofefw}

 Let $(\mathcal{N}_{i})_{i\in \N}$ be an infinite sequence of intersection theoretically atypical (relative to $\mathcal{M}$) suborbit closures that don't lie in any other suborbit closure. The first input is as in the beginning of \cite[Proof of Thm. 1.5]{zbMATH06890813}: the Zariski closure of their union is a finite union of orbit closures (cf. \Cref{remkzarcl}), so it is enough to show that they are not Zariski dense in $\mathcal{M}$. (Here is where we use the maximality of the $\mathcal{N}_{i}$.) We will later interpret this argument in \S\ref{sec:compoforbitclosures} as the first step in an explicit construction of the Zariski closure of $\bigcup_{i \in \N} \mathcal{N}_{i}$, and hence the finitely many atypical orbit closures themselves.

We may fix the dimension $e$ of the suborbit closures $\mathcal{N}_{i}$ we consider; it suffices to prove the theorem for each $e$ separately. Similarly, we may fix the dimension $j$ of the inverse images $(r \times \nu)^{-1}(Z_{\mathcal{N_{i}}})$. 

Now let us apply \autoref{protogeoZP} to our situation using the family $f^{(j)}$. The hypotheses are satisfied by \autoref{ineqinsidebundle}, which gives the atypicality inequality with $e = j - \dim \mathbf{H}_{\mathcal{M}}$, and \autoref{rigidprop} which produces the desired weakly special families. We therefore obtain finitely many families $\{ h_{i} : C_{i} \to B_{i} \}_{i=1}^{m}$ of strict weakly special subvarieties of $S$ such that each $\mathcal{N}_{i}$ maps into a fibre of one of the $h_{i}$, and each map $C_{i} \to \mathcal{M}$ is quasi-finite. The proof is then completed by the following: 

\begin{lem}
\label{orbitclnofamlem}
Let $h : C \to B$ be an algebraic family of subvarieties of $\mathcal{M}$ such that $\pi : C \to \mathcal{M}$ is quasi-finite. Then only finitely many fibres of $h$ contain a suborbit closure of $\mathcal{M}$.
\end{lem}

\begin{proof}
Let $I = (\mathcal{N}_i)_{i=1}^{\infty}$ be a sequence of orbit closures appearing in the fibres of $h$. Thanks to \Cref{isolation}, after passing to a subsequence there is an orbit closure $\mathcal{T}$ containing the orbit closures associated to $I$ such that the sequence of measures associated to $I$ converges to $\mu_{\mathcal{T}}$. Fix an open analytic neighbourhood $O \subset \mathcal{T}$, chosen small enough so that $\pi^{-1}(O) = O_{1} \sqcup \cdots \sqcup O_{\ell}$ is a disjoint union of analytic neighbourhoods of $\pi^{-1}(\mathcal{T})$ which map bijectively onto $O$. Then by equidistribution $O$ intersects infinitely many orbit closures arising from the fibres of $h$, hence one of the $O_{i}$, say it is $O_{1}$, must satisfy the property that the fibres of $h$ above $h(O_{1})$ contain infinitely many orbit closures.

Applying \Cref{isolation} again to the subsequence of orbit closures in fibres lying above $h(O_{1})$, and pulling back along $C \to \mathcal{M}$, we obtain an algebraic subvariety of $C$ which projects to an algebraically constructible set contained in $O_{1}$. After taking $O_{1}$ small enough this implies that the projection consists of points, so in fact the orbit closures of interest lie in a single fibre.
\end{proof}
The fact that in each stratum $\Omega \mathcal{M}_g (\kappa)$, all but finitely many orbit closures have rank 1 and degree at most 2 follows from \Cref{propcomparison}.

\subsection{Abundance of typical orbit closures}
Let $\mathcal{M}$ be an orbit closure with associated monodromy datum $(\mathbf{H}, D^{0})$, and $S \subset \mathcal{M}$ be a Zariski closed subvariety with the same monodromy group. 

\begin{prop}
\label{denseinSprop}
Let $\mathcal{N}\subset \mathcal{M}$ be a maximal suborbit closure, with associated datum $(\mathbf{H}_{\mathcal{N}}, D^{0}_{\mathcal{N}})$, and $Z_\mathcal{N} \subset H^1_{abs,s_{0}} \times \ch{D}^{0}$
be the subvariety obtained from \S\ref{rephsec}. If $\mathcal{N}\subset \mathcal{M}$ is such that
\begin{equation}\label{typeqab}
\codim_{Z_{\mathcal{M}}} (Z_\mathcal{N})= \codim_{\mathcal{M}}  \mathcal{N}
\end{equation}
 Then orbit closures with the same invariants $(r, d, t)$ as $\mathcal{N}$ are analytically dense in ${\mathcal{M}}$. (In fact, also the totally real number field $K=K_{\mathcal{N}}$ is preserved.)
\end{prop}
The proof of the proposition is quite similar to the one given in Hodge theory for part 2. of \Cref{thmhodgelocus} and in this setting by Eskin-Filip-Wright. In both cases, ultimately, the density comes from the density of $\mathbf{H}(\Q)$ in $\mathbf{H}(\R)$. 

\begin{proof}
The latter part of the proposition follows from the former, thanks to \Cref{isolation}. Write $\widetilde{\mathcal{M}}$ for the preimage of ${\mathcal{M}}$ in $H^1_{abs,s_{0}} \times  \ch{D}^{0}$. The dimension equality in \eqref{typeqab} asserts that $Z_{\mathcal{N}}$ and $\widetilde{S}$ intersect transversely at some smooth point $s \in \mathcal{N}$. In particular, there exists an open neighbourhood $U \subset \mathbf{H}(\R)$ of the identity such that $g \cdot (Z_{\mathcal{N}})$ and $\widetilde{\mathcal{M}}$ continue to intersect transversely near $s$ for all $g \in U$. If we take $g \in \mathbf{H}(\Q) \cap U$, then the translate $g D^{0}_{\mathcal{N}}$ defines mixed Hodge structures whose absolute part has a $\mathbb{Q}$-summand $g J$ with real multiplication by $g K g^{-1}$ and with torsion relations $T g^{-1}$. Moreover, $g E_{\mathcal{N}}$ is a corresponding $g K g^{-1}$-eigenspace. Thus, from \autoref{mainthmalgfilip} it suffices to show that the projection to $\mathcal{M}$ of $g Z_{\mathcal{N}} \cap \widetilde{\mathcal{M}}$ is algebraic. This follows by applying the algebraicity of mixed Hodge loci, or by noticing that, since we are dealing with loci pulled back from the mixed Shimura variety $\mathcal{A}_{g,n}$ and its bundle $\Omega \mathcal{A}_{g,n} \to \mathcal{A}_{g,n}$, each of the imposed conditions is algebraic.
\end{proof}

\section{Proof of Geometric Zilber-Pink for VMHS}\label{new2}

Let $S$ be a smooth quasi-projective variety, and $\V \to S $ be a $\Z$VMHS, with associated Hodge-monodromy datum $(\mathbf{H}_S \subset \mathbf{G}_S, D^0_S \subset D_S)$, as in \Cref{paraspace}. This section proves \Cref{thm:mixedZP} announced in the introduction, and will recover, as a special case, \Cref{thmhodgelocus}. We continue with the vocabulary from \S\ref{sectonperiodmaps}.

\begin{thm}[Geometric mixed ZP]\label{thm:mixedgeomzp}
There is a finite set $\Sigma= \Sigma_{(S,\V)}$ of triples $(\mathbf{H},D_H, \mathbf{N})$, where $(\mathbf{H},D_H)$ is some sub-Hodge datum of the generic Hodge datum $(\mathbf{G}_{S},D_{S})$, $\mathbf{N}$ is a normal subgroup of $\mathbf{H}$ whose reductive part is semisimple, and such that the following property holds.

For each monodromically atypical maximal (among all monodromically atypical subvarieties) $Y \subset S$ there is some $(\mathbf{H},D_H, \mathbf{N})\in \Sigma$ such that, up to the action of $\Gamma$, $D^{0}_{Y}$ is the image of $\mathbf{N}(\R)^+ \mathbf{N}(\C)^{u} \cdot y$, for some $y \in D_H$.
\end{thm}

Often the above theorem, or variants of it, are stated assuming that the $Y$ are of positive period dimension (cf. \Cref{perioddim}). However, since the varieties of period dimension zero belong to a common algebraic family, we may include them in the statement as well. If one includes the positive period dimension hypothesis, one can conclude that the groups $\mathbf{N}$ appearing in the statement are non-trivial.

\subsection{Some preliminaries}\label{prelhodgedataconj}
Let $(S,\V)$ be a pair as above. To ease the notation, we set $\ch{D}= \ch{D}_S^0$, and $\mathbf{G}=\mathbf{G}_S$ .
\begin{lem}
\label{wspdomainoverparamlem}
Consider the set 
\begin{displaymath}
\mathcal{Q} := \{ [\ch{Q}] : \ch{Q} \textrm{ is a weakly special subdomain of }\ch{D} \} .
\end{displaymath}
Then there exists a variety $\mathcal{Y}$, an embedding $\mathcal{Q} \hookrightarrow \mathcal{Y}(\mathbb{C})$, and a family $g : \mathcal{D} \to \mathcal{Y}$ of subvarieties of $\ch{D}$ such that the fibre above the point $[\ch{Q}]$ is the variety $\ch{Q}$.
\end{lem}

\begin{proof} By combining \autoref{weakmtdef} and \autoref{opennessremark} we see that all weakly special subdomains $\ch{Q}$ of $\ch{D}$ are of the following form: there is an algebraic subgroup $\mathbf{N} \subset \mathbf{G}$, itself a semidirect product $\mathbf{N} = \mathbf{N}^{ss} \ltimes \mathbf{N}^{u}$ of a unipotent and semisimple subgroup of $\mathbf{G}$, and a point $t \in \ch{D}$, such that $\ch{Q} = \mathbf{N}(\mathbb{C}) \cdot t$. We claim that in order to overparameterize weakly special subdomains of $\ch{D}$, it suffices to overparameterize both semisimple subgroups of $\mathbf{G}$, and unipotent subgroups of $\mathbf{G}$. Indeed, suppose that we have done this, and let $\mathcal{F}_{\operatorname{uni}} \to \mathcal{Y}_{\operatorname{uni}}$ and $\mathcal{F}_{\operatorname{ss}} \to \mathcal{Y}_{\operatorname{ss}}$ be the corresponding families of unipotent and semisimple groups. Then we can construct a family $\mathcal{F}_{\operatorname{prod}}$ over $\mathcal{Y}_{\operatorname{prod}} = \mathcal{Y}_{\operatorname{uni}} \times \mathcal{Y}_{\operatorname{ss}}$ of constructible algebraic sets which consists of all products $U \cdot S$, where $U$ (resp. $S$) is a unipotent (resp. semisimple) subgroup of $\mathbf{G}$, and $\cdot$ is the product operation of $\mathbf{G}$; note that this is just the image of $U \times S$ under the group operation $\cdot : \mathbf{G} \times \mathbf{G} \to \mathbf{G}$, and is a priori not an algebraic subgroup. However, all algebraic subgroups of $\mathbf{G}$ which are the semidirect product of a semisimple and unipotent subgroup of $\mathbf{G}$, and in particular all the groups $\mathbf{N}$ discussed above, will appear as some fibre of $\mathcal{F}_{\operatorname{prod}} \to \mathcal{Y}_{\operatorname{prod}}$. We now consider the map $\mathcal{F}_{\operatorname{prod}} \times \ch{D} \to \ch{D}$ given by $((g, y), t) \mapsto g \cdot t$ which lifts the action of $\mathbf{G}_{\mathbb{C}}$ on $\ch{D}$. Then all $\ch{Q}$ will appear as the image of some fibre of $\mathcal{F}_{\operatorname{prod}} \times \ch{D} \to \mathcal{Y}_{\operatorname{prod}} \times \ch{D}$.

The general problem reduces to the case where $\mathbf{G} = \GL(V)$ for some finite-dimensional $\mathbb{Q}$-vector space $V$ by choosing an embedding $\mathbf{G} \hookrightarrow \GL(V)$. Then complex semisimple subgroups of reductive algebraic groups belong to finitely many conjugacy classes (see \cite{zbMATH03237938} and \cite[Prop. 0.1]{zbMATH07411859}), which shows the desired overparameterization in this case. For the case of unipotent groups, it suffices, by conjugating, to reduce to subgroups of the group of upper triangular matrices with $1$'s on the diagonal. Since the exponential map is algebraic upon restriction to a nilpotent Lie algebra, this implies that the degrees (with respect to a natural affine embedding $\GL(V) \hookrightarrow \mathbb{A}^{n^2}$ with $n = \dim V$) of unipotent subgroups of $\GL(V)$ can be universally bounded. This implies the desired fact for unipotent groups, since one can overparameterize all subvarieties of $\mathbb{A}^{n^2}$ of bounded degree by an algebraic family (over a possibly disconnected base).
\end{proof}

Given a set $T$ of irreducible subvarieties of a fixed variety $D$, we write $\mathcal{I}(T)$ for the set of all irreducible subvarieties of $D$ obtained by taking irreducible components of arbitrary intersections of elements of $T$.

\begin{lem}
\label{intfamlem}
Suppose that $g : Z \to Y$ is a family of irreducible subvarieties of some projective variety $D$. Let $T$ be the set of its fibres. Then there exists a family $\mathcal{I}(g)$ of subvarieties of $D$ whose fibres are exactly the elements of $\mathcal{I}(T)$.
\end{lem}

\begin{proof}
Fix an ample line bundle on $D$ so that each subvariety of $D$ has a degree. By taking a flattening stratification of $g$ and using the fact that there are only finitely many components of the Hilbert scheme of $D$ parameterizing varieties of bounded degree, one observes that the elements of $T$ have uniformly bounded degree. Thus, this is also true for the elements of $\mathcal{I}(T)$. We can then use the Hilbert scheme to find a family $h : Z' \to Y'$ of irreducible varieties such that all elements of $\mathcal{I}(T)$ are among the fibres of $h$.

Now from $g$, one can construct families $g^{(2)}, g^{(3)}$, etc., where the fibres of $g^{(j)}$ are $j$-fold intersections of fibres of $g$. The desired set of fibres is then those fibres of $h$ which are components of some $g^{(j)}$ for $j \leq \dim D$. The condition on $Y'$ that the fibre lies in some component of some $g^{(j)}$ is constructible, so we may construct $\mathcal{I}(g)$ from $h$ after partitioning the constructible locus in $Y'$ corresponding to $\mathcal{I}(T)$ appropriately.
\end{proof}

We now construct the ``over-parametrization'' that will be used in the proof of the geometric Zilber-Pink conjecture. Let $P$ be the period torsor associated to $(S,\V)$ as in \Cref{constructionBT}, with its algebraic evaluation map $\nu: P \to \ch{D}$, and $\mathcal{W}$ be the set of all varieties obtained as fibres of the family $g$ constructed in \autoref{wspdomainoverparamlem}.

\begin{prop}\label{prop2par}
There exists a family $f : \mathcal{Z} \to \mathcal{Y}$ of subvarieties of the period torsor $P \to S$ such that:
\begin{itemize}
\item[(1)] the fibres of $f$ are exactly the varieties in $\nu^{-1}(M)$ for $M\in \mathcal{I}(\mathcal{W})$; and
\item[(2)] all weakly special subvarieties $Y$ arise as follows: there is a point $y \in \mathcal{Y}(\C)$ such that $Y$ is a component of the projection to $S$ of $\mathcal{Z}_{y} \cap \mathcal{L}$ for some leaf $\mathcal{L}$ above $S$.
\end{itemize}
\end{prop}

\begin{proof}
Note that (2) follows from (1). This is then  achieved by applying \autoref{intfamlem} to the family $g$ constructed in \Cref{wspdomainoverparamlem} and pulling back along $r$ (the space $\ch{D}$ is a smooth projective variety, as explained in \Cref{paraspace}).
\end{proof}

The class of weakly special subdomains of $\ch{D}$ is stable under taking components of intersections, however, because $\mathcal{W}$ also includes varieties that are not weakly special domains, it is not automatic that irreducible components of intersections of elements in $\mathcal{W}$ are again in $\mathcal{W}$. As it is convenient to work with a class of varieties stable under intersection, we work with $\mathcal{I}(\mathcal{W})$ instead of $\mathcal{W}$.

\subsection{Proof of \Cref{thm:mixedgeomzp}}
\label{geozpproofsec}

Thanks to the next lemma, in the course of the proof of the geometric Zilber-Pink conjecture, we will freely translate the group theoretic data $(\mathbf{H},D_H, \mathbf{N})$ into families of weakly special subvarieties which are in turn related to the setting of \autoref{protogeoZP2}.
\begin{lem}\label{lemmapre}
\Cref{thm:mixedgeomzp} is equivalent to proving that monodromically atypical weakly special subvarieties of $S$ belong to finitely many algebraic families, where each fibre of each family is a monodromically atypical weakly special variety.
\end{lem}
In particular, the above shows that \autoref{thm:mixedgeomzp} is equivalent to \autoref{thm:mixedZP}. For simplicity, from now on, we drop the adjective ``monodromically'' in front of words typical and atypical (\Cref{defatyphodge} plays no role in this section, and we are always concerned with \Cref{moatyp}).

\begin{proof}
We consider a family $h : C \to B$ of weakly special subvarieties of $S$, and show that we can find finitely many triples $(\mathbf{H}, D_{H}, \mathbf{N})$ which describe the Hodge-monodromy datum of the fibres of $h$. We may assume that both $B$ and $C$ are irreducible. Then applying \autoref{loctopfiblem} below there are Zariski open subsets $U \subset C$ and $V \subset C$ such that the monodromy groups of the fibres of $U \to V$ are constant, and $U \cap C_{b}$ is Zariski open in $C_{b}$. Since algebraic monodromy groups are determined on any Zariski open, this implies that the algebraic monodromy group $\mathbf{N}$ of the fibres of the family $h^{-1}(V) \to V$ is constant, and arguing by Noetherian induction, we may assume that $C = V$.

Now consider pullback $\restr{\mathbb{V}}{C}$ of $\mathbb{V}$ to a $\mathbb{Z}\textrm{VMHS}$ on $C$. Then $C_{b}$ is associated to the same (weak) Hodge data as a subvariety of $C$ as it is as a subvariety of $S$, so it suffices to prove the claim for this new variation. Now let $(\mathbf{H}, D_{H})$ be the generic Hodge datum of $C$ (which agrees with that of the Zariski closure of its image in $S$), and observe that this is also the generic Hodge datum of $C_{b}$ for $b \in B$ generic. Then by \autoref{monodromytheorem}, $\mathbf{N}$ is normal in the derived subgroup of $\mathbf{H}$, which gives a triple $(\mathbf{H}, D_{H}, \mathbf{N})$ of the desired form. One easily checks that the family of weakly special subvarieties defined by $(\mathbf{H}, D_{H}, \mathbf{N})$ agrees with $h$ from the fact that $(\mathbf{H}, D_{H}, \mathbf{N})$ is the weakly special datum associated to a generic fibre.

\vspace{0.25em}

For the reverse direction, we may argue as in the proof of \autoref{rigidprop} to produce from the data $(\mathbf{H}, D_{H}, \mathbf{N})$ a family $h : C \to B$ of weakly special subvarieties of $S$.
\end{proof}

\begin{lem}
\label{loctopfiblem}
Suppose that $g : U \to V$ is a dominant morphism of irreducible complex algebraic varieties. Then after passing to Zariski open subsets $V' \subset V$ and $U' \subset g^{-1}(V')$, the induced morphism $U' \to V'$ is a fibration for the Euclidean topology.
\end{lem}

\begin{proof}[Proof of \Cref{loctopfiblem}]
If $g$ is proper this follows from generic smoothness combined with Ehresmann's fibration theorem. The general case is proven using Whitney stratifications; see \cite[Cor. 5.1]{MR0481096} and also \cite[Proof of Thm. 3.1.10]{zbMATH06126017}.
\end{proof}

\begin{proof}[Proof of \Cref{thm:mixedgeomzp}]\label{proofimplication}
It suffices to prove the theorem after replacing $S$ with a finite \'etale covering. Then standard techniques in Hodge theory can be used to reduce to the case where the period map $\Psi$ is quasi-finite; see indeed \cite{zbMATH03337393, zbMATH03615066}, as well as \cite{bakker2020quasiprojectivity} and the references therein. We will prove that, for each $e$, the $e$-dimensional atypical weakly special subvarieties of $S$ belong to finitely many algebraic families as in \Cref{lemmapre}.

We let $f$ be as in \autoref{prop2par}, and, for any integer $j$, write $f^{(< j)}$ for the subfamily defined by the property that the varieties $\nu(\mathcal{Z}_{y})$ have dimension $< j$. (This is a priori a constructible condition on $\mathcal{Y}$, but one can make it algebraic by replacing $\mathcal{Y}$ with a finite union of strata.) We say a pair of integers $(e,j)$ is in the \emph{atypical range} if $j < \rho(e) := e + \dim \ch{D}^{0}_{S} - \dim S$. For each $e$ we consider  
\begin{equation}\label{eqYdatyp}
\mathcal{Y}(e):=\mathcal{Y}(f^{(<\rho(e))},e)=\{(x,y)\in P \times \mathcal{Y}: \dim_{x} (\mathcal{L}_{x} \cap \mathcal{Z}_{y}) \geq e, \dim \nu( \mathcal{Z}_{y}) < \rho(e)\}.
\end{equation}
This is the locus in $P \times \mathcal{Y}$ where $e$-dimensional atypical intersections associated to atypical weakly specials can appear, indeed, since the family $f$ comes as pull-back from $\ch{D}^0_S$, the condition $\dim_{x} (\mathcal{L}_{x} \cap \mathcal{Z}_{y}) \geq e$ is equivalent to $\dim_{\nu(x)} \nu(\mathcal{L}_{x} \cap \mathcal{Z}_{y}) \geq e$. Thanks to \Cref{nondenseA} and \Cref{prop2par} this is a constructible set. In fact, we will be interested in two refinements of $\mathcal{Y}(e)$. These will be needed to choose $y$ such that the intersections $\mathcal{L}_{x} \cap \mathcal{Z}_{y}$ are normalized: in particular $\mathcal{L}_{x} \cap \mathcal{Z}_{y}$ should not lie inside a smaller $\mathcal{Z}_{y'}$, and we will want the intersection $\mathcal{L}_{x} \cap \mathcal{Z}_{y}$ to be Zariski dense in $\mathcal{Z}_{y}$. 
\begin{align}
\label{Kcond1}
\mathcal{K}(e)' &= \left\{ (x,y) \in \mathcal{Y}(e) : \begin{array}{c} \textrm{ there exists } (x,y') \in \mathcal{Y}(e') \\ \textrm{ with }e' > e \textrm{ and } \mathcal{Z}_y \subsetneq \mathcal{Z}_{y'} \end{array} \right\},  \\
\label{Kcond2}
\mathcal{K}(e) &= \left\{ (x,y) \in \mathcal{K}(e)' : \begin{array}{c} \textrm{ there exists } (x,y') \in \mathcal{Y}(e) \\ \textrm{ with }\mathcal{Z}_{y'} \subsetneq \mathcal{Z}_y \end{array} \right\}.
\end{align}

\noindent These are once more constructible subsets, thanks to \Cref{nondenseA}.

The following proposition is the key step to prove the geometric Zilber-Pink and crucially uses the Ax-Schanuel theorem.
\begin{prop}
\label{L3}
For each $(x,y)\in \mathcal{Y}(e) \setminus \mathcal{K}(e)$, let $U$ be a component at $x$ of $ \mathcal{Z}_{y} \cap \mathcal{L}_{x}$ of dimension $\geq e$. Then the projection of $U$ to $S$ surjects onto the germ of a strict weakly special subvariety $Y$ of $S$ (of dimension $e$). 
\end{prop}
\begin{rem}
The argument described below is inspired by \cite[Lem. 6.15]{MR3867286}, (the corrected version\footnote{At first sight, the proof appearing in \cite[Prop. 6.6]{2021arXiv210708838B} looks significantly simpler. However some care is required to tackle the induction outlined in \emph{op. cit.}. See also \cite[Sec. 3.3]{2025arXiv250203071B} for an o-minimal proof along the lines we follow here.} of) \cite[Prop. 6.6]{2021arXiv210708838B}, as well as \cite[Prop. 14]{binyamini2021effective}.
\end{rem}

Because all the varieties in our family $f$ arise by pulling back subvarieties of $\ch{D}^{0}_{S}$, it will be convenient for the following arguments to refer to the Ax-Schanuel theorem for intersections with the graph of a period map in $S \times \ch{D}^{0}_{S}$ (namely \Cref{astheoremperioddomain}), although of course all the arguments could instead be rephrased in terms of intersections with leaves of $P$. We have a natural map $(\pi \times \nu) : P \to S \times \ch{D}^{0}_{S}$, where $\pi : P \to S$ is the structure morphism. 

\begin{proof}
Let $(x,y)\in \mathcal{Y}(e) \setminus \mathcal{K}(e)$ and a component $U \subset (\mathcal{Z}_{y} \cap \mathcal{L}_x)$ with $\dim_x U \geq e$. Since $\mathcal{Z}_{y}$ is a fibre of the family $f^{(<\rho(e))}$, the definition of $\rho(e)$ implies that we are in the atypical range, i.e.
\begin{equation}
\label{atypineqL3}
\codim_{S \times \ch{D}^{0}_{S}} (\pi \times \nu)(U) < \codim_{S \times \ch{D}^{0}_{S}} (\pi \times \nu)(\mathcal{L}_x) + \codim_{S \times \ch{D}^{0}_{S}} \nu(\mathcal{Z}_{y}).
\end{equation}
Hence, setting $s=\pi(x)$, if $(Z, s) \subset (S,s)$ is a germ obtained as the image of $U$, then by the Ax-Schanuel theorem \autoref{astheoremperioddomain} the germ $(Z, s)$ lies in a strict weakly special subvariety of $Y$ of $S$. Let $Y$ be the smallest weakly special subvariety of $S$ containing $(Z, s)$, and let $(\mathbf{H}_Y,D^{0}_Y \subset \ch{D}^{0}_Y)$ be its monodromy datum. We fix a Hodge-theoretic embedding of $(\mathbf{H}_Y,D^{0}_Y\subset \ch{D}^{0}_Y)$ into $(\mathbf{H}_S, D^{0}_S\subset \ch {D}^{0}_S)$ and after replacing $x$ and $\mathcal{Z}_{y}$ by an $\mathbf{H}_{S}(\mathbb{C})$-translate we may assume that the orbit $\ch{D}' := \mathbf{H}_{Y}(\mathbb{C}) \cdot \nu(x)$ agrees with $\ch{D}^{0}_{Y}$. We want to show that $\nu(\mathcal{Z}_y)=\ch{D}'$.

\begin{lem}
\label{L5}
The image $\nu(\mathcal{Z}_{y}) \subset \ch{D}^{0}_{S}$ lies in $\ch{D}'$.
\end{lem}

\begin{proof}
Consider an irreducible component $C'$ of $\nu(\mathcal{Z}_{y}) \cap \ch{D}'$ containing the image of $U$, which must have dimension $\leq j := \dim \nu(\mathcal{Z}_{y})$. To prove the lemma, it is enough to show that we must have $\dim C'= j$. Note that by construction (cf. \Cref{intfamlem}), $C'$ is again the image of a fibre of $f$, so we can write $C'=\nu( \mathcal{Z}_{y'})$, for some $y' \in \mathcal{Y}$. Since $\mathcal{Z}_{y'}$ contains the projection of $U$; in particular $(x,y')\in \mathcal{Y}(e)$ and $\mathcal{Z}_{y'} \subseteq \mathcal{Z}_{y}$. The last inclusion is not an equality precisely when $\dim C'<j$. But points where this happens are all contained in $\mathcal{K}(e)$ as a consequence of the second condition (\ref{Kcond2}) defining $\mathcal{K}(e)$. \end{proof}

\noindent Suppose now instead that we have a strict inclusion $\nu(\mathcal{Z}_{y})\subsetneq \ch{D}'$, which is equivalent to the fact that $ \dim  \ch{D}'-j >0$. A priori $Y$ is either an atypical or a typical weakly special subvariety of $(S,\V)$. We argue each case separately and in both cases we will derive the desired contradiction: for the former we contradict (\ref{Kcond1}), and for the latter the minimality of $Y$.

If $Y$ is atypical, we must have $\codim_S Y < \codim_{\ch{D}_S^0} \ch{D}'$, and so $(\dim Y, \dim \ch{D}')$ are in the atypical range. Since we are still assuming that $\nu(\mathcal{Z}_{y})\subsetneq \ch{D}'$, this implies that $e':= \dim Y > e$. By construction $Y$ induces a point $(x,y')\in \mathcal{Y}(e')$ such that $\mathcal{Z}_y \subsetneq \mathcal{Z}_{y'}$. But this is not possible
because such a pair $(x,y')$ was removed in (\ref{Kcond1}).

We might therefore assume that $Y$ is typical. This means that $\textrm{codim}_{S} Y = \textrm{codim}_{\ch{D}^{0}_{S}} \ch{D}'$. We consider the same situation for the pair $(Y, \restr{\mathbb{V}}{Y})$ and the associated period map\footnote{Strictly speaking one should replace $Y$ with a finite covering of a smooth resolution to make sense of this definition, but as this does not change the dimension of $Y$ and can only increase the dimension of $U$ the reasoning is unchanged by such a replacement. } $\Psi: Y \to \Gamma_Y \backslash D^{0}_Y$. Thanks to \Cref{L5}, we can consider the subvariety $Y\times \nu (\mathcal{Z}_{y})\subset Y \times \ch {D}_Y$. Then the projection of $U$ to $Y \times \ch{D}^{0}_{Y}$ is necessarily equal to some analytic component of the intersection
\begin{displaymath}
\left(Y \times \nu (\mathcal{Z}_{y}) \right)\cap \left(Y \times_{\Gamma_Y \backslash D_Y} D_Y\right).
\end{displaymath}
Since again $U$ is an atypical intersection of $(Y,\V_Y)$, as shown in \Cref{intermlem} below, we can apply \Cref{astheoremperioddomain} to find a strict weakly special subvariety $Y'$ whose germ contains the projection of $U$ to $S$. This contradicts the minimality of $Y$, and produces the contradiction we were aiming for.

\begin{lem}\label{intermlem}
We have
\begin{equation}
\codim_{Y\times \ch{D}^{0}_Y} (\pi \times \nu)(U) <  \codim_{Y\times \ch{D}^0_Y}Y\times \nu (\mathcal{Z}_{y}) + \codim_{Y\times \ch{D}^0_Y}( Y \times_{\Gamma_Y \backslash D_Y} D_Y).
\end{equation}
Or equivalently, $\dim Y - e < \dim \ch{D}^{0}_{Y} - j$.

\end{lem}
\begin{proof}
Recall that we started with $e,j$ in the atypical intersection range, i.e. $\dim S - e < \dim \ch{D}^{0}_S-j$. We add on both sides $- \codim_S Y$, to obtain
\begin{displaymath}
\dim Y - e < \dim \ch{D}^{0}_S -j -\codim_S Y.
\end{displaymath}
Since, by the typicality assumption on $Y$, we know that $\codim_{S} Y = \codim_{\ch{D}^{0}_{S}} \ch{D}^{0}_{Y}$, the RHS becomes $\dim \ch{D}^{0}_{Y} - j$, as desired.
\end{proof}

We have shown that $\nu(\mathcal{Z}_y)=\ch{D}'$, which implies that $U$ is the germ of the weakly special subvariety $Y$, concluding the proof of \Cref{L3}.
\end{proof}

Thus, by \autoref{L3}, each germ $U$ arising in the definition $\mathcal{Y}(e) \setminus \mathcal{K}(e)$ surjects onto the germ of a weakly special subvariety of $S$ of dimension $e$. The next lemma shows that, in fact, the locus $\mathcal{Y}(e) \setminus \mathcal{K}(e)$ contains all germs of maximal atypical subvarieties.

\begin{lem}
\label{L10}
Each maximal (among atypical) atypical weakly special $Y$ of dimension $e$ induces a point of $\mathcal{Y}(e) \setminus \mathcal{K}(e)$.
\end{lem}

\begin{proof}
As observed after \eqref{eqYdatyp}, it is true by construction that each smooth germ of a $e$-dimensional atypical subvariety $Y$ induces a point $(x,y)=(x_Y,y_Y) \in \mathcal{Y}(e)$, where $\nu(\mathcal{Z}_y)=\mathbf{H}_Y(\C)\cdot \nu(x)=:\ch{D}'$ and $\dim_{x} (\mathcal{L}_{x} \cap \mathcal{Z}_{y}) = e$.

We need to check that not all such points are contained in $\mathcal{K}(e)$:
\begin{itemize}
\item If there exists $(x,y') \in \mathcal{Y}(e') $ with $e' > e$ and $ \ch{D}_Y\subsetneq \nu( \mathcal{Z}_{y'} )$, we would like to conclude that $Y$ is not maximal among atypical. We can assume that $y'$ is chosen in such a way that $ \mathcal{Z}_{y'} $ is minimal among the fibres of $f^{(<\rho (e'))}$ containing $\ch{D}_Y$. Consider $W$ the Zariski closure in $S$ of the projection of $\mathcal{L}_x \cap \mathcal{Z}_{y'} $. Since we are still in the atypical range for $e'$, by the Ax-Schanuel theorem, $W$ lies in some strict weakly special subvariety $Y_W$, and we must have $Y\subset W \subset Y_W$. We claim that the smallest such $Y_W$ has to be atypical. Let ($\mathbf{H}_W,\ch{D}_W)$ be its monodromy datum. Consider $\nu(\mathcal{Z}_{y'})\cap  \ch{D}_W$, by the minimality assumption of $\mathcal{Z}_{y'}$ and the stability under intersections of the family, we must have $  \nu(\mathcal{Z}_{y'}) \subset \ch{D}_W$. If we have equality, we conclude that $Y$ is atypical. Assume therefore that $  \nu(\mathcal{Z}_{y'}) \subsetneq \ch{D}_W$. If $Y_W$ is typical (i.e. $\codim_S Y_W= \codim_{\ch{D}_S^0}\ch{D}_W$), then $Y\times \nu(\mathcal{Z}_{y'}) \subset Y_W \times \ch{D}_W$ gives rise to an atypical intersection (\Cref{intermlem}), and this contradicts the minimality of $Y_W$.
\item If there exists $ (x,y') \in \mathcal{Y}(e) $ with $\nu(\mathcal{Z}_{y'}) \subset \ch{D}' $, then $\dim_{x} (\mathcal{L}_{x} \cap \mathcal{Z}_{y'}) = e$ and therefore the projection of $\mathcal{L}_{x} \cap \mathcal{Z}_{y'}$ to $S$ is equal to $Y$. This implies that $\mathcal{Z}_{y'}$ is invariant under the monodromy of $Y$, and therefore that $\nu(\mathcal{Z}_{y'}) = \ch{D}' $.
\end{itemize}
\end{proof}

Now we consider a countable collection $\{ h_{i} : C_{i} \to B_{i} \}^{\infty}_{i=1}$ of families of weakly special subvarieties, each of whose fibres are atypical and $e$-dimensional. One can construct this by refining the families in \autoref{rigidprop} so that all the fibres have dimension $e$, and monodromy data small enough to be in the atypical range. What we proved so far gives $\mathcal{Y}( e) \setminus \mathcal{K}(e) = \bigcup_{i=1}^{\infty} \mathcal{Y}(f^{(< \rho(e))}, e, h_{i})$, and thanks to \autoref{L10} all weakly special subvarieties induce points of this set. Applying either \S\ref{axschanfamsec}, or directly the constructibility lemma \autoref{lemma432}, one sees that in fact finitely many $h_{i}$ suffice, which gives the result.
\end{proof}

\section{Digression on differentially-defined intersections}
\label{yfdapproxsec}

In this section we give some further details about the loci $\mathcal{Z}(f,e)$ appearing in the proof of \autoref{nondenseA}, which will be important for our applications to computing orbit closures. We view $\mathcal{Z}$, and hence $\mathcal{Z}(f,e)$, inside the product $P \times \mathcal{Y}$, and therefore denote points of $\mathcal{Z}(f,e)$ by pairs $(x, y) \in P \times \mathcal{Y}$. Note that we have a natural identification $\mathcal{L}_{(x,y)} \cap f^{-1}(y) = \mathcal{L}_{x} \cap f^{-1}(y)$ of complex analytic germs, where ``$\mathcal{L}_{(x,y)}$'' denotes the leaf in $Q$ at $(x,y)$ and ``$\mathcal{L}_{x}$'' analogously denotes the leaf in $P$ at $x$. 

Throughout $P$ is the torsor of \autoref{examplebundle} associated to an integral variation of mixed Hodge structure $\mathbb{V}$. 

\begin{prop}
\label{intbecometypprop}
Fix a component $R \subset \mathcal{Z}(f,e)$, and let $T$ be the Zariski closure of its image in $S$. 
Then there exists a normal algebraic subgroup $\mathbf{N}_{R} \subset \mathbf{H}_{T}\subset \GL(\mathcal{V}_{s_{0}})$, defined over $\mathbb{Q}$ with respect to the underlying $\mathbb{Q}$-local system, with the following properties: 
\begin{itemize}
\item[(i)] every irreducible analytic germ $U \subset \mathcal{L}_{x} \cap f^{-1}(y)$ of dimension $e$ corresponding to $(x,y) \in R$ lies inside an $\mathbf{N}_{R}$ sub-torsor $P_{(x,y)} \subset \restr{P}{T}$ lying over a weakly special subvariety $Y_{(x,y)} \subset T$ with algebraic monodromy group contained in $\mathbf{N}_{R}$; and
\item[(ii)] for $x,y, P_{(x,y)}$ as in (i) one has 
\begin{equation}
\label{thingsbecometypineq}
\dim_{x} (f^{-1}(y) \cap P_{(x,y)}) \geq \dim P_{(x,y)} - \dim_{x} (\mathcal{L}_{x} \cap P_{(x,y)}) + e .
\end{equation}
\end{itemize}
Moreover the varieties $Y_{(x,y)}$ can be put into a common family $g : C \to B$ of subvarieties of $T$, with $B$ possibly disconnected.
\end{prop}

\begin{proof}
Let $N \subset \mathbf{H}_{T}$ be a $\mathbb{Q}$-normal subgroup. Then arguing as in the proof of \autoref{rigidprop} we may construct, from the quotient Hodge datum $(\mathbf{H}_{T}, D_{T})/N$, a finite cover $T_{N} \to T$ and a period map $\Xi_{N} : T_{N} \to I_{N}$ such that the irreducible components of the fibres of $\Xi_{T}$ map to weakly special subvarieties of $T$ with algebraic monodromy group contained in $N$.

Considering the pullback $(Q_{N}, \delta_{N})$ of $(Q, \delta)$ to $T_{N} \times \mathcal{Y}$, with pullback map $q_{N} : Q_{N} \to Q$ and projection $\pi_{N} : Q_{N} \to T_{N}$, we obtain a Zariski closed locus $\mathcal{Z}(f \cap \Xi_{N}, e) \subset Q_{N}$: this is the locus of points $z \in Q_{N}$ such that 
\[ \dim_{z} [\mathcal{L}_{z} \cap f^{-1}(y) \cap (\Xi_{N} \circ \pi_{N})^{-1}(\Xi_{N} \circ \pi_{N}(z))] \geq e . \]
Then because $Q_{N} \to Q$ is finite, the image $\mathcal{Z}_{N}(f,e) := q_{N}(\mathcal{Z}(f \cap \Xi_{N}, e))$ is Zariski closed in $Q$.

Observe that $R \subset \bigcup_{N} \mathcal{Z}_{N}(f,e)$: indeed, the group $N = \mathbf{H}_{T}$ occurs, in which case one has $\mathcal{Z}_{N}(f,e) = \pi^{-1}(T) \cap \mathcal{Z}(f,e)$, where we write $\pi : Q \to S$ for the natural projection. Let $\mathcal{C}$ be the set of those $N \subset \mathbf{H}_{T}$ such that $R \subset \mathcal{Z}_{N}(f,e)$, and take $\mathcal{C}^{c}$ to be its complement in the set of all normal $\mathbb{Q}$-subgroups of $\mathbf{H}_{T}$. Let $A = \pi^{-1}(\HL(T, \restr{\V^\otimes}{T})) \cup (\bigcup_{N \in \mathcal{C}^{c}} \mathcal{Z}_{N}(f,e))$, which is a countable union of closed loci which intersect $R$ properly. Let $R^{\textrm{irr}} \subset R$ be the open locus of $(x,y)$ where the intersection $\mathcal{L}_{(x,y)} \cap f^{-1}(y)$ is irreducible. Then if we consider $(x_{0}, y_{0}) \in R^{\textrm{irr}} \setminus A$ the Zariski closure of the projection of $\mathcal{L}_{(x_{0}, y_{0})} \cap f^{-1}(y_{0})$ to $T$ has algebraic monodromy group $N \in \mathcal{C}$. Moreover by the irreducibility of $R$ (and hence $R^{\textrm{irr}}$) this $N$ is the same group that occurs for any point $(x_{0}, y_{0}) \in R^{\textrm{irr}} \setminus A$: if we had some other point $(x'_{0}, y'_{0}) \in R \setminus A$ for which $\mathcal{L}_{(x'_{0}, y'_{0})} \cap f^{-1}(y)$ mapped into a weakly special subvariety with smaller algebraic monodromy group $N'$ with $\dim N' < \dim N$, then since $N' \in \mathcal{C}$ we would find that $(x_{0}, y_{0}) \in \mathcal{Z}_{N'}(f,e)$ as well. For this $N$ we set $\mathbf{N}_{R} = N$.

Property (i) holds by construction. More specifically the definition of $\mathcal{Z}_{N}(f,e)$ dictates that the intersection $U = \mathcal{L}_{x} \cap f^{-1}(y)$ lies, after being pulled back to an analytic germ $U' \subset P_{N} = \restr{P}{T}$ above $T_{N}$, inside a fibre of $P_{N} \to I_{N}$. As the irreducible components $Y$ of fibres of $T_{N} \to I_{N}$ have algebraic monodromy group contained in $\mathbf{N}_{R}$, leaves above such components lie inside an $\mathbf{N}_{R}$-torsor in $P_{N}$ above $Y$. 

To verify (ii) we start by observing that the right-hand side of (\ref{thingsbecometypineq}) does not vary with $(x,y)$. Indeed, if the torsor $P_{(x,y)}$ is an $N$-torsor above $Y_{(x,y)}$, then $\mathcal{L}_{x} \cap P_{(x,y)}$ is just a leaf of $Y_{(x,y)}$, hence we have 
\[ \dim P_{(x,y)} - \dim_{x}(\mathcal{L}_{x} \cap P_{(x,y)}) = (\dim N + \dim Y_{(x,y)}) - \dim Y_{(x,y)} = \dim N . \]
It therefore suffices to show $\dim_{x} (f^{-1}(y) \cap P_{(x,y)}) \geq \dim N + e$ when the left-hand side is minimized, which occurs generically on the irreducible component $R$ by semicontinuity of dimension. 

We therefore verify the inequality at a point $(x,y) \in R_{\textrm{irr}} \setminus A$. In this case $Y_{(x,y)}$ has algebraic monodromy group exactly $\mathbf{N}_{R}$, and $\mathcal{L}_{x} \cap f^{-1}(y)$ does not map into a further weakly special subvariety of $Y_{(x,y)}$. Then \autoref{cor:asloc} applied to $Y_{(x,y)}$ and $P_{(x,y)}$ gives
\[ \codim_{x, P_{(x,y)}} (\mathcal{L}_{x} \cap f^{-1}(y)) \geq \codim_{x, P_{(x,y)}} (P_{(x,y)} \cap f^{-1}(y)) + \codim_{P_{(x,y)}} (\mathcal{L}_{x} \cap P_{(x,y)}) . \]
Applying the fact that $\dim_{x} (\mathcal{L}_{x} \cap f^{-1}(y)) \geq e$ this simplifies to
\begin{align*}
\dim P_{(x,y)} - e &\geq \dim P_{(x,y)} - \dim_{x} (P_{(x,y)} \cap f^{-1}(y)) + \dim N \\
\dim_{x} (P_{(x,y)} \cap f^{-1}(y)) &\geq \dim \mathbf{N}_{R} + e .
\end{align*}
\end{proof}

\section{Algorithmic setup}\label{finalsection:eff}

In this section we explain how the determination of the maximal atypical orbit closures (\Cref{efffinteich}) and maximal totally geodesic subvarieties of non-arithmetic ball quotients (\Cref{cordm}) can be made effective. In what follows by the term \emph{algorithm} we always understand \emph{terminating algorithm}.

\subsection{Computational models}
\label{compmodelsec}

First we say a word about the computational models we use, following \cite[\S2]{urbanik2021sets}. We will assume that all fields we work over have computable field operations; for applications it will even suffice to assume all fields are number fields. An affine variety $V$ over a field $K$ is then a finitely-presented $K$-algebra, which we represent computationally through explicit polynomials $f_{1}, \hdots, f_{j}$ in a polynomial ring $K[x_{1}, \hdots, x_{N}]$ such that the ideal $(f_{1}, \hdots, f_{j})$ defines the variety of interest. For modelling a more general variety $V$, we use a finite collection $\{ V_{i} \}_{i=1}^{n}$ of affine varieties with $V_{i} = \Spec R_{i} \subset \Spec K[x_{i1}, \hdots, x_{in_{i}}] =: \mathbb{A}^{n_{i}}$, and additionally assume that, for each pair $(i_{1}, i_{2})$, we have polynomials $g_{i_{1} i_{2}}$ in the coordinate ring of $\mathbb{A}^{n_{i_{1}}} \times \mathbb{A}^{n_{i_{2}}}$ which define the intersection of the diagonal $\Delta \subset V \times V$ with $V_{i_{1}} \times V_{i_{2}}$. Note that this gives us a natural finite presentation of the intersections $V_{i_{1}} \cap V_{i_{2}} = (V_{i_{1}} \times V_{i_{2}}) \cap \Delta$.

We will model $\mathbb{P}^{N} = \textrm{Proj}\, [x_{0}, \hdots, x_{N}]$ in the usual way with affine opens $\{ \{ x_{i} \neq 0 \} \}_{i=0}^{n}$. A morphism of varieties $f : V \to W$ is then represented computationally using:
\begin{itemize}
\item[(i)] covers $\{ V_{i} \}_{i=1}^{n}$ and $\{ W_{j} \}_{j=1}^{m}$ such that for each $j$, we have $f^{-1}(W_{j}) = \bigcup_{i \in I_{j}} V_{i}$ for some $I_{j} \subset \{ 1, \hdots, n \}$; and
\item[(ii)] morphisms $f_{ij} : V_{i} \to W_{j}$ for each $i \in I_{j}$ with $I_{j}$ as above.
\end{itemize}
To compose morphisms one in general has to refine the representing covers.

A coherent sheaf $\mathcal{F}$ on $V$ is given by module data compatible with a fixed cover $\{ V_{i} \}_{i=1}^{n}$. In particular, it consists of:
\begin{itemize}
\item a finitely-presented $R_{i}$-module $M_{i}$ for each $V_{i} = \Spec R_{i}$ in the cover;
\item for each intersection $V_{i} \cap V_{j} = \Spec R_{ij}$, a finitely presented $R_{ij}$-module $N_{ij}$, together with a restriction morphism $\res_{ij} : M_{i} \to N_{ij}$ of $R_{i}$-modules;
\item for all $1 \leq i, j \leq n$ a map $\xi_{ij} : N_{ij} \xrightarrow{\sim} N_{ji}$ of $R_{ij}$-modules such that the collection $\{ \xi_{ij} \}_{1 \leq i, j \leq n}$ defines a gluing datum.
\end{itemize}
All the morphisms above are encoded using the images of the generators coming from the fixed presentations. In the case where $\mathcal{F}$ is a vector bundle of rank $r$ we will always assume that the modules $M_{i}$ and $N_{ij}$ are free of rank $r$ and come with an explicit basis. A morphism of coherent sheaves can then be represented via morphisms of the underlying modules with the obvious compatibilities. 

Given a model for $V$ we have a standard model for $\Omega^{1}_{V}$ whose constituent modules $M_{i}$ are quotients of the standard differential module $\Omega^{1}_{\mathbb{A}^{n_{i}}}$ by $df_{1}, \hdots, df_{j}$, where $(f_{1}, \hdots, f_{j})$ is the ideal defining $V_{i}$. Given two coherent sheaves $\mathcal{F}, \mathcal{V}$ modelled as above, we also obtain models for the tensor and hom-sheaf constructions associated to these sheaves; see \cite[Prop. 2.1]{urbanik2021sets}.

Given a locally free sheaf $\mathcal{F}$ over $V$, there is a natural algebraic vector bundle $\mathbb{B}(\mathcal{F}) \to V$ whose sheaf of sections recovers $\mathcal{F}$ up to isomorphism. We will wish to use both perspectives in a constructive way, for which we require the following simple fact.

\begin{lem}
\label{constrcatequiv}
The functor $\mathbb{B} : \textrm{LocFree}(V) \to \textrm{VecBun}(V)$ from the category of locally free sheaves on $V$ to the category of vector bundles on $V$ is computable. More precisely, after fixing a computational model $( \{ \Spec R_{i} \}_{i=1}^{n}, \{ g_{ij } \}_{i,j=1}^{n} )$ for $V$ as above, there exists
\begin{itemize}
\item[(i)] an algorithm which, given a computational model $(\{ M_{i} \}_{i=1}^{n}, \{ N_{ij} \}_{i,j=1}^{n}, \{ \xi_{ij} \}_{i,j=1}^{n} )$ for $\mathcal{F}$ compatible with the one for $V$, outputs a model $(\{ F_{i} \}_{i=1}^{n}, \{ h_{ij} \}_{i=1}^{n})$ for $\mathbb{B}(\mathcal{F})$, a model $\{ \pi_{i} : F_{i} \to V_{i} \}_{i=1}^{n}$ for $\pi : \mathbb{B}(\mathcal{F}) \to V$, and models for the maps $(+),(-),(\cdot)$ of bundles corresponding to addition, inversion and scalar multiplication; and 
\item[(ii)] an algorithm which, given in addition a computational model 
\begin{displaymath}
(\{ M'_{i} \}_{i=1}^{n}, \{ N'_{ij} \}_{i,j=1}^{n}, \{ \xi'_{ij} \}_{i,j=1}^{n} )
\end{displaymath}
 for $\mathcal{F}'$ and a model $( \{ s_{i} : M_{i} \to M'_{i} \}_{i=1}^{n}, \{ t_{ij} : N_{ij} \to N'_{ij} \}_{i,j=1}^{n} )$ for a morphism $\mathcal{F} \to \mathcal{F}'$, outputs a model for the map $\mathbb{B}(\mathcal{F}) \to \mathbb{B}(\mathcal{F}')$ compatible with the models produced by (i).
\end{itemize}
\end{lem}

\begin{proof}
~ \begin{itemize}
\item[(i)] The standard construction of $\mathbb{B}(\mathcal{F})$ as in \cite[\href{https://stacks.math.columbia.edu/tag/01M1}{Section 01M1}]{stacks-project} associates to the affine open $\Spec R_{i}$ the symmetric algebra $\textrm{Sym}\, M_{i}$; because $\mathcal{F}$ is locally free this can be presented as a polynomial ring $P_{i}$ over $R_{i}$ whose generators are given by some basis $m_{i1}, \hdots, m_{ir}$ for $M_{i}$. The diagonal ideal in $\Spec P_{i} \times \Spec P_{j}$ is then generated by the elements $\{ g_{ij} \}^{n}_{i,j=1}$ and the elements $m_{ik} - \res_{ji}^{-1}(\xi_{ij}(\res_{ij}(m_{ik}))$ for $1 \leq k \leq r$, so it suffices to check that we can invert the map $\textrm{res}_{ji}$. This is just a matter of inverting a matrix over $R_{ji}$, which can be done using Cramer's rule.

That one has a model for the projection is immediate as we may send the fixed generators for $R_{i}$ coming from the presentation to the same elements of $P_{i}$.

The maps $(+) : \mathbb{B}(\mathcal{F}) \times_{S} \mathbb{B}(\mathcal{F}) \to \mathbb{B}(\mathcal{F})$, $(-) : \mathbb{B}(\mathcal{F}) \to \mathbb{B}(\mathcal{F})$ and $(\cdot) : \mathbb{B}(\mathcal{O}_{S}) \times_{S} \mathbb{B}(\mathcal{F}) \to \mathbb{B}(\mathcal{F})$ are constructed in the natural way; we give just the argument for $(+)$. In this case the construction above associates to $M_{i} \oplus M_{i}$ symmetric algebra generated by $\{ m^{(1)}_{ik} \}_{k=1}^{r} \sqcup \{ m^{(2)}_{ik} \}_{k = 1}^{r}$ to the affine set $V_{i}$, with an analogous construction for the diagonal ideals. Then the morphism is just defined on generators by fixing the generators of $R_{i}$ and the usual map $m_{ik} \mapsto m^{(1)}_{ik} + m^{(2)}_{ik}$.

\item[(ii)] The morphisms of modules $s_{i}$ induce natural maps of the associated symmetric algebras and given that our presentation of the modules $\Spec P_{i}$ (resp. $\Spec P'_{i}$) is relative to a basis for $M_{i}$ (resp. $M'_{i}$) it is easily seen that this gives a morphism $\mathbb{B}(\mathcal{F}) \to \mathbb{B}(\mathcal{F}')$ by, for each $i$, fixing the generators of $R_{i}$ and acting with the maps $s_{i}$ on the basis elements.
\end{itemize}
\end{proof}

\begin{rem}
It is less obvious how one computes the inverse of this functor; that is, how given an algebraic map $\pi : \mathcal{B} \to V$ and maps $(+), (-)$ and $(\cdot)$ one computes the associated trivializations. As we will want to use both perspectives in our computations, it will always be important for us that we start with objects on the locally free sheaf side before passing to the bundle picture.
\end{rem}

Finally let us mention that we can represent constructible subsets of an algebraic variety $V$ as a union $\bigcup_{i=1}^{n} C_{i} \cap D_{i}$ of finitely many intersections between an open and closed subvariety; more details are given in \cite[\S2]{urbanik2021sets}. Then preimages and images of constructible sets are also computable, as detailed in \cite[Prop. 2.1]{urbanik2021sets}.

\subsection{Computing algebraic de Rham data}
\label{computedRdatasec}

In this section we explain that the locally free sheaf and connection data that we will need to implement our algorithms can be computed. The first result is the following, proven in \cite{urbanik2021sets} (see also \cite{2022arXiv220611389U, zbMATH07871772} for related works).

\begin{prop}
\label{compconnecprop}
Suppose we are given $f : X \to S$ a smooth projective morphism of relative dimension $m$ together with a fixed closed embedding $\iota: X \hookrightarrow \mathbb{P}^{N}_{S}$ over $S$. Then for any $k$ there is an algorithm which takes as input models for the tuple $(S, X, f, \iota)$, compatible with the standard model for $\mathbb{P}^{n}$, and outputs:
\begin{itemize}
\item[(i)] a model for $S$ as a union of affine opens $S_{i} = \Spec R_{i}$, refining the one given as input; and
\item[(ii)] for each $i$, models for the triple $(\mathcal{V}_{i}, F^{\bullet}_{i}, \nabla_{i})$ of locally free sheaf data on $S_{i}$, where
\begin{align*}
\mathcal{V}_{i} &= R^{k} f_{*} \Omega^{\bullet}_{X_{S_{i}}/S_{i}} \\
F^{j}_{i} &= R^{k} f_{*} [0 \to \cdots \to \Omega^{j}_{X_{S_{i}}/S_{i}} \to \Omega^{j+1}_{X_{S_{i}}/S_{i}} \to \cdots \to \Omega^{m}_{X_{S_{i}}/S_{i}} ] ,
\end{align*}
and where $\nabla_{i} : \mathcal{V}_{i} \to \mathcal{V}_{i} \otimes \Omega^{1}_{S_{i}}$ is the Gauss-Manin connection, as defined in \cite{katzoda}. 
\end{itemize}
\end{prop}

\begin{proof}
This is \cite[Thm 2.4]{urbanik2021sets}. That the model in (i) refines the one given as input is clear from the proof.
\end{proof}

\noindent We wish to explain how the ideas used in the proof of \cite[Thm 2.4]{urbanik2021sets} generalize to the mixed variation that underlies the orbit closure setting; in fact, because one does not have to deal with resolving de Rham complexes associated to higher-dimensional projective varieties, the computations will be simpler in this case. We therefore instead assume that $f : X \to S$ is a family of relative curves, $\iota : X \hookrightarrow \mathbb{P}^{n}_{S}$ is a fixed projective embedding, and that we have in addition a relative divisor $\zeta: Z \hookrightarrow X$ over $S$. We also assume that we have a fixed section $\omega$ of $\Omega^{1}_{X/S}$ such that $Z_{s}$ is the divisor of zeros of $\omega_{s}$.

To describe the relative de Rham cohomology bundle $\mathcal{V}$ whose fibres above $s$ correspond to the de Rham cohomology of the pair $(X_{s},Z_{s})$, we begin by recalling the construction of the cone complexes associated to the maps $\Omega^{\bullet}_{X} \to \zeta_{*} \Omega^{\bullet}_{Z}$ and $\Omega^{\bullet}_{X/S} \to \zeta_{*} \Omega^{\bullet}_{Z/S}$, following \cite[Def. A.7]{zbMATH05233837}. They are given by
\begin{align*}
\textrm{Cone}^{q}(\zeta) &= \Omega^{q+1}_{X} \oplus \Omega^{q}_{Z} \\
\textrm{Cone}^{q}(\zeta/S) &= \Omega^{q+1}_{X/S} \oplus \Omega^{q}_{Z/S} \\
d(x, z) &= (-dx, -(\iota^{*} x) + dz) .
\end{align*}
As in \cite[Thm. 3.22]{zbMATH05233837} this complex is naturally equipped with the structure of a (for us, relative) mixed Hodge complex, giving rise to a long exact sequence
\[ \cdots \to R^{0} f_{*} \Omega^{\bullet}_{Z/S} \to \underbrace{R^{0} f_{*} \textrm{Cone}^{\bullet}(\zeta)}_{= \mathcal{V}} \to R^{1} f_{*} \Omega^{\bullet}_{X/S} \to \cdots \]
of algebraic vector bundles which equipped with the Hodge and weight filtrations $F^{\bullet}$ and $W$ and the isomorphism with
\[ \cdots \to R^{0} f_{*} \mathbb{Z}_{Z} \to R^{0} f_{*} \textrm{Cone}^{\bullet}(\mathbb{Z}_{X} \to \zeta_{*} \mathbb{Z}_{Z}) \to R^{1} f_{*} \mathbb{Z} \to \cdots \]
gives an exact sequence of mixed variations of Hodge structures.

We wish to prove:

\begin{prop}
\label{computeHodgebundle}
Given the data $(Z, X, S)$ and $(f, \iota, \zeta)$, there exists an algorithm to compute a model for $(\mathcal{V}, F^{\bullet}, \nabla)$, where $\nabla : \mathcal{V} \to \mathcal{V} \otimes \Omega^{1}_{S}$ is the Gauss-Manin connection.
\end{prop}

\begin{proof}
The proof is similar to \cite[Thm. 2.4]{urbanik2021sets}, although it is in various respects simpler. As in \cite[Thm. 2.4]{urbanik2021sets}, we begin by computing a relative affine cover for $X$ by complements of relative hyperplane sections; because $X$ is a relative curve, one only needs two such relative affines $U_{1}$ and $U_{2}$, whose complements are families of points over $S$. We may assume that $Z \subset U_{1}$ and $Z \subset U_{2}$ by choosing the hyperplane sections sufficiently generic via analogous reasoning as in \cite[Lem. 2.5]{urbanik2021sets}. We have an exact sequence
\[ 0 \to \underbrace{f^{*} \Omega^{\bullet+1}_{S} \oplus (f \circ \zeta)^{*} \Omega^{\bullet}_{S}}_{\mathcal{I}} \xrightarrow{\iota} \textrm{Cone}^{\bullet}(\zeta) \to \textrm{Cone}^{\bullet}(\zeta/S) \to 0 , \]
which because we are in relative dimension $1$ has only one interesting term, given by
\begin{equation}
\label{reldRconstr}
0 \to f^{*} \Omega^{1}_{S} \oplus (f \circ \zeta)^{*} \mathcal{O}_{S} \to \Omega^{1}_{X} \oplus \Omega^{0}_{Z} \to \Omega^{1}_{X/S} \oplus \mathcal{O}_{Z/S} \to 0 . 
\end{equation}

The complex $\textrm{Cone}^{\bullet}(\zeta)$ comes with a natural product structure $\wedge$ which is the direct sum of the same operation on the constituent complexes. This allows us to define a decreasing filtration $L^{\bullet}$ of $\textrm{Cone}^{\bullet}(\zeta)$ by requiring the $L^{i} \textrm{Cone}^{k-1}(\zeta)$ be spanned locally by elements of the form $\iota(x_{1}) \wedge \cdots \wedge \iota(x_{i}) \wedge y_{i+1} \wedge \cdots \wedge y_{k}$ with $x_{1}, \hdots, x_{i} \in \mathcal{I}$. Because we are in the setting of a relative curve, only the $k = 1$ term is interesting, and in this case it is given by
\begin{align*}
L^{0} \textrm{Cone}^{0}(\zeta) &= \Omega^{1}_{X} \oplus \mathcal{O}_{Z} \\
L^{1} \textrm{Cone}^{0}(\zeta) &= f^{*} \Omega^{1}_{S} \oplus (f \circ \zeta)^{*} \mathcal{O}_{S} .
\end{align*}
In particular $L^{2} = 0$ in our case, which together with the exactness of (\ref{reldRconstr}) implies that the rows in the diagram 
\begin{equation}
\label{bigconnconstrdiag}
\begin{tikzcd}
& 0 \arrow[d] & 0 \arrow[d] & 0 \arrow[d] & \\
0 \arrow[r] & (f \circ \zeta)^{*} \Omega^{\bullet-1}_{S} \otimes \Omega^{\bullet-1}_{Z} \arrow[r, "\wedge"]  \arrow[d] & \Omega^{\bullet}_{Z} \arrow[r] \arrow[d] & \Omega^{\bullet}_{Z/S} \arrow[r]  \arrow[d] & 0   \\
0 \arrow[r]   & (f^{*} \Omega^{\bullet}_{S} \oplus (f \circ \zeta)^{*} \Omega^{\bullet-1}_{S}) \otimes \textrm{Cone}^{\bullet-1}(\zeta/S) \arrow[r, "\wedge"] \arrow[d] & \textrm{Cone}^{\bullet}(\zeta) \arrow[r]  \arrow[d] & \textrm{Cone}^{\bullet}(\zeta/S) \arrow[r]  \arrow[d] & 0   , \\
0 \arrow[r] & f^{*} \Omega^{\bullet}_{S} \otimes \Omega^{\bullet}_{X} \arrow[r, "\wedge"] \arrow[d] & \Omega^{\bullet+1}_{X} \arrow[r] \arrow[d] & \Omega^{\bullet+1}_{X/S} \arrow[r] \arrow[d] & 0 \\
& 0  & 0  & 0 &
\end{tikzcd}
\end{equation}
are exact. (In the non relative curve setting one has to additionally quotient by $L^{2}$ in the middle column.) We claim the Gauss-Manin connection is the connecting homomorphism $\delta : \mathcal{V} \to \Omega^{1}_{S} \otimes \mathcal{V}$ associated to the middle row; here by ``Gauss-Manin'' connection we mean the connection whose associated bundle of flat sections is the complexification $\mathbb{V}_{\mathbb{C}}$ of the variation of mixed Hodge structure $\mathbb{V}$ on $S$ with fibres $H^{1}(X_{s}, Z_{s}; \mathbb{Z})$. The above diagram together with the natural mixed Hodge module structures on the associated complexes induces a diagram of sheaves on $S^{\textrm{an}}$
\begin{center}
\begin{tikzcd}
& 0 \arrow[d] & 0 \arrow[d] & 0 \arrow[d] \\
0 \arrow[r] & R^{0} f_{*} \mathbb{C} \arrow[r] \arrow[d] & R^{0} f_{*} \Omega^{\bullet}_{X^{\textrm{an}}/S^{\textrm{an}}} \arrow[r, "\delta_{X/S}"] \arrow[d] & \Omega^1_{S} \otimes R^{0} f_{*} \Omega^{\bullet}_{X^{\textrm{an}}/S^{\textrm{an}}} \arrow[d] \\
0 \arrow[r] & R^{0} (f \circ \zeta)_{*} \mathbb{C} \arrow[r] \arrow[d] & R^{0} (f \circ \zeta)_{*} \Omega^{\bullet}_{Z^{\textrm{an}}/S^{\textrm{an}}} \arrow[r, "\delta_{Z/S}"] \arrow[d] & \Omega^{1}_{S} \otimes R^{0} (f \circ \zeta)_{*} \Omega^{\bullet}_{Z^{\textrm{an}}/S^{\textrm{an}}} \arrow[d]  \\
0 \arrow[r] & \mathbb{V}_{\mathbb{C}} \arrow[r] \arrow[d] & \mathcal{V}^{\textrm{an}} \arrow[r, "\delta"] \arrow[d] & \Omega^{1}_{S} \otimes \mathcal{V}^{\textrm{an}} \arrow[d]   \\
0 \arrow[r] & R^{1} f_{*} \mathbb{C} \arrow[r] \arrow[d] & R^{1} f_{*} \Omega^{\bullet}_{X^{\textrm{an}}/S^{\textrm{an}}} \arrow[r, "\delta_{X/S}"] \arrow[d] & \Omega^{1}_{S} \otimes R^{1} f_{*} \Omega^{\bullet}_{X^{\textrm{an}}/S^{\textrm{an}}} \arrow[d] \\
& 0  & 0  & 0 .
\end{tikzcd}
\end{center}

The arrows between the first and second columns are the Betti-analytic de Rham comparisons, and the second set of arrows between the second and third columns are the connecting homomorphisms from the diagram (\ref{bigconnconstrdiag}). Every row and column is exact except for possibly the row involving $\delta$; for the rows involving $\delta_{Z/S}$ and $\delta_{X/S}$ this is due to the theorem of \cite{katzoda}. It then follows from a diagram chase that the row involving $\delta$ is also exact. 

More precisely, one begins with a local section $h \in \mathcal{V}^{\textrm{an}}$ in the kernel of $\delta$. Its image in $R^{1} f_{*} \Omega^{\bullet}_{X^{\textrm{an}}/S^{\textrm{an}}}$ then comes, via exactness, from a section of $R^{1} f_{*} \mathbb{C}$ which lifts to a section $h'$ of $\mathbb{V}_{\mathbb{C}}$. The difference $h - h'$ then has a unique lift to the relative cohomology sheaf obtained as the quotient of the first two non-zero entries in the middle column, which from commutativity of the diagram necessarily comes from a section of the relative cohomology sheaf $R^{0} (f \circ \zeta)_{*} \mathbb{C} / R^{0} f_{*} \mathbb{C}$. It then follows that the difference $h - h'$ arises from an element of $\mathbb{V}_{\mathbb{C}}$, and as we already know this for $h'$, we obtain the conclusion for $h$.

\vspace{0.5em}

What remains to be shown is that we can compute the vector bundle $\mathcal{V}$, its filtered subbundle $F^{1} \mathcal{V}$, and the map $\delta$. Note that as a consequence of \cite[Prop. 2.1]{urbanik2021sets}, direct sums, tensor products, quotients, and pullbacks of finite rank locally free sheaves on algebraic varieties are all computable, in the sense that one can compute explicit models of the resulting locally free sheaves from models of the input data; see \cite[\S2.1]{urbanik2021sets}. Thus it is possible to compute all the objects in the diagram (\ref{bigconnconstrdiag}) of vector bundles on $X$. This allows us to construct, as in \cite[Thm. 2.4]{urbanik2021sets}, a \v{C}ech double complex $C^{\bullet,\bullet}$ computing the cohomology of the complex $\textrm{Cone}^{\bullet}(\zeta)$.

However the terms of the double complex $C^{\bullet,\bullet}$ have infinite rank over $\mathcal{O}_{S}$, and so to show that we may actually use it to compute $\mathcal{V}, F^{1} \mathcal{V}$ and $\delta$ the essential difficulty, as in \cite[Lem. 2.6]{urbanik2021sets}, is to produce a finite rank subcomplex whose cohomology agrees with the cohomology of these sheaves. Here the task is even easier than in \cite[Lem. 2.6]{urbanik2021sets}, since we already know what the correct rank of these cohomology bundles are: the reduced cohomology in degree $0$ has rank $n = \textrm{length}(\kappa) - 1$, the pure cohomology in degree $1$ rank $2g$, and therefore $\mathcal{V}$ has rank $2g + n$. As in the proof of \cite[Thm. 2.4]{urbanik2021sets}, we may iteratively approximate the cohomology by \v{C}ech complexes indexed by the order of the pole along a union of relative hyperplane sections. Since each such approximation gives a subcomplex of the true \v{C}ech cohomology complex, we obtain an increasing sequence of coherent subsheaves of $\mathcal{V}$. Computing successively higher-order approximations, we terminate when we reach a locally free bundle of rank $2g + n$.
\end{proof}

\section{Algorithms}\label{secappliactionseff}

We now explain how our methods can be used to give algorithms to compute atypical special loci in cases of interest. Our algorithms will work with computational models --- in the sense described by the previous section --- for the data $(\mathcal{V}, W_{\bullet}, F^{\bullet}, \nabla, P)$ over a smooth algebraic variety $S$, with everything defined over a fixed number field $E$. Here $(\mathcal{V}, W_{\bullet}, F^{\bullet}, \nabla)$ are the vector bundles, filtrations and the connection associated to a fixed mixed integral variation of Hodge structure $\mathbb{V}$, and $P$ is the associated torsor constructed as in \autoref{examplebundle}. In the notation of \autoref{constrcatequiv}, $P$ is a subbundle of $\mathbb{B}(\sheafhom(\mathcal{V}, \mathcal{O}_{S} \otimes_{E} \mathcal{V}_{s_{0}}))$ for some point $s_{0} \in S(E)$. Note that one can compute the connection associated to $P$ from the data $(\mathcal{V}, \nabla)$. Note as well that we regard the algebraic de Rham incarnation of the group $\mathbf{H}_{S}$ as included in the data of the $\mathbf{H}_{S}$-torsor $P$.

A basic observation that we will use is the following:

\begin{lem}
\label{computeZfdlem}
There exists an algorithm to compute the locus $\mathcal{Z}(f,e)$ appearing in the proof of \autoref{nondenseA}.
\end{lem}

\begin{proof}
The proof of \autoref{nondenseA} describes a recursive construction of $\mathcal{Z}(f,e)$ using the algebraic data $(P, f, \nabla)$, and each step in this construction admits a computational interpretation in our fixed computational model.
\end{proof}

\subsection{Computing weakly specials}

\begin{defn}
Let $Y \subset S$ be an irreducible subvariety of $S$. 
Then the \emph{type} of $(Y,\mathbb{V}_{|Y})$ is the $\mathbf{H}_{S}(\mathbb{C})$-equivalence class of the sub-flag variety $\ch{D}^{0}_{Y} \subset \ch{D}^{0}_{S}$. 
\end{defn}
(The notion of type in the pure setting already appears in \cite[\S4.3]{urbanik2021sets}.)

One reason we are interested in the \emph{type} of $Y$ rather than the variety $\ch{D}^{0}_{Y}$ itself, is because the presence of the $\mathbf{H}_{S}(\mathbb{C})$-action makes types amenable to computation, whereas computing $\ch{D}^{0}_{Y}$ exactly (not just up to $\mathbf{H}_{S}(\mathbb{C})$-equivalence) requires computing periods. Computationally we will represent a type as any representative of its equivalence class. These representatives will be viewed inside the variety $\ch{D}^{0}_{S}$, which we construct explicitly using the algebraic de Rham data associated to the fibre $\mathcal{V}_{s_{0}}$ above the fixed point $s_{0}$. In particular, we can construct $\ch{D}^{0}_{S}$ from the algebraic incarnations of $W_{\bullet}, F^{\bullet}$ and $P$ by taking the $\mathbf{H}_{S,s_{0}}$ orbit of $F^{\bullet}_{s_{0}}$ in the variety of Hodge flags on $\mathcal{V}_{s_{0}}$ compatible with $W_{\bullet, s_{0}}$. 

Given types $\mathcal{C}_{1}, \mathcal{C}_{2}$ represented by $\ch{D}_{1}$ and $\ch{D}_{2}$ we say $\mathcal{C}_{1} \leq \mathcal{C}_{2}$ if there exists $g \in \mathbf{H}_{S}(\mathbb{C})$ such that $\ch{D}_{1} \subset g \ch{D}_{2}$. The following is immediate from the description of weakly specials given in \S\ref{wspsubsinHdgthy}.

\begin{lem}
An irreducible subvariety $Y \subset S$ is weakly special if and only if it is maximal for its type. If $W \subset Y$ and the dimension of the algebraic monodromy of $Y$ drops upon restricting to $W$, so does the dimension of the type. \qed
\end{lem}

\begin{lem}
\label{typecompalg}
There exists an algorithm which, upon input an irreducible subvariety $Y \subset S$ defined over a number field, outputs the type of $Y$.
\end{lem}

\begin{proof}
Recall that in \autoref{wspdomainoverparamlem} we constructed a family $g : \mathcal{D} \to \mathcal{Y}$ of subvarieties of $\ch{D}^{0}_{S}$ containing all weakly special domains. By enlarging the family (or inspecting its actual construction), one may assume that it contains all $\mathbf{H}_{S}(\mathbb{C})$-translates of weakly special domains. Pulling back along $\nu : P \to \ch{D}^{0}_{S}$ this becomes a family $f : \mathcal{Z} \to \mathcal{Y}$ of subvarieties of $P$. Let $f_{Y}$ be the family whose fibres are the intersections between fibers of $f$ and $\restr{P}{Y}$. 

Then computing the locus $\mathcal{Z}(f,e) \subset \mathcal{Z}$ with $e = \dim Y$ and taking its image $\mathcal{I} \subset \mathcal{Y}$ we obtain the locus of all subvarieties of $\ch{D}^{0}_{S}$ belonging to $g$ which contain the image of some leaf above $Y$. The variety $\mathcal{Y}$ is naturally partitioned as $\mathcal{Y} = \bigsqcup_{m \geq 0} \mathcal{Y}_{m}$, where the fibres of $f$ above $\mathcal{Y}_{m}$ have dimension $m$. If we then consider the smallest $m$ such that $\mathcal{I} \cap \mathcal{Y}_{m}$ is non-empty, we obtain exactly the moduli of all representatives of the type of $Y$. 
\end{proof}

\noindent \autoref{typecompalg} generalizes easily to the case where we have a family $g : C \to B$ of subvarieties of $S$.

\begin{prop}
\label{typestratprop}
There exists an algorithm which, given as input a family $g : C \to B$ of irreducible subvarieties of $S$ defined over a number field, outputs a constructible partition $B = \bigsqcup_{i=1}^{k} B_{i}$ such that above each $B_{i}$ the fibres of $g$ have the same type.
\end{prop}

\begin{proof}
We may assume $B$ is irreducible. The basic idea is that we can carry out the previous algorithm over the function field $Q(B)$ of $B$ to compute the type over a Zariski open subset $B' \subset B$, and then applying the argument to the family over $B \setminus B'$ we obtain the result by Noetherian induction. Indeed, after base-changing all the objects $S, P, \mathcal{V}, \nabla$, etc., along $E \hookrightarrow Q(B)$ we may apply the same computational procedure appearing in \autoref{typecompalg} to $Y = C_{Q(B)}$ viewed as a subvariety of $S_{Q(B)}$. Applying the algorithm in \autoref{typecompalg} for this $Y$ requires finitely many elements $d_{1}, \hdots, d_{j} \in Q(B)$, so we may then interpret the computation as taking place over an open locus $B' \subset B$ where $d_{1}, \hdots, d_{j}$ are defined, and then this computation computes the type of the fibres of $g$ above $B'$. Computing recursively gives the result. 
\end{proof}

\begin{cor}
\label{infsecfamiliestype}
Work under the same assumptions as \autoref{typecompalg}. Then there exists a non-terminating algorithm that outputs an infinite sequence $\{ g_{i} : C_{i} \to B_{i} \}_{i=1}^{\infty}$ of families of closed irreducible subvarieties of $S$ with the following two properties:
\begin{itemize}
\item[(1)] every irreducible subvariety of $S$ appears exactly once at a fibre of some $g_{i}$; and
\item[(2)] for each $i$, the fibres of $g_{i}$ all have the same type.
\end{itemize}
\end{cor}

\begin{proof}
By the theory developed in \cite{lella2012computable} it is possible to compute components of Hilbert schemes of projective algebraic varieties together with the universal family over them. This can be used to compute a sequence $\{ g_{i} : C_{i} \to B_{i} \}_{i=1}^{\infty}$ of families of subvarieties of a fixed quasi-projective variety $S$ such that every irreducible subvariety of $S$ occurs as a fibre; cf. \cite[\S5.3]{urbanik2021sets}. One then applies \autoref{typestratprop} to refine this sequence by stratifying the bases $B_{i}$ by the type of the fibre. 
\end{proof}

\begin{prop}
\label{terribletorsorsearch}
Let $R$ and $T$ be as in \autoref{intbecometypprop}. Then there exists an algorithm that computes:
\begin{itemize}
\item[(A)] an integer $n \geq 0$, and a flat locally free subsystem $\mathcal{F}_{N} \subset \restr{\left( \bigoplus_{0 \leq a, b \leq n} \mathcal{V}^{\otimes a} \otimes (\mathcal{V}^{*})^{\otimes b} \right)}{T}$;
\item[(B)] an algebraic group $N$; and
\item[(C)] a family $g : C \to B$ of irreducible subvarieties of $T$, with $B$ possibly disconnected;
\end{itemize}
such that
\begin{itemize}
\item[(i)] for each fibre $C_{b}$ of $g$ the sections of $\restr{\mathcal{F}_{N}}{C_{b}}$ are flat;
\item[(ii)] at each point $s \in C_{b}$, the group stabilizing all lines in $\mathcal{F}_{N,s}$ is abstractly isomorphic to $N$; and
\item[(iii)] each germ $U$ as in \autoref{intbecometypprop} lies inside an $N$-subtorsor $P'$ of some $\restr{P}{C_{b}}$ which is preserved by the connection $\nabla$, and such that the inequality in \autoref{intbecometypprop}(ii) holds if $P_{(x,y)}$ is replaced by $P'$.
\end{itemize}
\end{prop}

\begin{proof}
The basic idea is that the statement of \autoref{intbecometypprop} guarantees that data as in the statement exist (there one has $N = \mathbf{N}_{R}$). It is possible to computationally enumerate a countable set of possibilities for this data (i.e., to search for it), and test when one has found data satisfying the hypotheses in the statement. Therefore an algorithm exists, which consists of searching for data satisfying the required properties and terminating when it has been found.

Let us give some additional details. Because $\mathcal{Z}(f,e)$ is defined over a number field (because $f, P, \nabla$, etc., are), so are the varieties $R$ and $T$. Then for each tensor power $\restr{\mathcal{V}^{\otimes a} \otimes (\mathcal{V}^{*})^{\otimes b}}{T}$, maximal locally free submodules $\mathcal{F}_{N}$ composed of $\mathbf{N}_{R}$-invariants are also defined over a number field: such submodules are characterized by the fact that they restrict to modules of flat sections upon restricting a generic fibre of the family $g$ appearing in \autoref{intbecometypprop}, but this condition can be defined over a number field as a consequence of \autoref{restfamlem} below. Since $\mathbf{N}_{R}$ is determined by its invariant subspaces, computing such a submodule also lets one compute the group $\mathbf{N}_{R}$.

Now by computationally enumerating families $g : C \to B$ of constant type with $C \to T$ dominant using \autoref{infsecfamiliestype} (applying \autoref{infsecfamiliestype} with $S = T$), and submodules $\mathcal{F}_{g}$ of tensor powers of $\restr{\mathcal{V}}{T}$ which restrict to modules of flat sections over generic points of $g$, one obtains a computational enumeration of pairs $(g, \mathcal{F}_{g}, N_{g})$, with $N_{g}$ the abstract group stabilizing fibres of $\mathcal{F}_{g}$, which potentially satisfy (i), (ii) and (iii) above. To check whether this is in fact the case, one uses $\mathcal{F}_{g}$ to construct a family $f_{g}$ from $f$ obtained by intersecting the fibres of $f$ with the $N_{g}$-torsors above the fibres of $g$, checks that $\mathcal{Z}(f_{g}, e)$ surjects onto $R$. One then constructs the locus in $\mathcal{Z}(f_{g}, e)$ where the inequality in \autoref{intbecometypprop}(ii) holds, and checks that it maps to a dense subset of $R$.
\end{proof}

\begin{rem}
\label{maketerriblesearcheffrem}
As our goal here is merely to show the existence of algorithms for computing atypical loci defined by differential conditions (e.g., orbit closures), and not implementing them, we have not concerned ourselves with the efficiency of the algorithms we propose. However in practice one should be able to improve the ``search procedure'' in \autoref{terribletorsorsearch} in two ways:
\begin{itemize}
\item[(i)] The intersections parameterized by $R$ can be used directly to find flat submodules $\mathcal{F}_{\mathbf{N}_{R}}$ associated to $\mathbf{N}_{R}$: the basic idea is to adopt a similar approach to the proof of the Ax-Schanuel theorem \autoref{thm:newAS} appearing in {\cite[Thm. 3.6]{2021arXiv210203384B}}, which is, as noticed for instance by Dogra \cite[Thm. 4]{2022arXiv220604304D}, effective. Ultimately the proof of the Ax-Schanuel theorem constructs a Lie subalgebra of the Lie algebra of $\mathbf{H}_{S}$ whose associated algebraic subgroup gives a subtorsor where the atypical intersection of \autoref{thm:newAS} occurs. Carrying this out in families should simplify the process of finding $\mathbf{N}_{R}$, $\mathcal{F}_{\mathbf{N}_{R}}$, and its associated torsors. 
\item[(ii)] It should be possible to construct the family $g$, at least over a Zariski open subset of $C$ which has dense image in $T$, using $\mathcal{F}_{\mathbf{N}_{R}}$ directly. The point is that the variation of mixed Hodge structure from which $g$ is constructed in \autoref{intbecometypprop} is naturally associated to a quotient $\mathcal{Q}$ of some tensor power of $\mathcal{V}$ by a $\mathbf{N}_{R}$-invariant subbundle, and given this quotient $\mathcal{Q}$ one can understand the fibres of $g$ as maximal loci over which the Hodge flag on $\mathcal{Q}$ is preserved by the connection. 
\end{itemize}
\end{rem}

\begin{lem}
\label{restfamlem}
Up to replacing the base $B$ with a disjoint union of strata, the family $g$ appearing in \autoref{intbecometypprop} admits a model over a number field if all the data $S, P, \nabla, f, R, T$, etc. are defined over a number field.
\end{lem}

\begin{proof}
Fix some family $\overline{g} : \overline{C} \to \overline{B}$ of subvarieties of $T$, containing all fibres of $g$ among its fibres, such that $\overline{g}$ is defined over a number field: this can always be done, for instance, by spreading out. Then consider the constructible locus $B' \subset \overline{B}$ of points $\overline{b}$ where the fibre $\overline{g}^{-1}(\overline{b})$ has the type contained in the type $\mathcal{C}_{R}$ defined by a $\mathbf{N}_{R}(\mathbb{C})$-orbit in $\ch{D}^{0}_{T}$: we have seen in \autoref{typestratprop} that this is a condition defined over a number field. The fibres of $g$ are then exactly the fibres above $B'$ which are maximal (not contained in another fibre above $B'$), and this is again a condition defined over a number field.
\end{proof}

\subsection{Atypical orbit closures}
\label{sec:compoforbitclosures}

In this section we give our algorithm for computing atypical orbit closures. Throughout we fix an ambient orbit closure $\mathcal{M} = \Omega \mathcal{M}_{g}(\kappa)$, and write $Z_{\mathcal{M}}$ for the variety defined in \autoref{orbitoverparamlem}. We assume we have already computed the triple $(\mathcal{V}, W_{\bullet}, F^{\bullet}, \nabla)$ consisting of the Hodge bundle $\mathcal{V}$, the weight and Hodge filtrations $W_{\bullet}$ and $F^{\bullet}$, and the Gauss-Manin connection $\nabla$, and moreover that we have also computed the torsor $P$ with its $\mathbf{H}_{S}$-equivariant connection (also denoted $\nabla$). That we may assume this data is available is justified by \autoref{computeHodgebundle} and \S\ref{comporbitinputdatasec} below.

\begin{thm}
\label{mainorbitclosurealg}
For each stratum $S = \Omega \mathcal{M}_{g}(\kappa)$, there exists an algorithm which computes the set of all maximal atypical orbit closures $\mathcal{N} \subset S$.
\end{thm}

\subsubsection{Computing the input data}
\label{comporbitinputdatasec}

To start with, we will explain how we may compute the input data to \autoref{mainorbitclosurealg}. The family $f : \mathcal{Z} \to \mathcal{Y}$ we take to be the same family appearing in \S\ref{sectionoverpar} in the proof of finiteness of maximal atypical orbit closures.

Now as was explained in (\ref{eqmonodromy}) and \S\ref{locsysandmonosec}, in this case the group $\mathbf{H}_{S}$ is nothing more than the group stabilizing the global symplectic form $F$, which means that $P \subset \mathbb{B}(\sheafhom(\mathcal{V}, \mathcal{V}_{s_{0}} \otimes_{E} \mathcal{O}_{S}))$ may be defined by 
\[ P = \{ (s,\varphi) : \hspace{0.5em} F(\varphi(v), \varphi(w)) = F(v,w) \hspace{0.5em} \textrm{ for all } v,w \in \mathcal{V}_{s} \hspace{0.5em} \} . \]
(In the above $\varphi \in \Hom(\mathcal{V}_{s}, \mathcal{V}_{s_{0}})$.) In our setting the form $F$ is given by the cup product. In \v{C}ech cohomology this can be computed by an explicit formula, given for instance in \cite[\href{https://stacks.math.columbia.edu/tag/01FP}{Section 01FP}]{stacks-project}. Note that \cite[\href{https://stacks.math.columbia.edu/tag/01FP}{Section 01FP}]{stacks-project} also explains the relationship with the wedge product of differential forms.

Using this formula and our \v{C}ech computation of $\mathcal{V}$ by computing explicit generators, one can compute $F$ as an algebraic map $F : \mathcal{V} \otimes_{\mathcal{O}_{S}} \mathcal{V} \to \mathcal{O}_{S}$ of locally free sheaves and then construct $P$ as above. 

\vspace{0.5em}

To construct $f$ we start by observing we can compute the map $r \times \nu$ appearing in \autoref{mapfromPsurj}. Here, both factors are constructed by viewing an element $(s, \varphi) \in P$ as a map $\varphi : \mathcal{V}_{s} \to \mathcal{V}_{s_{0}}$ and evaluating $\varphi$ on $\omega_{s}$ (in the case of $r$) and $F^{\bullet}_{s}$ (in the case of $\nu$). That the section $\omega$ of $\mathcal{V}_{\textrm{abs}}$ can be computed is immediate from the description of the computation of the algebraic de Rham cohomology, since it is naturally a section of $\Omega^{1}_{S}$ and therefore appears in the \v{C}ech cohomology calculation. We also saw in the proof of \autoref{computeHodgebundle} that we may compute the filtration.

Given this, constructing $f$ reduces to computing the families appearing in \autoref{gaolemma} in the case where $\mathcal{M} = \Omega \mathcal{M}_{g}(\kappa)$. This can either be done by following the over-parametrization construction in \S\ref{prelhodgedataconj}, or through a more careful classification of the relevant mixed Shimura data, as for instance in \cite{zbMATH06801925} and \cite{zbMATH07305885}.

\subsubsection{No Intermediate Typicality}

In this section we establish an important fact about the differential conditions that define atypical orbit closures. This will say, roughly, that atypical orbit closures $\mathcal{N} \subset \mathcal{M}$ are always atypical intersections inside any intermediate subvariety $\mathcal{N} \subsetneq Y \subset \mathcal{M}$ which satisfies $\mathbf{H}_{\mathcal{N}} \subsetneq \mathbf{H}_{Y}$. This will be used to ensure that the differential conditions we impose on orbit closures define loci with the correct monodromy groups.

The remainder of this subsection is devoted to the proof of the following proposition.

\begin{prop}
\label{nointtypprop}
Let $Y \subset \mathcal{M}$ be an irreducible algebraic subvariety containing an atypical orbit closure $\mathcal{N}$. Suppose that $\mathbf{N} \subset \mathbf{H}_{S}$ is a $\mathbb{Q}$-algebraic group containing $\mathbf{H}_{Y}$, and let $P_{Y} \subset P$ be a fixed algebraic $\mathbf{N}$-torsor above $\mathcal{M}$ containing a leaf $\mathcal{L}_{Y}$ arising from the $\mathbb{Q}$-local system $\restr{\sheafhom(\mathbb{V}, \mathbb{V}_{s_{0}})}{Y}$. 

Recall from \S\ref{rephsec} that $\mathcal{N}$ arises as the image in $\mathcal{M}$ of an irreducible component $C$ of the locus $(r \times \nu)^{-1}(Z_{\mathcal{N}}) \cap \mathcal{L}_{Y}$. Fix a component $I \subset (r \times \nu)^{-1}(Z_{\mathcal{N}})$ containing $C$. Suppose that
\begin{equation}
\label{typcondonC}
\dim I \geq \dim P_{Y} - \dim \mathcal{L}_{Y} + \dim \mathcal{N} .
\end{equation}
Then $\mathbf{H}_{\mathcal{N}} = \mathbf{N}$.
\end{prop}

\begin{proof}
We start by fixing a point $x \in C$ which lies in the smooth locus of $(r \times \nu)^{-1}(Z_{\mathcal{N}}) \cap \mathcal{L}_{Y}$. Consider the $\mathbf{N}(\mathbb{C})$ orbit $O_{Y} = \mathbf{N}(\mathbb{C}) \cdot (r \times \nu)(x)$. Then $(r \times \nu)^{-1}(O_{Y})$ is a $\mathbf{N}(\mathbb{C})$-subtorsor of $P_{Y}$ which contains $C$. We fix a component $P_{0} \subset (r \times \nu)^{-1}(O_{Y})$ containing $I$ (and hence $C$), write $Y_{0} \subset Y$ for its image in $Y$, and write $\mathcal{L}_{0} \subset \mathcal{L}_{Y} \cap P_{0}$ for the component containing $C$. Then (\ref{typcondonC}) can be alternatively written as
\begin{equation}
\label{typcondonC2}
\dim I \geq \dim P_{0} - \dim \mathcal{L}_{0} + \dim \mathcal{N} .
\end{equation}
We will write $\tau : P_{0} \to O_{Y}$ for the restriction of $r \times \nu$. We note that by \autoref{imofIY0} below we have $\tau^{-1}(\mathcal{Z}_{\mathcal{N}}) = I$ and that $I$ is a local complete intersection in $P_{0}$ of codimension $\codim_{O_{Y}} Z_{\mathcal{N}}$.

Now write $O_{1} = Z_{\mathcal{N}} \subset O_{Y}$, and for $g \in \mathbf{N}(\mathbb{C})$, set $O_{g} = g \cdot O_{1}$. Let $I_{1} = I$ and set $I_{g} = g \cdot (r \times \nu)^{-1}(O_{g})$. Consider the equality
\begin{equation}
\label{typcondonC3}
\codim_{P_{0}, w} (\mathcal{L}_{0} \cap I_{g}) = \codim_{P_{0}, w} I_{g} + \codim_{P_{0}, w} \mathcal{L}_{0} ,
\end{equation}
where $w$ varies in $\mathcal{L}_{0}$. We claim the equation (\ref{typcondonC2}) implies that (\ref{typcondonC3}) holds at $(w,g) = (x,1)$: since $I_{g} \subset P_{0}$ and $\dim_{x} (\mathcal{L}_{0} \cap I_{q}) = \dim \mathcal{N}$, one obtains (\ref{typcondonC3}) with ``$=$'' replaced by ``$\geq$'' by rearranging and adding $\dim P_{0}$ to both sides. On the other hand because $I_{1}$ is a local complete intersection, one also has (\ref{typcondonC3}) with ``$=$'' replaced by ``$\leq$'', giving the equality. 

\hspace{0.5em}

We now explain why one can use the condition (\ref{typcondonC3}) at $(x,1)$ to produce infinitely many atypical orbit closures unless in fact $\mathbf{H}_{\mathcal{N}} = \mathbf{N}$. 

\hspace{0.5em}

The intersections $\mathcal{L}_{0} \cap I_{g}$ fit into an analytic family $\zeta : \mathcal{F} \to \mathbf{N}(\mathbb{C})$. We may regard $\mathcal{F}$ as a closed analytic subvariety of $\mathcal{L}_{0} \times \mathbf{N}(\mathbb{C})$ containing $(x,1)$. Then by upper semicontinuity of fibre dimension, (\ref{typcondonC3}) is open in a neighbourhood $\mathcal{U} \subset \mathcal{F}$ containing $(x,1)$. More precisely, the quantity $\dim_{z} \zeta^{-1}(\zeta(z))$ is constant on $\mathcal{U}$. 

Now applying \cite[VII, Prop. 3]{zbMATH03269181}, we find that, shrinking $\mathcal{U}$ if necessary, the image $\zeta(\mathcal{U})$ is a closed analytic set containing $1 \in \mathbf{N}(\mathbb{C})$ inside some open neighbourhood $\mathcal{B} \subset \mathbf{N}(\mathbb{C})$. We will argue that in fact $\zeta(\mathcal{U})$ is open, and hence equal to $\mathcal{B}$. 

We first observe that the projection of $\mathcal{U}$ to $\mathcal{L}_{0}$ contains an open neighbourhood. Indeed $I_{1}$ intersects every fibre of $P_{0} \to Y_{0}$, hence this is also true for the translates $I_{g} = \tau^{-1}(O_{g})$. Because $\mathbf{N}(\mathbb{C})$ acts transitively on the fibres of $P_{0} \to Y_{0}$, this means that we can find, for any $w \in \mathcal{L}_{0}$, some $g$ for which $\tau^{-1}(O_{g})$ contains $w$. Moreover $g$ can be taken to depend continuously on $w$ in some neighbourhood of $x$, which means that $\mathcal{U}$ maps onto a neighbourhood $\mathcal{X}$ of $x$. 

Now choose $w \in \mathcal{X}$ where $\mathcal{L}_{0}$ is smooth. Consider a point $(w,g_{0}) \in \mathcal{U}$, and let $y$ be the image of $w$ in $Y_{0}$. Then $\mathcal{L}_{0}$ being smooth at $w$ implies that $Y_{0}$ is smooth in a neighbourhood $\mathcal{Y}$ of $y$, and hence the map $\restr{P_{0}}{\mathcal{Y}} \to O_{Y}$ obtained by restricting $\tau$ is a submersion of smooth manifolds. In particular, the variety $\tau^{-1}(O_{g_{0}})$ is smooth at $w$, and the intersection $\mathcal{L}_{0} \cap \tau^{-1}(O_{g_{0}})$ is, locally at $w$, a transverse intersection of smooth manifolds.

Now applying \cite[Ch. 1, \S6, Ex. 11]{zbMATH03562121}, there exists a neighbourhood $\mathcal{B}_{0}$ of $g_{0}$ and $\mathcal{W}$ of $w$ such that the intersections $\mathcal{L}_{0} \cap \tau^{-1}(O_{g})$ are non-empty and contain a point of transverse intersection in $\mathcal{W}$ for $g \in \mathcal{B}_{0}$. In particular, $\zeta$ is a submersion at $(w, g_{0})$. It then follows that $\zeta(\mathcal{U})$ contains an open neighbourhood of $g_{0}$. But $\zeta(\mathcal{U})$ was closed analytic in $\mathcal{B}$, so in fact $\zeta(\mathcal{U}) = \mathcal{B}$. 

\vspace{0.5em}

All of the above has shown that we have a neighbourhood $\mathcal{B}$ of $1 \in \mathbf{N}(\mathbb{C})$ such that the intersection of $\tau^{-1}(O_{1})$ and $\mathcal{L}_{0}$ has a component $C_{g}$ of dimension $e = \dim \mathcal{N}$ for all $g \in \mathcal{B}$. If we now take $g \in \mathbf{N}(\mathbb{Q})$, the components $C_{g}$ satisfy all the conditions of \autoref{mainthmalgfilip}: as in the proof of \autoref{denseinSprop} the translated condition $g \cdot Z_{\mathcal{N}}$ defines a locus with $\mathbf{N}(\mathbb{Q})$-conjugate real multiplication, torsion, and eigenform conditions, and by construction one has $\dim \mathcal{N} = \dim C_{g}$. By varying $g \in \mathbf{N}(\mathbb{Q}) \setminus \mathbf{H}_{\mathcal{N}}(\mathbb{Q})$ one obtains infinitely many atypical orbit closures inside $Y_{0}$, unless in fact $\mathbf{H}_{\mathcal{N}} = \mathbf{N}$. 
\end{proof}

\begin{lem}
\label{imofIY0}
The image of $I$ in $Y$ is $Y_{0}$, and $I = (r \times \nu)^{-1}(Z_{\mathcal{N}}) \cap P_{0}$. Moreover $I$ is a local complete intersection in $P_{0}$.
\end{lem}

\begin{proof}
Consider the map $\tau : P_{0} \to O_{Y}$ obtained by restricting $r \times \nu$. This is a surjective map which is moreover surjective on each fibre of $P_{0} \to Y_{0}$. It follows that $\tau^{-1}(Z_{\mathcal{N}})$ surjects onto $Y_{0}$, and is a fibration over $Y_{0}$. In particular, $\tau^{-1}(Z_{\mathcal{N}})$ is irreducible containing $I$. But $I$ is a component of $\tau^{-1}(Z_{\mathcal{N}})$.

From this description we see that $I$ has codimension $\codim_{O_{Y}} Z_{\mathcal{N}}$. But it is also locally defined by $\codim_{O_{Y}} Z_{\mathcal{N}}$ functions, since it is the pullback of the local complete intersection $Z_{\mathcal{N}} \subset O_{Y}$.
\end{proof}

\begin{prop}
\label{nointtypprop2}
If in \autoref{nointtypprop} the group $\mathbf{N}$ is instead assumed to be a complex algebraic group (not necessarily defined over $\mathbb{Q}$), one obtains instead the conclusion $\mathbf{H}_{\mathcal{N}} = \mathbf{H}_{Y}$. 
\end{prop}

\begin{proof}
The fact that the fixed $\mathbf{N}$-torsor contains the leaf $\mathcal{L}_{Y}$ --- whose Zariski closure, as a consequence of \autoref{galequalsmono}, is a $\mathbf{H}_{Y}$-torsor --- implies that $\mathbf{H}_{Y} \subset \mathbf{N}$. The proof of \autoref{nointtypprop} then proceeds unchanged until the final paragraph, where one instead varies $g \in \mathbf{H}_{Y}(\mathbb{Q}) \setminus \mathbf{H}_{\mathcal{N}}(\mathbb{Q})$, and produces infinitely many atypical orbit closures unless $\mathbf{H}_{\mathcal{N}} = \mathbf{H}_{Y}$.
\end{proof}

\subsubsection{Analysis of Orbit-closure-like Differential Intersections}

Fix a component $R \subset \mathcal{Z}(f,e)$ and let $T \subset S$ be the Zariski closure of its image.

\begin{defn}
We say that $R$ \emph{contains an orbit closure point} $(x,y)$ if there exists an orbit closure $\mathcal{N} \subset S$ of dimension $e$ such that
\begin{itemize}
\item[-] $x \in P(\mathbb{C})$ is a point above $\mathcal{N}$ which corresponds to a rational element $\varphi \in \Hom(\mathbb{V}_{s}, \mathbb{V}_{s_{0}})$; and
\item[-] we have $f^{-1}(y) = (r \times \nu)^{-1}(Z_{\mathcal{N}})$, and that the germ $(\mathcal{L}_{x} \cap f^{-1}(y), x)$ (which we may assume is irreducible after possibly replacing $x$) projects surjectively onto a germ of $\mathcal{N}$. 
\end{itemize}
\end{defn}

Suppose $R$ contains an orbit closure point $(x,y)$. We consider the group $\mathbf{N}_{R} \subset \mathbf{H}_{T}$ associated to $R$, with both groups embedded into $\GL(\mathcal{V}_{s_{0}})$ using the flat structure of $\mathbb{V}$, as in \autoref{intbecometypprop}. Let $\pi : P \to S$ denote the natural projection. Then applying \autoref{intbecometypprop}(i) one has that
\[ \pi(\mathcal{L}_{x} \cap f^{-1}(y), x) = (\mathcal{N}, \pi(x)) \subset (Y_{(x,y)}, x) \subset (T, x) \]
which implies that $\mathbf{H}_{\mathcal{N}} \subset \mathbf{H}_{Y_{(x,y)}} \subset \mathbf{N}_{R}$, where the last containment is part of \autoref{intbecometypprop}(i).  (Here we use the notation $(A, a)$ to denote the analytic germ of $A$ at a point $a \in A$.)

\begin{lem}
Suppose $R$ contains an orbit closure point $(x,y)$ corresponding to $\mathcal{N}$. Then we have $\mathbf{H}_{\mathcal{N}} = \mathbf{N}_{R}$. 
\end{lem}

\begin{proof}
This follows by combining \autoref{intbecometypprop} with \autoref{nointtypprop}.
\end{proof}

\begin{lem}
Suppose $R$ contains an orbit closure point $(x,y)$. Then the reductive quotient of $\mathbf{H}_{T}$ is contained in the Mumford-Tate group $M := (\operatorname{Res}_{K/\mathbb{Q}} \operatorname{GSp}_{2r,K}) \times \operatorname{GSp}_{2(g - dr)}$ associated to the pure condition in \autoref{mainthmalgfilip}. In particular, $T$ lies in a component of the locus defined by the condition $S[\mathcal{N}]$ introduced in \S\ref{diagramatyp}. 
\end{lem}

\begin{proof}
Because $\mathbf{N}_{R}$ is normal in $\mathbf{H}_{T}$, the same is true of their reductive quotients $\mathbf{N}^{\textrm{red}}_{R}$ and $\mathbf{H}^{\textrm{red}}_{T}$. These groups may be identified with the algebraic monodromy groups of $\mathcal{N}$ and $T$ for the pure $\mathbb{Z}$VHS $\mathbb{V}_{\textrm{abs}} = \mathbb{V} / W_{0}$ over $S$. 

Algebraic monodromy groups of pure variations of Hodge structures are semisimple (this follows from the reductivity of pure Mumford-Tate groups and \autoref{monodromytheorem}), so we have an almost-direct product decomposition $\mathbf{H}^{\textrm{red}}_{T} = \mathbf{N}^{\textrm{red}}_{R} \cdot \mathbf{L}$. Recall that, using \cite[Thm. 1.5]{zbMATH06323273}, we have an isotypic direct sum decomposition $\restr{\mathbb{V}_{\textrm{abs}}}{\mathcal{N}} = \mathbb{U}_{\textrm{abs}} \bigoplus \mathbb{W}_{\textrm{abs}}$ where the monodromy of $\mathbb{U}_{\textrm{abs}}$ is just the derived subgroup of $\textrm{Res}_{K/\mathbb{Q}} \textrm{GSp}_{2r,K}$ (see \cite[5.3.15]{filipnotes}). Now any element $\ell \in \mathbf{L}(\mathbb{C})$ must commute with the action of $\textrm{Res}_{K/\mathbb{Q}} \textrm{GSp}_{2r,K} \subset \mathbf{N}^{\textrm{red}}_{R}$, so it follows that $\ell$ acts trivially on $\mathbb{U}_{\textrm{abs}}$ and acts entirely through $\Sp(\mathbb{W}_{\textrm{abs}})$. In particular, the group $\mathbf{H}^{\textrm{red}}_{T}$ lies inside $M$.

From the theorem of the fixed part \cite{sch73}, one knows that $\mathbf{H}^{\textrm{red}}_{T}$-stable tensors associated with $\mathbb{V}_{\textrm{abs}}$, Hodge at some point of $T$, are Hodge everywhere on $T$. In particular, the tensors defining the condition $S[\mathcal{N}]$, which are Hodge along $\mathcal{N}$ and stable under $\mathbf{H}^{\textrm{red}}_{T}$, are also Hodge over $T$. The result follows. 
\end{proof}

\subsubsection{Proof of \autoref{mainorbitclosurealg}}

\begin{itemize}
\item[1.] Start by computing $\mathcal{Z}(f,e)$ with $e$ the dimension of a hypothesized atypical orbit closure; that we can do this is a consequence of \autoref{computeZfdlem}. Label the components of $\mathcal{Z}(f,e)$ by $\{ R_{i} \}_{i=1}^{m}$ and their images in $S$ by $T_{i}$.

\item[2.] For each $(R_{i}, T_{i})$, apply \autoref{terribletorsorsearch} to obtain an associated family $g_{i} : C_{i} \to B_{i}$. 

\item[($\star$)] Observe that as a consequence of \autoref{terribletorsorsearch}(iii) and \autoref{nointtypprop2} one knows that, for each atypical orbit closure $\mathcal{N}$ of dimension $e$, there is an index $i$ and point $b \in B_{i}$ such that $\mathcal{N} \subset g^{-1}_{i}(b)$ and $\mathbf{H}_{\mathcal{N}} = \mathbf{H}_{g^{-1}_{i}(b)}$.

\item[3.] Using \autoref{compendalglem} below, compute the endomorphism algebra $\mathcal{E}_{i}$ of the Jacobian over $T_{i}$, which we represent in terms of a set of generators represented by self-correspondences on the curve over $T_{i}$. The algebra of correspondences acts on the de Rham cohomology sheaf $\mathcal{V}_{T_{i}} := \restr{\mathcal{V}_{\textrm{abs}}}{T_{i}}$ through its action on the relative de Rham complex over $T_{i}$, and so we may compute an embedding $\mathcal{E}_{i} \hookrightarrow \sheafend(\mathcal{V}_{T_{i}})(T_{i})$ and identify $\mathcal{E}_{i}$ with its image. Then the central idempotents in $\mathcal{E}_{i}$ compatible give rise to a decomposition into isotypic summands $\bigoplus_{k=1}^{k_{0}} \mathcal{V}_{T_{i}} = \mathcal{V}_{ik}$. We fix a summand $\mathcal{V}_{ik}$ whose endomorphism algebra $\mathcal{E}_{ik}$ is isomorphic to $K$ for $K$ a totally real field of the correct degree, if such a summand exists, and throw away all indices $i$ for which one does not exist. Compute $\mathcal{E}_{ik}$ by computing the finite set of central idempotents and the direct sum decomposition.

\item[4.] For each family $g_{i}$ and index $k = k_{i}$ as above, compute the family $g_{ik} : C_{ik} \to B_{ik}$ where $(X,\omega) \in C_{ik} \subset C_{i}$ if $\omega$ is an eigenform for the action of $\mathcal{E}_{ik}$ on $H^{1}_{\textrm{dR}}(X)$, which is naturally identified with $\mathcal{V}_{\textrm{abs}, (X, \omega)}$. 

\item[($\star$)] We note that after this step, each atypical orbit closure $\mathcal{N}$ of dimension $e$ appears as an irreducible component of some fibre of some $g_{ik}$ as a consequence of \autoref{mainthmalgfilip} and the fact that the real-multiplication summand appearing in \autoref{mainthmalgfilip} is isotypic \cite[Thm. 1.5]{zbMATH06323273}.

\item[5.] For each family $g'_{ik}$, compute the family $g'_{ik} : C'_{ik} \to B_{ij}$ whose non-empty fibres are exactly the $e$-dimensional subvarieties parameterized by $g_{ik}$ which are locally linear for the developing map. This we do by considering the family $f' : \mathcal{Z}' \to \mathcal{Y}'$ which is the pullback under $P \to \mathcal{V}_{s_{0}}$ of the family of all linear subspaces of $\mathcal{V}_{s_{0}}$ of dimension $e = 2r + t$. We then consider $f'_{ik}$ as the family of subvarieties of $P$ which are intersections of fibres of $f'$ with the restriction of $P$ to some fibre of $g'_{ij}$. Then we compute the locus $\mathcal{Z}(f'_{ik},e)$, and define $C'_{ik} \subset C_{ik}$ as the image of $\mathcal{Z}(f'_{ik},e)$.

\item[6.] After varying over all choices of $i$ and $k$, compute the finite set $\Xi = \{ \mathcal{N}_{1}, \hdots, \mathcal{N}_{c} \}$ of varieties obtained as $e$-dimensional components of some $C'_{ik}$ which consist of a single fibre. By \autoref{foundthemall} below, all maximal atypical orbit closures of dimension $e$ are among this set. Use \autoref{decideiforbclosure} below to compute the subset $\Xi' \subset \Xi$ consisting of the orbit closures, and output $\Xi'$.
\end{itemize}

\begin{lem}
\label{compendalglem}
Let $f : C \to U$ be a family of smooth projective curves defined over a number field, with $U$ irreducible. Then there exists an algorithm to compute the endomorphism algebra of the relative Jacobian of $f$, represented in terms of explicit self-correspondences.
\end{lem}

\begin{proof}
Over a number field a detailed algorithm is given in \cite{zbMATH07009723}. Given this, the existence of an algorithm for computing relative Jacobians of families defined over a number field follows from a by-day-by-night procedure. Namely, by day one searches for families of correspondences (themselves defined over a number field) inside $C \times_{U} C$ and computes their image in algebraic de Rham cohomology, and by night one computes endomorphism algebras of Jacobians at increasingly general points in $U(\overline{\mathbb{Q}})$. Eventually the ranks of the algebras computed at the general point and over $U$ match, and at this point we terminate.
\end{proof}

\begin{rem}
We expect a much more efficient generalization of \cite{zbMATH07009723} to be possible in practice, but for now we only concern ourselves in \autoref{compendalglem} with the existence of an algorithm. 
\end{rem}

\begin{lem}
\label{foundthemall}
All maximal atypical orbit closures $\mathcal{N}$ of dimension $e$ are among the set $O$ computed in step 6. 
\end{lem}

\begin{proof}
From our remark labelled $(\star)$ after Step 4 and the fact that orbit closures are locally linear, one knows that all such $\mathcal{N}$ appear as a fibre of some $g'_{ik}$. What needs to be shown is that this fibre is necessarily isolated (an irreducible component of $C'_{ik}$).

The situation to be understood is that we have a family $g : C \to B$ of subvarieties of $\Omega \mathcal{M}_{g}(\kappa)$ with $C$ and $B$ irreducible, all having the same algebraic monodromy group $\mathbf{N}_{R}$, such that at least one fibre is an orbit closure, and such that all fibres are locally linear under $\textrm{Dev}$ of dimension $2r + t$. Moreover we may assume that the Zariski closure $T$ of the image of $C$ in $\Omega \mathcal{M}_{g}(\kappa)$ lies in the locus defined by a fixed real-multiplication and eigenform condition in the sense of \autoref{mainthmalgfilip}. 

We argue that, after fixing a small analytic neighbourhood $\mathcal{T} \subset T$ which intersects the fixed orbit closure $\mathcal{N} \subset T$, the restriction of $\textrm{Dev}$ to $\mathcal{T}$ maps into the fixed linear subspace $F_{\mathcal{N}} \subset H^{1}_{\textrm{rel},s_{0}}$ corresponding to $\mathcal{N}$. Let us recall how this linear subspace is constructed: as discussed in \autoref{explaindimbound}, one has
\begin{itemize}
\item[(i)] $2g - 2r$ linear conditions coming from the requirement that $\omega$ lie in an eigensubspace for the action of a fixed totally real field $K_{\mathcal{N}}$ in the endomorphism algebra of the Jacobian over $T$; and
\item[(ii)] $n-t$ linear conditions coming from the ``twisted torsion'' condition of \autoref{mainthmalgfilip}.
\end{itemize}
Now condition (i) holds on all of $T$ for a fixed $K_{\mathcal{N}}$ by construction, so it follows that $\restr{\textrm{Dev}}{\mathcal{N}}$ lands inside a linear subspace of codimension at least $2g - 2r$. For the linear relations corresponding to (ii) one can construct as in \cite[Rem. 1.5]{zbMATH06575346} flat linear functionals $\varphi_{\ell} : H^{1}_{\textrm{rel},s_{0}} \to \mathbb{C}$ which vanish on the image of $\textrm{Dev}$ restricted to $\mathcal{N} \cap \mathcal{T}$. The functionals $\varphi_{\ell}$ are invariant under the algebraic monodromy group $\mathbf{N}_{R}$ (the monodromy of $\mathcal{N}$ acts on the lattice $\Lambda$ appearing in \cite[Rem. 1.5]{zbMATH06575346} through a finite group).

Now each fibre of $C_{b} = g^{-1}(b)$, by construction, also has algebraic monodromy group $\mathbf{N}_{R}$, so the functionals $\varphi_{\ell}$ are also algebraic monodromy-invariant over $C_{b}$. This implies they are flat, and hence take constant values on $\textrm{Dev}(C_{b} \cap \mathcal{T})$. Now by construction, $C_{b} \cap \mathcal{T}$ maps under $\textrm{Dev}$ into the linear subspace corresponding to condition (i), but now also into an affine subspace satisfying an additional set of $n-t$ linear conditions of the form $\varphi_{\ell} = \kappa$ for a constant $\kappa$. On the other hand $C_{b} \cap \mathcal{T}$ maps into a linear subspace of dimension $2r + t$, so in fact one must have $\kappa = 0$ and $C_{b} \cap \mathcal{T}$ maps into the same subspace which defines $\mathcal{N}$. But $\textrm{Dev}$ is locally injective, so this implies that the family $g$ is trivial.
\end{proof}

\begin{lem}
\label{decideiforbclosure}
There exists an algorithm which decides if a given subvariety $\mathcal{N} \subset S$ is an orbit closure.
\end{lem}

\begin{proof}
It suffices to produce two procedures:
\begin{itemize}
    \item[-] a procedure $\mathbf{Y}$ which, if $\mathcal{N}$ is an orbit closure, eventually outputs ``Yes'' (and does not terminate otherwise); and
    \item[-] a procedure $\mathbf{N}$ which, if $\mathcal{N}$ is not an orbit closure, eventually outputs ``No'' (and does not terminate otherwise).
\end{itemize}
Given these two procedures we can construct a third which runs both $\mathbf{Y}$ and $\mathbf{N}$ in parallel, and outputs either ``Yes'' or ``No'' whenever $\mathbf{Y}$ or $\mathbf{N}$ terminate with an answer.

The procedure $\mathbf{Y}$ works by searching for endomorphisms in the relative Jacobian over $\mathcal{N}$ and an appropriate torsion condition matching \autoref{mainthmalgfilip} in the manner described in the previous lemmas. If $\mathcal{N}$ is an orbit closure, eventually enough such algebro-geometric data is constructed to certify $\mathcal{N}$ is an orbit closure using \autoref{explaindimbound}.

To give the procedure $\mathbf{N}$, we start with the following construction. Choose a point $s \in \mathcal{N}(\Qbar)$. We have a Gauss-Manin system $(\mathcal{V}, \nabla)$ on $S$ computed in \S\ref{comporbitinputdatasec}. If we fix an algebraic basis $\omega^{1}, \hdots, \omega^{m}$ for $\mathcal{V}$ in some algebraic neighbourhood $U$ of $s$ in $S$, and differential forms $ds_{1}, \hdots, ds_{n}$ which trivialize $\Omega^{1}_{U}$, then one obtains from $\nabla$ a system of differential equations for formal power series $f_{ij}$ centered at $s$ with the property that the formal sections $b^{k} = \sum_{ik} f_{ik} \omega^{i}$ of $\mathcal{V}$ at $s$ are flat (see for instance \cite[\S3.2.1]{zbMATH07871772}). For convenience, we may choose $\omega^{1} = \omega$ to be the canonical section of $\mathcal{V}$ corresponding to canonical family over $\Omega \mathcal{M}_{g}(\kappa)$ of holomorphic one-forms. The space of initial conditions for the matrix-valued solutions $[f_{ij}]_{1 \leq i, j \leq m}$ to this differential system are a torsor under $\GL_{m}(\mathbb{C})$, and the radius of convergence of any such solution is independent of the initial condition chosen. From the explicit algebraic functions defining the connection matrices for $\nabla$ relative to the basis $ds_{1}, \hdots, ds_{n}$, one can upper bound the derivatives of the $f_{ij}$, and thereby compute a radius $R$ such that any solution $[f_{ij}]_{1 \leq i, j \leq m}$ as above converges in an analytic ball around $s$ of radius $R$ (in the coordinates determined by $ds_{1}, \hdots, ds_{n}$. Write $B_{R} \subset S(\mathbb{C})$ for this ball.

Our procedure $\mathbf{N}$ will work by selecting points $\{ x_{\ell} \}_{\ell = 1}^{N} \in B_{R} \cap \mathcal{N}(\mathbb{C})$, approximating some solution matrix $[f_{ij}(x_{\ell})]$ coming from periods to some desired accuracy, and checking if the data of the period matrices $\{ [f_{ij}(x_{\ell})] \}_{\ell = 1}^{N}$ is consistent with $\mathcal{N}$ being an orbit closure. For this we use the following criterion, from \cite[Eq. (5.1)]{zbMATH06731862}
\begin{equation}
\label{eq:quoted_text}
\parbox{\dimexpr\linewidth-6em}{%
Any closed subvariety $\mathcal{N} \subset \Omega \mathcal{M}_{g}$ locally defined by real linear equations is $\SL_{2}(\mathbb{R})$-invariant.
}
\end{equation}
It follows that if $\mathcal{N}$ is not an affine invariant manifold, then $\mathcal{N} \cap B_{R}$ is not a real affine subspace in period coordinates. For each $x_{\ell} \in B_{R} \cap \mathcal{N}(\mathbb{C})$ the period coordinates of $x_{\ell}$ can be computed from the matrix $[f_{ij}(x_{\ell})]$ by observing that the first row of the inverse matrix $[f_{ij}(x_{\ell})]^{-1}$ gives the periods of $\omega = \omega^1$. This gives a sequence of points $y_{i} = \left(\int_{b_{1}} \omega_{x_{i}}, \hdots, \int_{b_{m}} \omega_{x_{i}} \right) \in \mathbb{C}^{m}$ for $i = 1, \hdots, N$.

Now suppose that one knows all of the $y_{i}$ to within accuracy $\ep = 1/N$. Then if we set $w = \dim_{\mathbb{C}} \mathcal{N}$ we can solve for a $m \times w$ matrix $\mathbf{A}$ and a vector $\mathbf{c} \in \mathbb{C}^{w}$ such that 
\begin{equation}
\label{linearsystemtosolve}
 \mathbf{A} y_{i} + \mathbf{c} = 0 
\end{equation}
for all $i = 1, \hdots, N$. When a unique solution $(\mathbf{A}, \mathbf{c})$ exists, we can determine $\mathbf{A}$ and $\mathbf{c}$ to within some accuracy $\beta(\ep)$ with $\beta(\ep) \to 0$ as $N \to \infty$. Now supposing that $\mathcal{N}$ is \emph{not} an affine invariant manifold, it follows from \eqref{eq:quoted_text} that for $N$ large enough, one of two things happens:
\begin{itemize}
    \item[(1)] the linear system (\ref{linearsystemtosolve}) is overdetermined, in the sense that there does not exist any $(\mathbf{A}, \mathbf{c})$ and $y'_{i}$ with $\| y'_{i} - y_{i} \| < \ep$ such that $\mathbf{A} y'_{i} + \mathbf{c} = 0$ holds for all $i$; or
    \item[(2)] there exists $(\mathbf{A}, \mathbf{c}, \{ y'_{i} \}^{N}_{i=1})$ with $\|y'_{i} - y_{i}\| < \ep$ such that $\mathbf{A} y'_{i} + \mathbf{c} = 0$, but for all such tuples either $\mathbf{A}$ or $\mathbf{c}$ has a non-trivial imaginary part
\end{itemize}
Both (1) and (2) can be tested by carrying out row-reduction, estimating errors appropriately. Thus, taking iteratively larger values of $N$, and sampling the points $x_{\ell}$ uniformly inside $B_{R} \cap \mathcal{N}(\mathbb{C})$, one eventually computes the matrices $\mathbf{A}$ and $\mathbf{c}$ to sufficient accuracy (or shows they cannot exist) and therefore determines that $\mathcal{N}$ is not real affine-linear in period coordinates. The procedure $\mathbf{N}$ then outputs ``No''.
\end{proof}


\subsection{Effective geometric Zilber-Pink}
\label{sec:effgeoZPsec}

In this section we explain why our proof of the geometric Zilber--Pink conjecture is effective, and allows one to compute atypical weakly special loci. This generalizes \cite{binyamini2021effective}, which takes place in the special case of variations coming from Shimura varieties. Our method is essentially the same, but we do not require any multiplicity estimates.

\begin{thm}
There exists an algorithm which, given the data $(\mathcal{V}, W_{\bullet}, F^{\bullet}, \nabla)$ as input (in the sense of \S\ref{compmodelsec}) outputs the families of \autoref{thm:mixedZP}.
\label{effgeoZP}
\end{thm}

\begin{proof}
By \autoref{computeZfdlem}, the loci $\mathcal{Z}(f,e)$ appearing in the proof of \autoref{nondenseA} can be computed. Thus the same is true for the loci $\mathcal{Y}(e), \mathcal{K}(e)', \mathcal{K}(e)$, etc., appearing in the proof in \S\ref{geozpproofsec}. Given this, \autoref{L3} and \autoref{L10} together imply that the points parameterized by $\mathcal{Y}(e) \setminus \mathcal{K}(e)$ parameterize exactly the germs of the weakly special subvarieties referred to in the statement of \autoref{thm:mixedZP}.

To obtain the actual families corresponding to $\mathcal{Y}(e) \setminus \mathcal{K}(e)$, the most na\"ive approach is to apply the ``search algorithm'' \autoref{infsecfamiliestype} and, for each $g_{i}$, construct the locus $\mathcal{G}(g_{i}) \subset \mathcal{Y}(e) \setminus \mathcal{K}(e)$ consisting of points whose germs agree (upon projecting to $S$) with germs of some fibre of $g_{i}$. (This can be done by replacing the families $f$ appearing in loci like $\mathcal{Y}(f,e)$ with families like $f \cap g_{i}$, cf. the proof of \autoref{protogeoZP2}.) Applying \autoref{lemma432} one will have $\bigcup_{i \in I} \mathcal{G}(g_{i}) = \mathcal{Y}(e) \setminus \mathcal{K}(e)$ for some finite set $I$, and at this point one terminates with the output $\{ g_{i} \}_{i \in I}$. 
\end{proof}

\begin{rem}
Analogously to \autoref{maketerriblesearcheffrem}, one expects to be able to avoid the search procedure in the proof of \autoref{effgeoZP} with more detailed analysis. In the simplest case, when one happens to know that the families $g_{i}$ have only a single fibre (this is true in the next example in \S\ref{sec:balleffective}), one just takes the image of each component of $\mathcal{Y}(e) \setminus \mathcal{K}(e)$ in $S$. 
\end{rem}

\subsection{Example: totally geodesic subvarieties of ball quotients}\label{sec:balleffective}

In this section we prove \autoref{cordm} announced in the introduction. By \autoref{effgeoZP} and \autoref{compconnecprop}, it suffices to construct a geometric variation of pure Hodge structures for which the maximal totally geodesic subvarieties are maximal atypical weakly special subvarieties. Essentially what one is doing is making the strategy appearing in \cite{BU} effective.

First of all we recall that in $\PU(1,2)$, by the work of Deligne, Mostow and Deraux, Parker, Paupert, there are 22 known commensurability classes of non-arithmetic lattices (NA from now on). In $\PU(1,3)$ only 2 (thanks to Deligne, Mostow and Deraux). For details and precise references we refer for example to \cite[\S 2.5]{BU}. Parker \cite{zbMATH07308593} showed that all the currently known non-arithmetic lattices in $\PU(1,2)$ are monodromy groups of higher hypergeometric functions.  Here a totally geodesic subvariety is said to be \emph{maximal} if there are no other totally geodesic subvarieties trivially containing it. For $n>3$ non-arithmetic lattices are currently not known to exist.

For each non-arithmetic lattice $\Gamma$ from \cite[Thm. 12.9  and 12.11]{zbMATH03996010}, there is an algebraic family defined over $\Q$ of algebraic curves $X$ whose monodromy group $\Gamma$ is a subgroup of $\operatorname{Aut}(H^1(X_0,\Z))$ which is not of finite covolume in its Zariski-closure. (A conjecture of Simpson \cite[Page 9]{Simpson} predicts that every NA lattice comes from the monodromy of a `motivic' VHS, possibly defined over a number field; this is not known in general, but it is true, by construction, in the DM cases). We recall below more explicitly such families of curves. See also the viewpoint offered in \cite{zbMATH06149478}.

We start with $n+3$ \emph{weights} $\mu = \{\mu_i\}_{i=0\dots n+2}$. They are rational numbers such that each $\mu_i \in (0,1)$ and they add up to $2$. We set $N$ to be the least
common denominator of the $\mu_i$. We look at the quasi-projective variety over $\Q$

\begin{displaymath}
 Q_n =\{x = (x_0,x_1, \dots ,x_{n+2}) \in {\mathbb{P}}^{n+3}: x_j= x_k \text{  iff  } k=j\}/ Aut(\mathbb{P}^1),
\end{displaymath}
where $Aut(\mathbb{P}^1)$ acts diagonally and freely. 
The Appell-Lauricella hypergeometric functions are solutions $H_\mu$ of a system of linear partial differential equations in $n$ variables $x_i,i=2, \dots, n+1$ with regular singularities along $x_i=0,1,\infty$, $x_i=x_j$ (for $i\neq j$). The system $H_\mu$ has an $n+1$ dimensional solution space. We denote its monodromy group by $\Gamma_\mu$; it is a subgroup of $\PU(1,n)$.

Over $Q_n$, we have a family of smooth projective curves,  with singular affine model given by
\begin{displaymath}
v^N=u^{N\mu_0}(u-1)^{N\mu _1}\prod_{i=2}^{n+1} (u-x_i)^{N \mu_i}.
\end{displaymath}
The period map associated to the $\Z$VHS in question is given by the Prym variety of the above family of covers of $\mathbb{P}^1$. 

\begin{exmp}
The example 15 from \cite[\S 14.3]{zbMATH03996010}, which gives a non-arithmetic lattice $\Gamma$ in $\PU(1,2)$, is given by the family of curves (with parameters $x_2,x_3$)
\begin{displaymath}
v^{12}=u^6(u-1)^4(u-x_2)^5(u-x_3)^4.
\end{displaymath}
It is a cyclic covering of order 12 of $\mathbb{P}^1$.
\end{exmp}

\begin{proof}[Proof of (\ref{cordm}):]
Let $\Gamma$ be a non-arithmetic lattice in $\PU(1,n)$ given as the monodromy of a smooth projective family of varieties (whose corresponding $\Z$VHS is denoted by $\widehat{\V}$). By the results of \cite[\S 5]{BU}, the (complex) totally geodesic subvarieties $S_\Gamma = \Gamma \backslash \mathbb{B}^n$ are the weakly special subvarieties for $(S_\Gamma, \widehat{\V})$ and are in fact atypical intersections (that's how the finiteness of the maximal ones is proved in \emph{op. cit.}). Now using \autoref{compconnecprop} to compute the data $(\mathcal{V}, F^{\bullet}, \nabla)$ associated to the smooth projective family, the result follows from \autoref{effgeoZP}.
\end{proof}

\bibliography{hodge_theory}
\bibliographystyle{abbrv}

\Addresses

\end{document}